\documentclass{amsart}
\usepackage{xcolor}
\usepackage{tikz-cd}

\usepackage[english]{babel}
\usepackage{alphabeta}  %Need for greek
\usepackage{amssymb}
\usepackage{amsmath}
\usepackage{bm} % for bold
\usepackage[all]{xy} % For making diagrams
\usepackage{mathrsfs} % Like mathcal, \mathscr

\usepackage{a4wide}

\usepackage{xcolor}

\usepackage{float}	 %Improves the interface for figures & tables
\usepackage{graphicx}
\usepackage{chngcntr} %Change the resetting of counters - I used it for figures numbering
\usepackage{caption} % provides many ways to customise the captions in floating environments like figure and table,

\usepackage{enumerate}

\usepackage[colorlinks,linkcolor={blue},citecolor={blue},urlcolor={purple},]{hyperref}

\usepackage[colorinlistoftodos,prependcaption,textsize=tiny]{todonotes}

\usepackage{easyReview}

\usepackage{mathtools}
\usepackage{physics} %Need this one

\usepackage{amsthm}
\usepackage{cleveref}   %must be loaded LAST

\renewcommand{\epsilon}{\varepsilon}
\renewcommand{\exp}[1]{e^{#1}}

\newcommand{\inv}[1]{#1^{-1}}

\renewcommand{\1}{{\bf 1}}
\newcommand\restr[2]{{% we make the whole thing an ordinary symbol
		\left.\kern-\nulldelimiterspace % automatically resize the bar with \right
		#1 % the function
		\vphantom{\big|} % pretend it's a little taller at normal size
		\right|_{#2} % this is the delimiter
}}

\newcommand{\C}{\ensuremath{\mathbb C}}

\newcommand{\R}{\mathbb{R}}

\newcommand{\N}{\mathbb{N}}

\newcommand{\F}{\mathscr F}
\newcommand{\E}{\mathbb E}

\renewcommand{\P}{\mathbb P}
\renewcommand{\O}{\mathcal O}
\renewcommand{\L}{\mathcal L}

\newcommand{\dWHs}{\,\mathrm{d}W(s)}

\newcommand{\la}{\lambda}
\newcommand{\eps}{\varepsilon}
\renewcommand{\tilde}{\widetilde}

\DeclareSymbolFont{matha}{OML}{txmi}{m}{it}% txfonts
\DeclareMathSymbol{\varv}{\mathord}{matha}{118}
\theoremstyle{plain}

\newtheorem{theorem}{Theorem}[section]
\newtheorem{lemma}[theorem]{Lemma}
\newtheorem{proposition}[theorem]{Proposition}
\newtheorem{corollary}[theorem]{Corollary}

\theoremstyle{definition}
\newtheorem{defn}[theorem]{Definition} % definition numbers are dependent on theorem numbers
\newtheorem{definition}[theorem]{Definition} % definition numbers are dependent on theorem numbers

\newtheorem{assumption}[theorem]{Assumption}

\theoremstyle{remark}
\theoremstyle{remark}
\newtheorem{remark}[theorem]{Remark}
\numberwithin{equation}{section}

\newcommand{\calL}{\mathcal{L}}
\newcommand{\coloneq}{\coloneqq}

\newcommand{\wt}{\widetilde}
\newcommand{\SMR}{{\rm SMR}}
\newcommand{\DSMR}{{\rm DSMR}}

	\allowdisplaybreaks

\title{Discrete stochastic maximal regularity}

\begin{document}

\author{Foivos Evangelopoulos-Ntemiris}
\address{Delft Institute of Applied Mathematics\\
Delft University of Technology \\ P.O. Box 5031\\ 2600 GA Delft\\The
Netherlands} \email{
F.A.Evangelopoulos-Ntemiris@tudelft.nl and 
foivosevangelopoulos@gmail.com}

\author{Mark Veraar}
\address{Delft Institute of Applied Mathematics\\
Delft University of Technology \\ P.O. Box 5031\\ 2600 GA Delft\\The
Netherlands} \email{M.C.Veraar@tudelft.nl}

\thanks{The authors have received funding from the VICI subsidy VI.C.212.027 of the Netherlands Organisation for Scientific Research (NWO)}

\begin{abstract}
In this paper, we investigate discrete regularity estimates for a broad class of temporal numerical schemes for parabolic stochastic evolution equations. We provide a characterization of discrete stochastic maximal $\ell^p$-regularity in terms of its continuous counterpart, thereby establishing a unified framework that yields numerous new discrete regularity results. Moreover, as a consequence of the continuous-time theory, we establish several important properties of discrete stochastic maximal regularity such as extrapolation in the exponent $p$ and with respect to a power weight. Furthermore, employing the $H^\infty$-functional calculus, we derive a powerful discrete maximal estimate in the trace space norm $D_A(1-\frac1p,p)$ for $p \in [2,\infty)$.
\end{abstract}

\keywords{Discrete stochastic maximal regularity, maximal estimates, stochastic evolution equations, time discretization schemes, rational approximation, $H^\infty$-calculus, $\mathcal{R}$-boundedness}

\subjclass[2010]{Primary: 46N40, 60H15, Secondary: 35B65, 42B37, 47D06, 60H35, 65J10, 65M12}

\maketitle
\setcounter{tocdepth}{1}
	\tableofcontents

\section{Introduction}

Maximal $L^p$-regularity techniques play a central role in the theory of both deterministic and stochastic evolution equations of parabolic type (see the monographs \cite{Analysis3, PrussSim} and the surveys \cite{agresti2025nonlinear, wilke2023linear}, as well as the references therein). Although maximal regularity is inherently a linear concept, it becomes a powerful tool for analyzing nonlinear problems through linearization techniques. In particular, it allows us to establish local well-posedness and regularity results, and formulate sharp blow-up criteria for global existence for a wide class of equations.

In the deterministic setting, a theory of discrete maximal $\ell^p$-regularity can already be found in \cite{AsSo} and was connected to discrete Fourier multiplier theory in  \cite{BlunckJFA, BlunckStudia}. The connections between discrete maximal regularity and numerical analysis—specifically, the stability and convergence of numerical schemes for (non)linear PDEs—have been explored in numerous works (see, for example, \cite{akrivis2022error, akrivis2017combining, akrivis2015fully, APW02, kemmochi_discrete_2016, kemmochi2018discrete, KovaczSIAM} and references therein), and this line of research remains highly active.

Motivated by these developments, we introduce a discrete analogue of stochastic maximal regularity. Our framework encompasses the exponential Euler method and a broad class of rational time discretization schemes, including the implicit Euler method. In addition to establishing new maximal regularity estimates, we derive discrete analogues of maximal estimates with optimal trace regularity. These results are expected to pave the way for improved convergence rates, as already observed in the deterministic setting. For instance, in the recent work \cite{Li-Zhou-Pathwise-uniform}, the authors used stochastic maximal regularity and non-optimal maximal estimates to prove pathwise uniform convergence estimates of a full discretization for the three-dimensional Allen-Cahn equation. The reader is also referred to \cite[Remark 3.3]{li-Xie-stability} for an elementary example demonstrating the use of discrete stochastic maximal regularity in proving stability estimates.

The seminal results of \cite{DoreVenni} and \cite{weis2001operator} provide necessary and sufficient conditions for maximal $L^p$-regularity in the deterministic continuous-time context. In the stochastic setting, sufficient conditions involving the $H^\infty$-calculus of an operator $A$ were established in \cite{NVWSMR, vanneerven2014stochasticintegrationbanachspaces}. For a more detailed historical account on (stochastic) maximal regularity, we refer the reader to the recent survey \cite{agresti2025nonlinear} and the monograph \cite{DPZ}.

Discrete versions of stochastic maximal regularity estimates are mainly known in a Hilbert space setting (see \cite{GyMil05, GyMil07, GyMil09, Kazashi}). An exception is \cite{li-Xie-stability}, where a discrete stochastic maximal regularity is proved for the implicit Euler scheme in $L^q$-spaces using the boundedness of the $H^\infty$-calculus. 

In our work, we establish a full theory of discrete stochastic maximal regularity, which in particular entails the following:
\begin{enumerate}[(a)]
    \item\label{it:new1} An equivalence of the continuous and discrete setting for a large class of numerical schemes;
    \item\label{it:new2} Discrete stochastic maximal regularity results in Hilbert spaces, spaces like $L^q$, and real interpolation spaces;
    \item\label{it:new3} A discrete  maximal estimate in the trace space $D_A(1-\frac1p,p)$;
    \item\label{it:new4} All results are presented in a time-weighted setting.
\end{enumerate}
Part \eqref{it:new2} is a direct consequence of our equivalence theorem in \eqref{it:new1} and the results of \cite{agresti2025nonlinear, LoVer, NVWSMR, vanneerven2014stochasticintegrationbanachspaces} by the second author and his collaborators. In particular, \eqref{it:new2} contains a far-reaching extension of the main result of \cite{li-Xie-stability}. 
The discrete maximal estimate mentioned in \eqref{it:new3} is one of the deepest results of the paper, and has never been considered in a stochastic setting before. It incorporates maximal estimates and parabolic smoothing. As is well-known in the theory of stochastic processes, it can be rather delicate to prove maximal estimates (see \cite{Talagrand}). Our proof relies on $H^\infty$-calculus, $\mathcal{R}$-boundedness techniques, fractional calculus, and deep results from interpolation theory.

\subsection{Formulation of our main results}
Below, we will formulate some of the main results of the paper. For simplicity, we only present the unweighted setting in this subsection.  

Let $X_0$ be (isomorphic to a closed subspace of) $L^q(\mathcal{O})$, for some $q\in [2, \infty)$, over a $\sigma$-finite measure space $(\mathcal{O}, \Sigma, \mu)$.
Suppose that $X_1$ is a Banach space such that  $X_1\hookrightarrow X_0$ densely and define the real and complex interpolation spaces
\[ X_{\alpha,p} \coloneq (X_0, X_1)_{\alpha,p} \ \ \text{and} \ \ X_{\alpha} \coloneq [X_{0}, X_1]_{\alpha} , \ \ \ \alpha\in (0,1), \ p\in (1,\infty).\]
Throughout this manuscript, we impose the following assumption:
\begin{assumption}\label{assum:mainDSMR} \
\begin{enumerate}[(1)]
\item\label{it:mainDSMR1} $A$ is a sectorial operator on $X_0$ of angle $\omega(A)<\frac \pi2$, and $D(A) = X_1$;
\item\label{it:mainDSMR2} $R\colon [0,\infty)\to \calL(X_0)$ is given by $R_{\tau} \coloneq r(\tau A)$, where $r$ is either the exponential function $r(z) \coloneq e^{-z}$, or $r\colon \Sigma_{\theta}\to \C$ is an $A(\theta)$-stable rational function (i.e.\ $|r(z)|\leq 1$ for $z\in \Sigma_{\theta}$) with  $\theta \in (\omega(A),\tfrac \pi2]$  that is consistent of order $\ell \ge 1$ (i.e.\  $|r(z) - \exp{-z}|\leq Cz^{\ell+1}$ for $z\to 0$) and satisfies $r(\infty) = 0$.
\end{enumerate}
\end{assumption}

Here, $ \Sigma_{\theta}\coloneq \{z\in \C\setminus\{0\}\colon |\arg z|<\theta\}$.
Given a time step $\tau>0$, let $t_n = n\tau$, $n \ge 0$, and consider the following discretization scheme: starting from $Y_0 \coloneq 0$, define recursively
\begin{equation} \label{Eq:Definition of approximation scheme-intro}
		Y_{n+1} \coloneq R_\tau Y_{n} + R_\tau \int_{t_n}^{t_{n+1}} g(s)  d W(s), \quad n\geq 0,
	\end{equation}	
where $W$ is a cylindrical Brownian motion in $\ell^2$ and $g\in L^p(\Omega;L^p (0,T;\gamma(\ell^2,X_{1/2})))$ is strongly progressively measurable. 

The scheme $R$ is said to have {\em discrete stochastic maximal $\ell^p$-regularity} if there exists a constant $C$, independent of $\tau$ and $g$, such that
\begin{equation}\label{DSMR definition introduction}
    \E \sum_{n\geq 0} \tau  \| A Y_n \|^p_{X_0}
\le C^p \E\|g\|_{L^p (0,T;\gamma(\ell^2,X_{1/2}))}^p.
\end{equation}
We emphasize that the formulation \eqref{Eq:Definition of approximation scheme-intro} is sufficiently general to encompass various approximation methods. In particular, one can take $g(t) = g_n$ for $t\in (t_n, t_{n+1}]$, where $g_n\in L^p(\Omega;\gamma(\ell^2,X_{1/2}))$ is $\F_{t_n}$-measurable. This leads to the more common scheme
\begin{equation}\label{Eq:Definition of approximation scheme-intro special case}
    Y_{n+1} \coloneq R_\tau Y_{n} + R_\tau g_n (W(t_{n+1}) - W(t_n)), \quad n\geq 0.
\end{equation}
In this case \eqref{DSMR definition introduction} becomes
$$
\E \sum_{n\geq 0} \tau  \| A Y_n \|^p_{X_0}
\le C^p \E \sum_{n\ge0} \tau \|g_n\|^p_{\gamma(\ell^2,X_{1/2}))}.
$$
    
In most cases, $X_{1/2}$ is a (subspace of a) (fractional) Sobolev space $H^{s,q}(D)$, and in this situation 
\[\gamma(\ell^2,X_{1/2}) = H^{s,q}(D;\ell^2).\]

Our first main result is that this property can be characterized in the following way.
\begin{theorem}[Characterization of Discrete Stochastic Maximal $\ell^p$-regularity]\label{thm:introequiv}
Let $q\in [2, \infty)$ and suppose that $X_0$ is isomorphic to a closed subspace of $L^q(\mathcal{O})$ with $(\mathcal{O}, \Sigma,\mu)$ a $\sigma$-finite measure space.
Suppose that Assumption \ref{assum:mainDSMR} holds.
Then for any $p\in [2, \infty)$ the following are equivalent:
\begin{enumerate}[(1)]
\item\label{it1:introequiv} $A$ has stochastic maximal $L^p$-regularity;
\item\label{it2:introequiv} $R$ has discrete stochastic maximal $\ell^p$-regularity.
\end{enumerate}
\end{theorem}
As already mentioned, the continuous analogue in \eqref{it1:introequiv} was introduced in \cite{AVstab,NVWSMR, vanneerven2014stochasticintegrationbanachspaces}. We will actually use the definition of \cite{AV19_QSEE_1,agresti2025nonlinear} where the space regularity is shifted by $1/2$ as this is more natural in applications to SPDEs (see Remark \ref{rem:BIPzero} for further information). 

Theorem \ref{thm:introequiv} remains valid for any space $X_0$ satisfying certain geometric properties, specifically it holds if $X_0$ is UMD and has type $2$. The strength of this result lies in the fact that it establishes discrete stochastic maximal regularity for a broad class of numerical schemes simultaneously. Moreover, it shows that it suffices to analyze the continuous-time setting, for which a well-developed theoretical framework already exists. Some of the immediate consequences for the discrete counterpart will be presented in Section \ref{sec:perm}.

The proof proceeds in several steps, detailed in Subsections \ref{SMR-EE-DSMR}, \ref{subs:equivDSMR}, and \ref{subs:DSMR-SMR}. Our approach not only establishes the equivalence in the stochastic setting but also suggests how similar equivalences could be obtained for other schemes in the deterministic setting. Indeed, in the deterministic case, such an equivalence was previously shown for the exponential Euler scheme in \cite{KaltonPortal} and for the implicit Euler scheme in \cite{kemmochi_discrete_2016}. Our techniques can also be used to prove the equivalence for other schemes in the deterministic case.

An important consequence of Theorem \ref{thm:introequiv} and a recent result in \cite{agresti2025nonlinear} is the following new result in the Hilbert space setting:

\begin{corollary}[Discrete Stochastic Maximal $\ell^p$-regularity for free]
Suppose that $X_0$ is a Hilbert space and that Assumption \ref{assum:mainDSMR} holds. Then $R$ has discrete stochastic maximal $\ell^p$-regularity for any $p\in [2, \infty)$.
\end{corollary}
In previous attempts, more structure was assumed on the operator $A$ (e.g.\ variational structure or $H^\infty$-calculus). We show that this is not needed. The corresponding deterministic analogue also holds and has been established for various numerical schemes (see, for example, \cite[Proposition 2.7]{BlunckStudia}).

In addition to $\ell^p$-estimates, it is often important to obtain parabolic $\ell^\infty$-bounds for the sequence $(Y_n)_{n\geq 1}$. In the continuous-time setting, such {\em maximal estimates} are available (see \cite{NVWSMR}), where they are derived from optimal mixed space-time regularity results, which in turn rely on the $H^\infty$-calculus of the operator $A$. In the discrete case, we establish the following natural analogue of these maximal estimates.
\begin{theorem}[Maximal estimate with parabolic regularization]\label{thm:maxest} 
Let $q\in [2, \infty)$ and suppose that $X_0$ is isomorphic to a closed subspace of $L^q(\mathcal{O})$ with $(\mathcal{O}, \Sigma,\mu)$ a $\sigma$-finite measure space.
Suppose that Assumption \ref{assum:mainDSMR} holds, that $A$ has a bounded $H^\infty$-calculus on $X_0$ of angle $<\pi/2$, and that $0\in \rho(A)$.
Then for any $p\in (2, \infty)$, there is a constant $C$ such that for every $g\in L^{p}_{\mathbb F}(\Omega;L^p(\R_+;\gamma(\ell^2,X_{1/2})))$ and every stepsize $\tau>0$,
\begin{equation*}
	\E \sup_{n\ge1} \| Y_n \|_{X_{1-\frac 1p,p}}^p \le C^p \E\|g\|^{p}_{L^p(\R_+;\gamma(\ell^2,X_{1/2}))},
\end{equation*}	
where $Y=(Y_n)_{n\ge0}$ is given by \eqref{Eq:Definition of approximation scheme-intro}.
\end{theorem}
It is important to note that $X_{1-\frac 1p,p}$ serves as the natural trace space associated with parabolic SPDEs. It is, in fact, the optimal space in which one can expect to obtain maximal estimates.

In the continuous-time setting, such estimates were proved using an optimal variant of the Da Prato–Kwapień–Zabczyk factorization method, combined with the operator-valued $H^\infty$-calculus, $\mathcal{R}$-boundedness of stochastic convolutions, and the analytical properties of the operator sum $ d/dt + A$.
However, this approach does not carry over to the discrete setting, and a new argument is required. The key difference is that no natural identities or estimates for fractional derivatives in time are available in the discrete setting. To solve this, we embed the problem into a continuous-time setting by interpolating linearly and we estimate the fractional time derivatives by hand. To do this, many terms arise, all of which need a separate argument. Some of the leading-order terms are estimated through the operator-valued $H^\infty$-calculus and new $\mathcal{R}$-boundedness results for stochastic convolutions.

A natural question is whether Theorem \ref{thm:maxest} also holds in the case
$p=2$. A key difference is that in this case there is no regularization. Indeed, if $X_0$ is a Hilbert space then  $X_{1/2,2} = X_{1/2}$ (see \cite[Corollary C.4.2]{Analysis1}), so that the regularity space for the solution $(Y_n)_{n\ge0}$ and the data $g$ coincide. Still, the maximal estimate turns out to be true at least when $X_0$ is a Hilbert space. Remarkably, it is enough to assume that either $0\in \rho(A)$ {\em or} that the $H^\infty$-calculus is bounded.
\begin{proposition}\label{prop:maxHilbert}
Let $X_0$ be a Hilbert space. Suppose that Assumption \ref{assum:mainDSMR} holds with $\theta=\pi/2$. Suppose that $0\in \rho(A)$ or that $A$ has a bounded $H^\infty$-calculus.
Then there is a constant $C$ such that for every $g\in L^{2}_{\mathbb F}(\Omega;L^2(\R_+;\gamma(\ell^2,X_{1/2})))$ and every stepsize $\tau>0$,
\begin{equation}
	\E \sup_{n\ge1} \| Y_n \|_{X_{1/2}}^2 \le C^2 \E\|g\|^{2}_{L^2(\R_+;\gamma(\ell^2,X_{1/2}))}.
\end{equation}	
\end{proposition}
The assumption $r(\infty)=0$ is not used in the proof of Proposition \ref{prop:maxHilbert}.

\subsection*{Overview}

In Section \ref{sec:Prel} we present some essential mathematical background, covering sectorial operators, functional calculus, and key properties of numerical approximation schemes $R_\tau = r(\tau A)$, including crucial error estimates. We also review stochastic integration theory and the Burkholder-Davis-Gundy inequalities.

Section \ref{sec:DSMR} formally defines discrete stochastic maximal $\ell^p$-regularity, denoted by $\DSMR(p,T)$. Our central achievement here is Theorem \ref{thm:mainequiv}, establishing the equivalence between $\DSMR$ for a scheme $R$ and the continuous $\SMR$ of the operator $A$, which directly yields new $\DSMR$ results (e.g., Corollary \ref{cor:DSMRLq}). This section contains Theorem \ref{thm:introequiv} as a special case. 

In Section \ref{sec:perm}, we demonstrate that $\DSMR$ inherits key permanence properties (such as $p$- and $T$-independence) from $\SMR$ via Theorem \ref{thm:mainequiv}. Notably, we also establish the equivalence of (discrete) stochastic maximal $\ell^p$-regularity to a time-weighted formulation in Theorem \ref{thm:weighted}.

Section \ref{sec:Rbdd} provides a critical technical tool for subsequent analysis. We establish the uniform $\mathcal{R}$-boundedness of relevant discrete stochastic convolution operator families, which is essential for proving the maximal estimates in Section \ref{sec:discretemax}.

The main result of Section \ref{sec:discretemax} is Theorem \ref{thm:maximalest} and contains the maximal estimate of Theorem \ref{thm:maxest} as a special case. This result is obtained by employing the $H^\infty$-functional calculus for the operator $A$ and the $\mathcal{R}$-boundedness results from Section \ref{sec:Rbdd}. A Hilbert space version for $p=2$ is also presented, and it particularly contains Proposition \ref{prop:maxHilbert} as a special case.

\subsection*{Notation}

Let $w_{\alpha}(t) = t^{\alpha}$. For $p\in [1, \infty)$, a stepsize $\tau>0$, and a Banach space $Z$, let $\ell^p_{\tau,w_{\alpha}}(Z)$ denote the space of sequences $z\coloneq (z_{j})_{j\geq 0}$ in $Z$ such that
\[\|z\|_{\ell^p_{\tau,w_{\alpha}}(Z)} \coloneq \Big(\sum_{n\geq 0} w_{\alpha}(t_{n+1}) \tau\|z_n\|^p_Z\Big)^{1/p}<\infty,\]
where $t_n = \tau n$ for $n\geq 1$. If a sequence $z = (z_n)_{n=0}^N$ consists of finitely many terms, we will use the same notation with the convention that $z_n = 0$ for $n\geq N+1$. 
For $T\in (0,\infty]$ we write $L^p (0,T,w_{\alpha};Z)$ for the space of measurable functions $f\colon (0,T)\to Z$ such that
\[\|f\|_{L^p (0,T,w_{\alpha};Z)} \coloneq \Big(\int_{0}^T \|f(t)\|_Z^p w_{\alpha}(t) dt\Big)^{1/p}<\infty.\]
Whenever $\alpha=0$, we will omit the weight from the above notations. 

For $a,b\in \R$ we will use the standard notation $a\lesssim b$ in case there is an (unimportant) constant $C$ such that $a\leq Cb$. Moreover, we write $a\simeq b$ if $a\lesssim b$ and $b \lesssim a$. In case we want to emphasize that $C$ depends on e.g.\ $p$, we write $a\lesssim_p b$ or $a \simeq_p b$. 

For a Banach couple $(X_0, X_1)$, we use 
\[ X_{\alpha,p} \coloneq (X_0, X_1)_{\alpha,p} \ \ \text{and} \ \ X_{\alpha} \coloneq [X_{0}, X_1]_{\alpha} , \ \ \ \alpha\in (0,1), \ p\in (1,\infty),\]
for the real and complex interpolation spaces respectively.

The operator $A$ will always be a sectorial operator on $X_0$ with domain $D(A) = X_1$. For $\theta \in (0,\pi)$, let $\Sigma_{\theta} \coloneq \{z\in \C\setminus\{0\}\colon  |\arg(z)|<\theta\}$ and let $\partial \Sigma_{\theta}$ denote its boundary oriented counterclockwise. The function $r\colon \Sigma_{\theta}\to \C$ denotes a rational function or $r(z)=e^{-z}$, and $R_{\tau} \coloneq r(\tau A)$.

The notation $W$ is used to denote a cylindrical Brownian motion, and $I_g(t) \coloneq  \int_0^t g d W$ for the stochastic integral of $g$. We also write $\Delta_n I_g \coloneq I_g(t_{n+1}) - I_g(t_{n})$.

\subsubsection*{Acknowledgements}
{The authors thank Katharina Klioba and Emiel Lorist for helpful comments. The authors also express their gratitude to the anonymous referees for their insightful suggestions.}

\section{Preliminaries}\label{sec:Prel}

\subsection{Sectorial operators, functional calculus, and interpolation}

For details on sectorial operators, semigroup theory, and functional calculus, the reader is referred to \cite{zbMATH01354832,Haase:2, hytonen2017analysis, Analysis3, kunstmann2004maximal}. We briefly recall some of the key concepts used throughout the paper. For details on interpolation theory the reader is referred to \cite{Analysis1, Tri95} and we will rely on standard results on real and complex interpolation. See the notation subsection at the end of the introduction.

Let $(A,D(A))$ be a closed operator on a Banach space $X$. The operator $A$ is called {\em sectorial} if the domain and the range of $A$ are dense in $X$ and there exists $\nu\in (0,\pi)$ such that $\sigma(A)\subseteq \overline{\Sigma_{\nu}}$, where $ \Sigma_{\nu}\coloneq \{z\in \C\setminus\{0\}\colon |\arg z|<\nu\}$, and there exists $C>0$ such that
\begin{equation}
\label{eq:sectorialA}
|\lambda|\|(\lambda-A)^{-1}\|_{\calL(X)}\leq C, \ \ \lambda \in \C\setminus \overline{\Sigma_{\nu}}.
\end{equation}
The {\em angle of sectoriality} $\omega(A)\in[0,\pi)$ is defined as the infimum over all $\nu$ for which a $C$ exists such that \eqref{eq:sectorialA} holds. 

If $\omega(A)<\pi/2$, then $-A$ generates a strongly continuous semigroup $(e^{-tA})_{t\geq 0}$, which extends to a bounded analytic function on a sector.

\subsubsection{Functional calculus}
Let $H^1(\Sigma_{\theta})$ be the set of all holomorphic functions $f\colon \Sigma_{\theta}\to \C$ such that 
\[\|f\|_{H^1(\Sigma_{\theta})} \coloneq \sup_{|\phi|<\theta} \int_{0}^\infty |f(s e^{i\phi})| \frac{ds}{s}<\infty.\]
Moreover, $H^\infty(\Sigma_{\theta})$ is the space of bounded holomorphic functions $f\colon\Sigma_{\theta}\to \C$ equipped with the supremum norm.

If $A$ is a sectorial operator of angle $\omega(A)\in [0,\pi)$ and $f\in H^1(\Sigma_{\theta})$ with $\theta\in (\omega(A),\pi)$ then one can define the bounded operator $f(A)$ by the Dunford integral (contour oriented counterclockwise)
 $$f(A) \coloneq  \frac{1}{2\pi i} \int_{\partial \Sigma_{\nu}} f(z) (z-A)^{-1} \,dz,$$
 where $\nu \in (\omega(A),\theta)$ is chosen arbitrarily  (see \cite[Section 10.2]{hytonen2017analysis}). Moreover, there is a constant $C=C(\theta,A)$ such that 
 \begin{equation} \label{ineq:H1 calculus bound}
 	\sup_{t>0}	\|f(tA)\|_{\calL(X)} \le C \|f\|_{H^1(\Sigma_\theta)}.
 \end{equation}

The above $H^1$-calculus is useful and can be used for any sectorial operator. In many important cases one can even prove that for all $\nu \in (\omega(A),\theta)$ and for every $f\in H^1(\Sigma_{\nu}) \cap H^\infty(\Sigma_{\nu})$ one has
\begin{align}\label{eq:Hinfty}
 \|f(A)\|_{\calL(X)}\leq C\|f\|_{H^\infty(\Sigma_{\nu})}.
\end{align}
In this case we say that $A$ has a {\em bounded $H^\infty$-calculus}. One can show that \eqref{eq:Hinfty} uniquely extends to all $f\in H^\infty(\Sigma_{\nu})$. The infimum over all possible $\nu$ is called {\em the angle of the $H^\infty$-calculus}. 

In the paper, we will use the $H^\infty$-calculus as a black box to prove several new results. By now, large classes of sectorial operators $A$ are known to have a bounded $H^\infty$-calculus. One could even say that on $L^q$-spaces the counterexamples are typically only rather academic. A comprehensive list of examples can be found in the notes of \cite[Chapter 10]{hytonen2017analysis}.

A general class of operators with a bounded $H^\infty$-calculus, in the case $X$ is a Hilbert space, is given by the following: all operators $A$ for which 
$-A$ generates a contraction semigroup on $X$. In this case, $\omega(A)$ coincides with the angle of the $H^\infty$-calculus. For details the reader is referred to \cite[Theorems 10.2.24 and 10.4.21]{hytonen2017analysis}.

Another class can be given on $X = L^q$ for $q\in (1, \infty)$: all operators $A$ for which $-A$ generates a positive contraction semigroup on $X$. In this case, one obtains that the angle of the $H^\infty$-calculus is $\leq \pi/2$. Moreover, if $A$ is also sectorial of angle $\omega(A)<\pi/2$, then one obtains that the angle of the $H^\infty$-calculus is $<\pi/2$ as well (although the two angles might differ). Details can be found in \cite[Theorems 10.7.12 and 10.7.13]{hytonen2017analysis}.

\subsubsection{Standard estimates for semigroups}
Let $(A,D(A))$ be a sectorial operator of angle $<\pi/2$ on $X$. It is well-known that $-A$ generates a strongly continuous bounded analytic semigroup $(e^{-zA})_{z\in \Sigma_{\sigma}}$ for some $\sigma>0$ (see \cite[Appendix G]{hytonen2017analysis}). Moreover, letting $\lambda>0$, we can apply \eqref{ineq:H1 calculus bound} to $f(z) = z^{\lambda} e^{-z}$, which is in $H^1(\Sigma_{\nu})$ for any $\nu<\pi/2$, to obtain 
\begin{align*}
  \sup_{t>0}\|(t A)^{\lambda} e^{-tA}\|_{\calL(X)}<\infty.
\end{align*}
This implies that, for any $\epsilon>0$, $\sup_{t>0}\|t^{\varepsilon} A^{1+\varepsilon} e^{-tA}\|_{\calL(D(A),X)}<\infty$. Interpolating the two bounds gives that there is an $M\geq 0$ such that
\begin{equation} \label{Eq:Shifted decay 1}
	\|A^{1+\varepsilon} e^{-tA}\|_{\calL(V,X)} \leq \frac{M}{t^{\frac12+\eps}}, \ t>0,
\end{equation}
where $V$ is any Banach space that is continuously embedded into   $(X, D(A))_{1/2,\infty}$. For example, one can take $V=D(A^{1/2})$, $V = [X, D(A)]_{1/2}$ or $V = (X, D(A))_{1/2,q}$ with $q\in [1, \infty]$.

Another standard bound for a bounded strongly continuous semigroup is 
\begin{equation} \label{Eq: Decay rate of semigroup difference}
	\|\exp{-tA} - \exp{-sA}\|_{\mathcal{L}(V, X)} \leq M (t - s)^{\alpha},
\end{equation}
where $\alpha \in [0,1]$ and $V$ is as above. Indeed, for $V= D(A)$ this holds for $\alpha=1$ and follows by writing $  \exp{-sA} - \exp{-tA}=\int_s^t Ae^{-rA} dr$ on $D(A)$, and using that $\sup_{r>0}\|e^{-rA}\|_{\calL(X)}<\infty$. For $V = X$ the bound \eqref{Eq: Decay rate of semigroup difference} holds with $\alpha=0$. The general case follows by interpolation. 

\subsubsection{Fractional powers}
For a sectorial operator the fractional power $A^{z}$ can be defined for any $z\in \C$ via the so-called extended functional calculus (see \cite[Chapter 3]{Haase:2} and \cite[Chapter 15]{Analysis3}). In general, these operators are again closed unbounded operators. 

The operator $A$ is said to have {\em bounded imaginary powers (BIP)} if $A^{it}$ extends to a bounded operator on $X$ for any $t\in \R$. In particular, this holds if $A$ has a bounded $H^\infty$-calculus. Indeed, one can apply the calculus to the holomorphic and bounded function $z\mapsto z^{it}$. An important consequence of BIP is the following identification of the domains of the fractional powers and the complex interpolation spaces.
\begin{lemma}\label{lem:BIP}
Suppose that $A$ is a sectorial operator and that $A$ has BIP. Then for all $\theta\in (0,1)$, \[D(A^{\theta}) = [X, D(A)]_{\theta}\]
with equivalent norms.
\end{lemma}
A proof can be found in \cite[Theorem 6.6.9]{Haase:2}, \cite[Theorem 15.3.9]{Analysis3}, \cite[1.15.3]{Tri95}.

\subsection{Approximation schemes}
	
		Let $X$ be a Banach space, $A$ a sectorial operator on $X$ with angle $\omega(A)<\frac \pi2$ and let $(\exp{-t A})_{t \ge 0}$ be the bounded analytic strongly continuous semigroup generated by $-A$. An approximation scheme is a strongly continuous function $R\colon [0,\infty)\to\L(X)$ with $R(0)=I$ that approximates the semigroup in the sense that there are $\alpha >0$ and a Banach space $V$ with $V\hookrightarrow X$ densely such that 
    \begin{equation} \label{eq:DefApproxScheme}
      \|R_ \tau ^n - \exp{-n\tau A} \|_{\L(V,X)} \leq C \tau^a, \qquad n\ge 0,\tau>0,
    \end{equation}
   where we set $R_\tau\coloneq R(\tau)$ and $C>0$ is a constant independent of $n,\tau$. A common choice for $V$ is the fractional domain $V=D(A^\alpha)$.  An approximation is called stable if there is a constant $C>0$ such that 
\begin{equation}\label{eq:stab}
  \|R_\tau ^n \|_{\L(X)} \le C, \quad n\ge 0, \tau >0.
\end{equation}
Typical examples of stable approximation schemes are the exponential Euler method $R_\tau = e^{-\tau A}$, the implicit Euler method $R_\tau\coloneq(1+\tau A)^{-1}$, the Crank-Nicolson method $R_\tau \coloneq(1-\frac 12 \tau A)(1+\frac 12 \tau A)^{-1}$ and, in general, any rational scheme $R_\tau \coloneq r(\tau A)$ where $r\colon \Sigma_\theta \to \mathbb C$ is a rational function that is consistent and $A(\theta)$-stable with $\theta \in (\omega(A),\frac \pi 2]$ (see \cite[Theorems 9.1 and 9.2]{thomee1984galerkin}).    Recall that a rational function $r$ is called 
	{\em consistent of order $\ell\ge1$} if $|r(z)-\exp{-z}| \le C|z|^{\ell+1}$ as $z \to 0$, and $A(\theta)$-stable if $|r(z)|\le1$ for $z \in \Sigma_\theta$. Note that if $r$ is $A(\theta)$-stable with $\theta \in (\omega(A) , \frac \pi 2] $ we can write
	$$r(z) = \gamma + \sum_{k=1}^K \sum_{j=1}^J \gamma_{j,k} (a_j+z)^{-k},$$
	where $\gamma= r(\infty)$, $K,J \in \N$ and the poles $-a_j \notin \Sigma_\theta$. Hence, we can define
	\begin{equation}\label{Eq:Definition of r(tauA)}
		r(\tau A) \coloneq \gamma I + \sum_{k=1}^K \sum_{j=1}^J \gamma_{j,k} (a_j+ \tau A)^{-k}, \quad \tau>0,
	\end{equation}
which is a bounded operator on $X$ since $-a_j/\tau \in \rho(A)$.

We focus on the exponential Euler scheme $R_\tau = e^{-\tau A}$ and on rational schemes $R_\tau=r(\tau A)$ with the extra condition that $r(\infty)=0$ since they possess stronger convergence and stability properties (see Lemma \ref{Lemma:Fractional powers of difference EE and rational}). This excludes the Crank-Nicolson scheme, but includes the commonly used implicit Euler scheme, which is consistent of order $\ell=1$ and $A(\pi/2)$-stable, and the sub-diagonal Pad\'e approximation schemes given by the rational functions $r_{n,n+1}=P_n/Q_{n+1}$ and $r_{n,n+2}=P_n/Q_{n+2}$ $(n\ge 1)$, where
		\begin{align*}
			P_n(z) \coloneq \sum_{j=0}^{n} \frac{(n+m-j)! \, n!}{(n+m)! \, j! \, (n-j)!} (-z)^j, \qquad
			Q_m(z) \coloneq \sum_{j=0}^{m} \frac{(n+m-j)! \, m!}{(n+m)! \, j! \, (m-j)!} z^j,
		\end{align*}
       which are consistent of order $\ell=2n+1$ and $\ell= 2n+2$ respectively, $A(\pi/2)$-stable, and satisfy $r_{n,n+1}(\infty)=r_{n,n+2}(\infty)=0$ (see \cite[Theorem 4.12]{zbMATH00050395}). Another example is the Padé approximation scheme given by $r_{0,3}(z)=(1+z+z^2/2!+z^3/3!)^{-1}$, which is consistent of order $\ell=3$, $A(\theta)$-stable with $\theta \le 88.23^\circ$ and satisfies the extra condition $r(\infty)=0$ (see \cite{zbMATH00050395}).

The following lemma provides powerful estimates for such rational functions $r$, and can be found in \cite[Lemmas 9.4 and 9.5]{thomee1984galerkin}.
\begin{lemma}\label{lem:rational-est}
Let $\theta\in (0,\pi/2]$. Let $r\colon \Sigma_{\theta}\to \C$ be a rational function that is consistent of order $\ell\ge1$ and $A(\theta)$-stable, and suppose that $r(\infty) = 0$.
Let $\nu\in (0,\theta)$.
Then there exist constants $c=c(r,\nu)$ and $C=C(r,\theta)$ such that for all $z\in \Sigma_{\nu}$ and $n\geq 1$,
\begin{alignat}{2}
	|r(z)^n - e^{-nz}| &\leq C n |z|^{\ell+1} e^{-cn|z|}, && \quad |z|\le 1, \label{Ineq:Difference of exp and rational z<1}
\\ |r(z)|^n& \leq C e^{-cn|z|}, && \quad |z|\leq 1,\label{Ineq:growth rational z<1}
\\	|r(z)|^n &\leq C \inv{|z|} e^{-cn}, &&\quad |z|\ge 1. \label{Ineq:Difference of exp and rational z>1}
\end{alignat}
Moreover, one can take $c \le \cos \nu$ so that $|\exp{-z}| \le \exp{-c|z|}$ for $z\in \Sigma_{\nu}$.
\end{lemma}

A key estimate needed in our proofs is an improvement of \eqref{eq:DefApproxScheme} and \eqref{eq:stab}:
\begin{lemma}\label{Lemma:Fractional powers of difference EE and rational}
 	Let $R_{\tau} = r(\tau A)$, where $r$ is either the exponential function $r(z) = e^{-z}$, or $r\colon \Sigma_{\theta}\to \C$ is a rational function which is consistent of order $\ell\ge1$ and $A(\theta)$-stable with $\theta \in ( \omega(A), \frac \pi 2]$, and $r(\infty) = 0$. Then, there is a constant $C>0$ that depends on $r,\theta$ and $A$ such that for every $\alpha \in [-1,1]$,
 \begin{alignat}{2} 
 \label{Ineq:Fractional powers of difference EE and rational}
 		\| A^\alpha (\exp{-\tau nA} -R_{\tau}^n)\|_{\L(X)} & \le C \frac{1}{ \tau^{\alpha}n^{\ell+\alpha}}, && \quad n \ge 1, \ \tau>0,
\\
\label{Ineq:Fractional powers of rational}
\| A^\alpha R_{\tau}^n\|_{\L(X)} &\le C \frac{1}{\tau^{\alpha}n^\alpha}, && \quad n \ge 1, \ \tau>0.
 \end{alignat}
 \end{lemma}
The estimate \eqref{Ineq:Fractional powers of difference EE and rational} was proved in \cite[Theorem 3.1]{zbMATH01330873} for $\alpha=1$ and $\ell \ge 2$.
The estimate \eqref{Ineq:Fractional powers of difference EE and rational} for $\alpha\in [-1,0]$ is already known, and can be found in \cite[Theorem 5.1]{BattyGomilkoTomilov}. There it is even presented for general $\alpha\in [-\ell -1,0]$ and with the condition $r(\infty) = 0$ replaced by $|r(\infty)|<1$. However, one can always reduce to $r(\infty) = 0$ by a standard trick (see \cite[Theorem 4.4]{larsson1991finite}).

 \begin{proof}
 Note that \eqref{Ineq:Fractional powers of rational} is immediate from \eqref{Ineq:Fractional powers of difference EE and rational} and the estimate $\sup_{t>0} t^{\alpha}\|A^{\alpha} e^{-t A}\|\leq C$.
 
 	 Let $f_{n}(z) \coloneq z^\alpha \big ( e^{-n z} - r(z)^{n}\big )$. Note that by \eqref{ineq:H1 calculus bound}, in order to prove \eqref{Ineq:Fractional powers of difference EE and rational}, it suffices to establish the estimate $\|f_{n}\|_{H^1(\Sigma_\theta)}\le C \frac{1}{ n^{\ell+\alpha}}$. To this end, let $\nu \in (\omega(A),\theta)$. 	
	By 
	\eqref{Ineq:Difference of exp and rational z<1}, there are constants $C=C(r,\theta)$ and $c=c(r,\nu)$ such that for every $|z| \le 1$ in the sector $\Sigma_{\nu}$, 
 	$$ |\exp{-nz} -r( z)^{n}| \le C n | z|^{\ell+1} e^{-c n| z|}$$
 	and consequently,
 	\begin{align*}
 		|f_{n}(z)| \le C n|z|^{\ell+1+\alpha} \exp{-c n|z|} .
 	\end{align*}
 	Hence,
 	\begin{align}
 		\int_{0}^{1} |f_{n} (\rho e^{\pm \nu i})| \frac{d\rho}{\rho}
 		&\le
 		C n \int_{0}^{1} \rho^{\ell+\alpha}  \exp{-c n\rho} \, d\rho \nonumber
 		\\
 		&\le C n \int_{0}^{\infty} \rho^{\ell+\alpha}  \exp{-c n\rho} \, d\rho \nonumber
 		\\
 		& = C c^{-(\ell+1+\alpha)} \Big(\frac{1}{ n}\Big)^{\ell+\alpha}  \int_0^\infty s^{\ell+\alpha} \exp{-s} \, ds 
= \tilde C \frac{1}{ n^{\ell+\alpha}}. \label{Ineq:Negative fractional powers of difference EE and rational small z}
 	\end{align}
 	In the case where $\alpha<1$,
 	by \eqref{Ineq:Difference of exp and rational z>1} we have that for $|z| \ge 1$ in the sector $\Sigma_{\nu}$,
 	$$ |r( z)^{n} | \le C | z|^{-1} e^{-c n}.$$
 Therefore, by the triangle inequality, and noting that $|\exp{-z}| \le \exp{-c|z|}$ for $z\in \Sigma_{\nu}$,
 	\begin{align*}
 		|f_n(z)|\le C |z|^{\alpha} e^{-cn|z|} + C|z|^{-1+\alpha} e^{-cn}.
 	\end{align*}
 	Consequently,
 	\begin{align*}
 		\int_{1}^\infty |f_{n}(\rho \exp{\pm \nu i})| \frac{d\rho}{\rho}
 		&\le  C \int_{1}^\infty  \frac 1{\rho^{1-\alpha}} \exp{-c n\rho} \, d\rho + C e^{-cn} \int_{1}^\infty \frac{1}{\rho^{2-\alpha}}\,d\rho
 		\\
 		& \le  C \int_{1}^\infty  \exp{-c n\rho} \, d\rho + C (1-\alpha)^{-1} e^{-cn}
 \leq \tilde C \frac{1}{n^{\ell+\alpha}}.
 	\end{align*}
 This shows that
 $$\|f_{n}\|_{H^1(\Sigma_\theta)} \coloneq \sup_{\nu\le \theta} \int_0^\infty |f_{n}(\rho \exp{\pm \nu i})| \frac{d\rho}{\rho} \le \tilde C \frac{1}{ n^{\ell+\alpha}}.$$

 Suppose now that $\alpha=1$. Note that $r(\infty)=0$ and, by the maximum modulus principle,  $|r(z)|<1$ on the interior of $\Sigma_\theta$. Hence, there are constants $c=c(r,\nu)$ and $C=C(r,\nu)$ such that  $|r(z)|^2 \le C |z|^{-2}$ and $|r(z)|\le e^{-c}$  for $|z|\ge 1$ in the sector $\Sigma_\nu$.  Therefore,  for $n\ge 2$ and $z \in \Sigma_\nu$ with $|z|\ge1$, 
 	$$ |r(z)|^n \le C|z|^{-2} e^{-c(n-2)}.$$
 	This implies that 
 	\begin{align*}
 		|f_{n}(z)| \le |z| \exp{-c n|z|} + C|z|^{-1}  e^{-c(n-2)} 
 	\end{align*}
 	and consequently
 	\begin{align*}
 		\int_{1}^\infty |f_{n}(\rho \exp{\pm \nu i})| \frac{d\rho}{\rho}
 		&\le  \int_{1}^\infty  \exp{-c n\rho} \, d\rho + C e^{-c(n-2)} \int_{1}^\infty \frac{1}{\rho^{2}}\,d\rho
 		 \le \tilde C \frac{1}{n^{\ell+1}}.
 	\end{align*}
 	This, together with \eqref{Ineq:Negative fractional powers of difference EE and rational small z}, show that for $n \ge 2$,
 	$$\|f_{n}\|_{H^1(\Sigma_\theta)} \coloneq \sup_{\nu\le \theta} \int_0^\infty |f_{n}(\rho \exp{\pm \nu i})| \frac{d\rho}{\rho} \le C \frac{1}{ n^{\ell+1}}.$$
 	For the case $n=1$, by the triangle inequality and since $\sup_{t>0}\|(t A) e^{-t A}\| < \infty$, it suffices to show that
 	$\|(\tau A) r(\tau A)\| \le C$ uniformly in $\tau$.
 	By \eqref{Eq:Definition of r(tauA)} and noting that $r(\infty)=0$ we estimate
 	\begin{align*}
 		\|(\tau A) r(\tau A)\| &= \Big \| \sum_{j=1}^{J}\gamma_{j,1} (\tau A) (a_j+\tau A)^{-1} + \sum_{k =2}^{K} \sum_{j= 1}^{J} \gamma_{j,k} (\tau A) (a_j+\tau A)^{-k} \Big\|
 		\\
 		&\le  \sum_{j =1}^{J} |\gamma_{j,1}| \, \|(\tfrac{\tau}{a_j} A)(1+\tfrac{\tau}{a_j} A)^{-1}\| + \sum_{k=2}^K \sum_{j=1}^J |\gamma_{j,k}|\, \|(\tau A)(a_j+\tau A)^{-k}\|
 		\le C
 	\end{align*} 	
 	by noting that $\sup_{t>0} \|(t A)(1+t A)^{-1}\| \le C$  and by applying \eqref{ineq:H1 calculus bound} for $ f_{k,j}(z)= z/(a_j+z)^k $ which belongs to $H^1(\Sigma_\nu)$ for $k \ge 2, j \ge 1$.
 \end{proof}

\subsection{Stochastic integration}
In principle, there is a full analogue of stochastic integration theory in infinite dimensions. However, geometric conditions on the underlying spaces are required. In order to give a satisfactory explanation of stochastic integration in a Banach space setting, we need $\gamma$-radonifying operators $\gamma(H,X)$ where $H$ is a Hilbert space and $X$ a Banach space. For details we refer to \cite[Chapter 9]{hytonen2017analysis}. Specializing to Hilbert spaces $X$, this class of operators reduces to the Hilbert-Schmidt operators $\calL_2(H,X)$. Moreover, in the important case $X = L^q(\mathcal{O})$, one can identify $\gamma(H,X)$ with $L^q(\mathcal{O};H)$.

\newcommand{\Borel}{\mathcal{B}}

Let $(\Omega,\F, \P)$ denote a probability space with filtration $\mathbb{F}=(\F_t)_{t\geq 0}$.
\begin{definition}
\label{def:Cylindrical_BM}
Let $H$ be a Hilbert space. A bounded linear operator $W\colon L^2(\R_+;H)\rightarrow L^2(\Omega)$ is said to be a {\em cylindrical Brownian motion} in $H$ if the following are satisfied:
\begin{itemize}
\item for all $f\in L^2(\R_+;H)$ the random variable $W(f)$ is centered Gaussian;
\item for all $t\in \R_+$ and $f\in L^2(\R_+;H)$ with support in $[0,t]$, $W(f)$ is $\F_t$-measurable;
\item for all $t\in \R_+$ and $f\in L^2(\R_+;H)$ with support in $[t,\infty]$, $W(f)$ is independent of $\F_t$;
\item for all $f_1,f_2\in L^2(\R_+;H)$ we have $\E(W(f_1)W(f_2))=(f_1,f_2)_{L^2(\R_+;H)}$.
\end{itemize}
\end{definition}
Given $W$, the process $t\mapsto W(\1_{(0,t]} h)$ is a Brownian motion for each $h\in H$.

The way to think about $W$ is that it is given by $t\mapsto \sum_{n\geq 1} W^n(t) h_n$, where
$(W^n)_{n\geq 1}$ are independent standard $\mathbb{F}$-Brownian motions, and $(h_n)_{n\geq 1}$ an orthonormal basis for $H$. However, since the above series does not define an $H$-valued random variable, the definition is given in a weaker sense.

However, the following convergence property does hold: if $S\in \gamma(H,X)$, then
\begin{align}\label{eq:convSWH}
S W(t) \coloneq \sum_{k\geq 1} W(t) h_k S h_k,
\end{align} 
 where the convergence takes place in $L^p(\Omega;X)$ for all $p\in [1, \infty)$.

A process $g \colon \R_+\times\Omega \to \calL(H,X)$ is called {\em $H$-strongly progressively measurable} if for all $t\in [0,T]$, $g|_{[0,t]}$ is strongly $\Borel([0,t])\otimes \F_t$-measurable (where $\Borel$ denotes the Borel $\sigma$-algebra).

For $0\leq a<b\leq T$ and a strongly $\F_a$-measurable $\xi\colon \Omega\to \gamma(H,X)$, the stochastic integral of $\1_{(a,b]} \xi$ is defined by
\begin{equation}
\int_0^{t} \1_{(a,b]} \xi d W\coloneq  \xi (W(b\wedge t)-W(a\wedge t)),
\end{equation}
where the series can be shown to be convergent as in \eqref{eq:convSWH} by the independence of $\F_a$ and $W(b\wedge t)-W(a\wedge t)$.

The space $L^p_{\mathbb{F}}((0,T)\times\Omega;\gamma(H,X))$ denotes the subspace of $L^p((0,T)\times\Omega;\gamma(H,X))$ consisting of all strongly progressively measurable processes. It can be shown that this coincides with the closure of the adapted step processes of finite rank (see \cite[Proposition 2.10]{NVW1}). Moreover, let $L^p_{\mathbb{F}}(\Omega;\ell^p_\tau(Z))$ denote the subspace of adapted elements in $L^p(\Omega;\ell^p_\tau(Z))$.

The following proposition provides a simple sufficient condition for stochastic integrability. For the definition of type $2$, see \cite[Chapter 7]{hytonen2017analysis}. Spaces of type $2$ include $L^q$, $W^{s,q}$, and $H^{s,q}$ etc.\ for $q\in [2, \infty)$ and $s\in \R$. The definition of UMD can be found in \cite[Chapter 4]{Analysis1}. Spaces with UMD include $L^q$, $W^{s,q}$, and $H^{s,q}$ with $q\in (1, \infty)$ and $s\in \R$.

\begin{proposition}[One-sided Burkholder-Davis--Gundy inequality]\label{prop:BDGtype2}
Let $X$ be a UMD Banach space with type $2$. Then for every $p\in [0,\infty)$, the mapping $g\mapsto \int_0^\cdot g\,\dd W$ extends to a continuous linear operator from $L^p_{\mathbb{F}}(\Omega;L^2(\R_+;\gamma(H,X)))$ into $L^p(\Omega;C_b([0,\infty);X))$. Moreover, for $p\in(0,\infty)$ there exists a constant $C_{p,X}$ such that for all $g\in L^p_{\mathbb{F}}(\Omega;L^2(\R_+;\gamma(H,X)))$, the following estimate holds
\begin{equation*}
\textstyle \E\sup_{t\geq 0}\Big\|\int_0^t g(s)\,\dd W(s)\Big\|_{X}^p \leq C_{p,X}^p \E\|g\|_{L^2(\R_+;\gamma(H,X))}^p.
\end{equation*}
\end{proposition}

The above result is sharp for Hilbert spaces $X$. However, if $X$ is not a Hilbert space, a more precise bound is available, which, though explicitly used only in Section \ref{sec:Rbdd}, forms the crucial basis for the proofs of Theorem \ref{thm:maxest} and Corollary \ref{cor:DSMRLq}. For simplicity, we only formulate the result for $X = L^q(\mathcal{O})$. For details the reader is referred to \cite{NVW1}.
\begin{proposition}[Two-sided Burkholder-Davis--Gundy inequality]\label{prop:BDGUMD}
Let $X = L^q(\mathcal{O})$ with $q\in (1, \infty)$. Then for every $p\in [0,\infty)$, the mapping $g\mapsto \int_0^\cdot g\,\dd W$ extends to a continuous linear operator from $L^p_{\mathbb{F}}(\Omega;L^q(\mathcal{O};L^2(\R_+;H))$ into $L^p(\Omega;C_b([0,\infty);L^q(\mathcal{O}))$. Moreover, for $p\in(0,\infty)$  and for all $g\in L^p_{\mathbb{F}}(\Omega;L^q(\mathcal{O};L^2(\R_+;H)))$, the following two-sided estimate holds
\begin{equation*}
\E\sup_{t\geq 0}\Big\|\int_0^t g(s)\,\dd W(s)\Big\|_{L^q}^p \eqsim_{p,q} \E\|g\|_{L^q(\mathcal{O};L^2(\R_+;H))}^p.
\end{equation*}
\end{proposition}
Observe that for $q\in [2, \infty)$, by identifying $\gamma(H,L^q(\O))=L^q(\mathcal{O};H)$, and by Minkowski's inequality,
$$L^2(\R_+;\gamma(H,L^q(\O))) = L^2(\R_+;L^q(\mathcal{O};H)) \hookrightarrow L^q(\mathcal{O};L^2(\R_+;H)),$$
which explains why Proposition \ref{prop:BDGUMD} gives a sharper estimate, which is even two-sided.

\section{Discrete stochastic maximal regularity}\label{sec:DSMR}
In this section, we introduce discrete stochastic maximal regularity for arbitrary schemes. A definition of the continuous variant is more standard and can be found in Definition \ref{def:contSMR}. The main result of this section shows that under mild conditions on the scheme, discrete stochastic maximal $\ell^p$-regularity and stochastic maximal $L^p$-regularity are equivalent. Since stochastic maximal $L^p$-regularity is known to hold for a large class of operators $A$, we immediately obtain a wide range of examples for the discrete setting.

From now on, we fix a probability space $(\Omega,\F, \P)$ with filtration $\mathbb{F} = (\F_{t})_{t\geq 0}$, and a cylindrical Brownian motion $W$ with respect to $\mathbb{F}$ on $H$.
Moreover, we suppose that $X_0$ and $X_1$ are UMD spaces with type $2$ such that $X_1\hookrightarrow X_0$ densely. Let $X_{\alpha} \coloneq [X_{0}, X_1]_{\alpha}$ for $\alpha\in (0,1)$ denote the complex interpolation spaces.

In the rest of this section we suppose that Assumption \ref{assum:mainDSMR} is satisfied. 
By the assumptions on $A$, one sees that $-A$ generates a bounded analytic strongly continuous semigroup, which will be denoted by $(e^{-tA})_{t\geq 0}$. Moreover, the fractional powers $A^{\alpha}$ for $\alpha>0$ are well-defined.  

Let $T \in (0,+\infty]$. We say that $\tau>0$ is {\em admissible} if $T/\tau\in \N$ or $T =\infty$. For an admissible $\tau$ let $\pi_\tau \coloneq \{t_n = n\tau \colon 0 \le n \le N \}$ be a uniform partition of $[0,T]$, where $N = T/\tau$ if $T<\infty$, and $N=\infty$ if $T = \infty$.
	For $g \in L^p_{\mathbb{F}}(\Omega;L^p (0,T;\gamma(H,X_{1/2})))$, we define $Y=(Y_n)_{n=0}^N$ recursively by
	\begin{equation} \label{Eq:Definition of approximation scheme}
	\begin{cases}
		Y_{n+1} \coloneq  R_\tau Y_{n} + R_\tau \Delta_n I_g, \quad n=0,\dots, N-1
		\\
		Y_0\coloneq 0.
	\end{cases}
	\end{equation}	
Here we set $I_g \coloneq  \int_0^\cdot g(s) d W(s)$ and
\[\Delta_n I_g \coloneq  I_g(t_{n+1}) - I_g(t_{n}) = \int_{t_{n}}^{t_{n+1}} g(s) d W(s).\]
Alternatively, we can write $Y_n$ as a discrete stochastic convolution
\begin{equation} \label{Eq: Uniform step discrete stochastic convolution}
	Y_n = \sum_{j=0}^{n-1} R_\tau^{n-j} \Delta_j I_g , \quad n=1,\dots, N.
\end{equation}

\subsection{Definition and basic properties}

The following definition is central in this section and describes a certain regularity property of $Y=(Y_n)_{n=0}^N$ as defined in \eqref{Eq:Definition of approximation scheme}.
\begin{defn}[Discrete Stochastic Maximal Regularity]\label{def:DSMR}
Suppose that Assumption \ref{assum:mainDSMR} holds.
Let $T\in(0,+\infty]$ and $p\in[2,+\infty)$. The scheme $R$ is said to have \textit{discrete stochastic $\ell^p$-maximal regularity} on $(0,T)$ if there is a constant $C>0$ such that, for every admissible $\tau>0$, every uniform partition $\pi_\tau \coloneq \{t_n=n \tau \colon 0\le n \le N\}$ of $[0,T]$, for every $g \in L^p_{\mathbb{F}} (\Omega ; L^p (0,T;\gamma(H,X_{1/2})))$
the approximation scheme $Y=(Y_n)_{n=0}^N$ given in \eqref{Eq:Definition of approximation scheme} belongs to $X_1$ a.s.\ and satisfies
	\begin{equation} \label{Ineq:DSMR definition}
		\| A Y \|_{L^p(\Omega;\ell^p_\tau(X_{0}))}  \le C \|g\|_{L^p(\Omega;L^p (0,T;\gamma(H,X_{1/2})))}.
	\end{equation}
	The least admissible constant $C$ is denoted by $C_{\DSMR(p,T)}^R$. In case the above holds, we will write $R\in \DSMR(p,T)$.
\end{defn}
The constant $C$ is allowed to depend on $T$, but should be uniform in the stepsize $\tau$. Note that by \eqref{Eq: Uniform step discrete stochastic convolution} and since $Y_0=0$, \eqref{Ineq:DSMR definition} is equivalent to
$$ \Big( \E \sum_{n=1}^{N-1} \tau \Big \| A \sum_{j=0}^{n-1} R_\tau^{n-j} \, \Delta_j I_g \Big \|^p_{X_0} \Big)^{1/p} \le C \big(\E\|g\|_{L^p (0,T;\gamma(H,X_{1/2}))}^p\big)^{1/p}.$$

Some further properties of operators with discrete stochastic $\ell^p$-maximal regularity are collected in Section \ref{sec:perm}.

\begin{remark} 
\begin{enumerate}[(1)]
  \item 
It is always true that $Y_n\in X_1$ a.s.\ since we assume $r(\infty) = 0$ (see \eqref{Eq:Definition of r(tauA)}). If $r(\infty)\neq 0$ one does not have $Y_n\in X_1$ in general unless $A$ is bounded. 

\item One can replace the homogeneous norm $\| A Y \|_{L^p(\Omega;\ell^p_\tau(X_{0}))}$ in \eqref{Ineq:DSMR definition} with the inhomogeneous one $\| Y \|_{L^p(\Omega;\ell^p_\tau(X_{1}))}$. This leads to the following inhomogeneous version of discrete stochastic $\ell^p$-maximal regularity:
\begin{equation} \label{Ineq:DSMR inhomogeneous}
		\| Y \|_{L^p(\Omega;\ell^p_\tau(X_{1}))}  \le C \|g\|_{L^p(\Omega;L^p (0,T;\gamma(H,X_{1/2})))}.
	\end{equation}
Clearly \eqref{Ineq:DSMR inhomogeneous} implies \eqref{Ineq:DSMR definition} and the two are equivalent if $0 \in \rho(A)$. Furthermore, when restricted to finite time intervals $T<\infty$ one need not assume that $0 \in \rho(A)$. Indeed, by Proposition \ref{prop:BDGtype2} and the stability result \eqref{eq:stab} we immediately obtain that
\begin{align*}
\|Y\|_{L^p(\Omega;\ell^p_\tau(X_{1/2}))} \lesssim_{p,X_0,R} T^{\frac1p} \|g\|_{L^2(0,T;\gamma(H,X_{1/2}))} \leq T^{\frac12} \|g\|_{L^p(0,T;\gamma(H,X_{1/2}))}.
\end{align*}
The same estimate holds if $X_{1/2}$ is replaced by $X_{\alpha}$ (for both $Y$ and $g$) for any $\alpha\in [0,1]$.
\end{enumerate}
\end{remark}

\begin{remark}\label{rem:BIPzero}
In case $A$ has bounded imaginary powers (BIP) and $0\in \rho(A)$, there are several equivalent ways to formulate discrete stochastic $\ell^p$-maximal regularity, and the same applies to the continuous case, which will be defined below in Definition \ref{def:contSMR}. In particular, BIP holds if $A$ has a bounded $H^\infty$-calculus. 

Under the assumption that $A$ has BIP, $A^{1/2}$ defines an isomorphism from $X_{1/2}$ to $X_{0}$, so that one could define 
discrete stochastic $\ell^p$-maximal regularity as 
\begin{equation*}
		\|A^{\frac12}Y \|_{L^p(\Omega;\ell^p_\tau(X_{0}))}  \le C \|g\|_{L^p(\Omega;L^p (0,T;\gamma(H,X_{0})))},
  \end{equation*}
   which is quite common. The advantage of \eqref{Ineq:DSMR definition} is that it aligns well with its deterministic variant, in which one has regularization from $X_0$ into $X_1$.
\end{remark}

\begin{remark}
In the special case where $g \coloneq \sum_{n=0}^{N-1} g_n \1_{[t_n,t_{n+1})}$ with $g_n \in L^p(\Omega,\F_{t_n},\P; \gamma(H,X_{1/2}))$, one can write
\[\Delta_n I_g = g_{n} \, \Delta W_{n},\]
where $\Delta_n W \coloneq W(t_{n+1}) - W(t_{n})$ and we used \eqref{eq:convSWH}. This case is a popular choice for the discrete setting as well. Note that in this situation \[\|g\|_{L^p (\Omega ; L^p (0,T;\gamma(H,X_{1/2})))} = \|(g_n)_{n\geq 0}\|_{L^p(\Omega;\ell^p_\tau( \gamma(H,X_{1/2})))}.\]
\end{remark}

\subsection{Equivalence of continuous and discrete $\SMR$}

The following definition is a special case of the definition of stochastic maximal regularity in \cite{agresti2025nonlinear}.
\begin{definition}[Continuous Stochastic Maximal Regularity]\label{def:contSMR}
Suppose that Assumption \ref{assum:mainDSMR}\eqref{it:mainDSMR1} holds.
Let $T\in(0,+\infty]$ and $p\in[2,+\infty)$. The operator $A$ is said to have \textit{stochastic $L^p$-maximal regularity} on $(0,T)$ if there is a constant $C>0$ such that, for every $g \in L^p_{\mathbb{F}} (\Omega ; L^p (0,T;\gamma(H,X_{1/2})))$, the mild solution
\begin{align}\label{eq:mildsol}
y(t) \coloneq \int_0^t e^{-(t-s)A} g(s) d W(s), \quad t\in(0,T),
\end{align}
belongs to $X_1$ a.s.\ and satisfies
	\begin{equation} \label{Ineq: SMR definition}
		\| A y \|_{L^p(\Omega;L^p(0,T;X_0))}  \le C \|g\|_{L^p(\Omega;L^p (0,T;\gamma(H,X_{1/2})))}.
	\end{equation}
	The least admissible constant $C$ is denoted by $C_{\SMR(p,T)}^A$. In case the above holds, we will write $A\in \SMR(p,T)$.
\end{definition}
Some basic properties of stochastic $L^p$-maximal regularity are collected in Proposition \ref{prop:continuoustime}.

The main result of this section is the following.
\begin{theorem}\label{thm:mainequiv}
Suppose that Assumption \ref{assum:mainDSMR} holds. Let $T\in(0,+\infty]$ and $p \in [2,\infty)$.
Then the following are equivalent:
\begin{enumerate}[(1)]
\item\label{it:mainequiv1} $A$ has stochastic maximal $L^p$-regularity on $(0,T)$;
\item\label{it:mainequiv2} $R$ has discrete stochastic maximal $\ell^p$-regularity on $(0,T)$.
\end{enumerate}
Moreover, $C_{\SMR(p,T)}^A \leq C_{\DSMR(p,T)}^R\leq K (C_{\SMR(p,T)}^A + 1)$,
where the constant $K$ depends on $p$, $X_0$, the function $r$, and on the sectoriality constant and angle of $A$.
\end{theorem}

The proof of the implication from $\SMR$ to $\DSMR$ will be given in Subsections \ref{SMR-EE-DSMR} and \ref{subs:equivDSMR} for the exponential Euler and rational schemes, respectively. The converse is simpler and will be proved in Subsection \ref{subs:DSMR-SMR}. 
See Figure \ref{fig:SMRiffDSMR} for a flowchart of the proof of Theorem \ref{thm:mainequiv}. Furthermore, in Theorem \ref{thm:weighted} we show that \eqref{it:mainequiv1} and \eqref{it:mainequiv2} are equivalent to their weighted counterparts.

\begin{figure}[H]
    \centering
    \caption{Flowchart of the proof of Theorem \ref{thm:mainequiv}.}
    \label{fig:SMRiffDSMR}
        \[\begin{tikzcd}
	{A\in \SMR(p,T)} && {\text{(EE) scheme} \in \DSMR(p,T)} \\
	& {\text{Rational scheme} \in \DSMR(p,T)}
	\arrow["{\text{Theorem \ref{SMR implies EE DSMR}}}", from=1-1, to=1-3]
	\arrow["{\text{Theorem \ref{EE DSMR implies r(z) DSMR Banach for r(infty)=0 shift}}}"', from=1-3, to=2-2]
        \arrow[from=2-2, to=1-3]
	\arrow["{\text{Theorem \ref{DSMR implies SMR shift}}}"', from=2-2, to=1-1]
\end{tikzcd}\]
\end{figure}
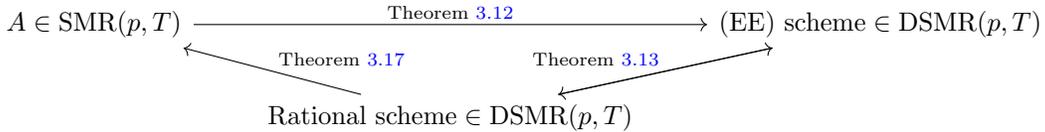

\begin{remark}[Quasi-uniform partitions] \label{remark:quasi uniform 1}
    Theorem \ref{thm:mainequiv} extends to quasi-uniform partitions $\pi$ of $[0,T]$, namely partitions that are given by variable time-steps $(\tau_n)_{n=1}^N$ for which there is a constant $\mu>0$ with
    $	 \mu^{-1} \le \frac{\tau_{\max}}{\tau_{\min}} \le \mu , $
    where $\tau_{\max}$ and $\tau_{\min}$ denote the maximum and minimum time-step, respectively. 
    Note that the constant $K$ of Theorem \ref{thm:mainequiv} will also depend on $\mu$. In this case, the approximation scheme  becomes 
    \begin{equation*} \label{Eq:Definition of approximation scheme quasi-uniform partition}
	\begin{cases}
		Y_{n+1} \coloneq  R_{\tau_{n+1}} Y_{n} + R_{\tau_{n+1}} \Delta_n I_g, \quad n=0,\dots, N-1
		\\
		Y_0\coloneq 0
	\end{cases}
	\end{equation*}	
    and the solution 
    $(Y_n)$ is given by 
    \begin{equation}\label{eq:approximation scheme quasi-uniform partition}
         Y_n = \sum_{j=0}^{n-1} R^{n,j}_\pi \Delta_j I_g , \quad n=1,\dots, N, 
    \end{equation}
    where $(R^{n,j})_{j \le n}$ denotes the discrete evolution family  given by  
$$ R^{n,j}_\pi  \coloneq \begin{cases}
	R_{\tau_n} \cdots R_{\tau_{j+1}}    &0\le j \le n-1
	\\
	I & j=n.
\end{cases} 
$$
We say that $R$ has discrete stochastic maximal $\ell^p$-regularity if for every $g \in L^p_{\mathbb{F}} (\Omega ; L^p (0,T;\gamma(H,X_{1/2})))$
the approximation scheme $Y=(Y_n)_{n=0}^N$ given in \eqref{eq:approximation scheme quasi-uniform partition}  belongs to $X_1$ a.s.\ and satisfies
$$ \Big( \E \sum_{n=1}^{N-1} \tau_{n+1} \Big \| A \sum_{j=0}^{n-1} R^{n,j}_\pi \, \Delta_j I_g \Big \|^p_{X_0} \Big)^{1/p} \le C \big(\E\|g\|_{L^p (0,T;\gamma(H,X_{1/2}))}^p\big)^{1/p}.$$
The constant $C$ does not depend on $g$ and $\pi$ but it is allowed to depend on $\mu$.
Since uniform partitions are quasi-uniform, Theorem \ref{DSMR implies SMR shift} proves the implication \eqref{it:mainequiv2}  $\implies$ \eqref{it:mainequiv1} of Theorem \ref{thm:mainequiv} for quasi-uniform partitions. In order to prove the converse implication, it suffices to check that Theorems \ref{SMR implies EE DSMR} and \ref{EE DSMR implies r(z) DSMR Banach for r(infty)=0 shift} extend to quasi-uniform partitions (see figure \ref{fig:SMRiffDSMR}).
Note that the proof of Theorem \ref{SMR implies EE DSMR} mainly relies on establishing \eqref{Eq:Must bound - SMR implies EE DSMR new Shifted} and the latter follows from estimating $\|\Phi(t,s,n,j)\|^2_{\calL(X_{1/2},X_0)} \lesssim [\tau(n-j-1)^2]^{-1}$
using standard sectoriality estimates of $A$. In the case of a quasi-uniform partition  $\Phi(t,s,n,j) = A (e^{-t_{n,j}A} - e^{-(t-s)A} )$, where $	t_{n,j}  \coloneq 	\tau_n  +\dots +\tau_{j+1}$,
and thus similar calculations give $\|\Phi(t,s,n,j)\|^2_{\calL(X_{1/2},X_0)}\lesssim  [\tau_{\max}(n-j-1)^2]^{-1}$. Therefore Theorem \ref{SMR implies EE DSMR} also holds for quasi-uniform partitions.  
Theorem \ref{EE DSMR implies r(z) DSMR Banach for r(infty)=0 shift} follows from estimate \eqref{Eq: Estimate for EE iff rational - shifted}, a consequence of Lemma \ref{Lemma:Fractional powers of difference EE and rational}. Since Lemmas \ref{lem:rational-est} and \ref{Lemma:Fractional powers of difference EE and rational} extend to  quasi-uniform partitions  (c.f.~\cite{Yan-variabletimesteps} and \cite{Saito-nonuniformpartitions}),  Theorem \ref{EE DSMR implies r(z) DSMR Banach for r(infty)=0 shift} also holds for quasi-uniform partitions.
Finally note that Proposition \ref{prop:Convergece of rational scheme regular G} also extends to quasi-uniform partitions.
\end{remark}

As a consequence, we obtain the following sufficient conditions for discrete stochastic maximal $\ell^p$-regularity. In Hilbert spaces, we always have discrete maximal $\ell^p$-regularity without any further conditions on $A$ besides sectoriality of angle $<\pi/2$.
\begin{corollary}[$\DSMR$ for free on Hilbert spaces]
Suppose that Assumption \ref{assum:mainDSMR} holds with $X_0$ a Hilbert space. Then $R$ has discrete maximal $\ell^p$-regularity on $(0,T)$ for all $p\in [2, \infty)$ and $T\in (0,\infty]$.
\end{corollary}
 \begin{proof}
 	 By the proof below \cite[Theorem 3.13]{agresti2025nonlinear}, $A$ has $\SMR(2,T)$. Therefore, arguing as in \cite[Theorem 8.2]{LoVer} (see also \cite[Theorem 3.12]{agresti2025nonlinear}), one obtains that $A$ has $\SMR(p,T)$ for any $p\in [2, \infty)$. The required result now follows from Theorem \ref{thm:mainequiv}.
 \end{proof}

In $L^q$-spaces, the result holds under the additional condition that $A$ has a bounded $H^\infty$-calculus of angle $<\pi/2$, which can be seen as a discrete analogue of the stochastic maximal regularity result of \cite{NVWSMR}.
\begin{corollary}[$\DSMR$ on $L^q$-spaces]\label{cor:DSMRLq}
Suppose that Assumption \ref{assum:mainDSMR} holds and $X_0$ is isomorphic to a closed subspace of $L^q(\mathcal{O})$ with $q\in (2, \infty)$. Suppose that $A$ has a bounded $H^\infty$-calculus of angle $<\pi/2$. Then $R$ has discrete maximal $\ell^p$-regularity on $(0,T)$ for all $p\in (2, \infty)$ and $T\in (0,\infty]$.
\end{corollary}
 \begin{proof}
 	By \cite[Theorem 1.1]{NVWSMR}, $A$ has $\SMR(p,T)$ and thus $R$ has $\DSMR(p,T)$ by Theorem \ref{thm:mainequiv}.
 \end{proof}

\begin{remark}
In the special case of the implicit Euler scheme $r(z) = (1+z)^{-1}$ and $0\in \rho(A)$, Corollary \ref{cor:DSMRLq} follows from the work of \cite{li-Xie-stability}. In the latter paper, the regularity is shifted by a regularity of $1/2$. However, due to BIP this leads to an equivalent definition as explained in Remark \ref{rem:BIPzero}. The proof in \cite{li-Xie-stability} can be seen as a discrete analogue of the technique in \cite{NVWSMR}. Due to Theorem \ref{thm:mainequiv}, one can directly apply the continuous-time setting and cover many more schemes at the same time.
\end{remark}

On real interpolation spaces, the only additional condition needed is invertibility of $A$.
\begin{corollary}[$\DSMR$ for free on real interpolation spaces]
Let $A$ be a sectorial operator of angle $<\pi/2$ on a space $E$ with UMD and type $2$, and suppose that $0\in \rho(A)$. Let $D_A(\alpha,q) \coloneq (E,D(A^n))_{\alpha/n,q}$ where $0<\alpha <n$ and $q\in [1, \infty]$. Let $X_i \coloneq D_A(\alpha+i,q), \ i\in \{0,1\},$ for some $\alpha>0$ and $q\in [2, \infty)$.
Suppose that Assumption \ref{assum:mainDSMR}\eqref{it:mainDSMR2} holds. Then $R$ has discrete maximal $\ell^p$-regularity on $(0,T)$ for all $p\in (2, \infty)$ and $T\in (0,\infty]$.
\end{corollary}
\begin{proof}
 It is well-known that $A$ defines a sectorial operator of angle $<\pi/2$ on $X_0 = D_A(\alpha,q)$. Moreover, $A$ has a bounded $H^\infty$-calculus by Dore's theorem (see \cite[Corollary 16.3.22]{Analysis3}). Thus, from Lemma \ref{lem:BIP} and the complex reiteration theorem (see \cite[Section 1.10.3]{Tri95}) we obtain $D(A^{\theta}) = [D_A(\alpha,q), D_A(\alpha+1,q)]_{\theta} = D_A(\alpha+\theta,q)$ for all $\theta\in [0,1]$. Therefore, our definition of
 $\SMR(p,T)$ coincides with the one used in \cite[Theorem 8.6]{LoVer}, and hence $A$ has $\SMR(p,T)$. It remains to apply Theorem \ref{thm:mainequiv}.
\end{proof}

\subsection{$\SMR$ implies discrete $\SMR$ for exponential Euler}\label{SMR-EE-DSMR}

As a first step, we prove that the exponential Euler method has discrete stochastic maximal regularity by a perturbation argument from the continuous-time setting.
	\begin{theorem}\label{SMR implies EE DSMR}
		Suppose that Assumption \ref{assum:mainDSMR} holds. Let $T\in(0,\infty]$ and $p \in [2,\infty)$. If $A$ has stochastic maximal $L^p$-regularity on $(0,T)$, then the exponential Euler scheme $(EE)$ has discrete maximal $\ell^p$-regularity. Moreover,
		$$C^{(EE)}_{\DSMR(p,T)} \lesssim_{p,X,A} C^A_{\SMR(p,T)}+1,$$
where the implicit constant only depends on $p$, $X$, and the sectoriality constant and angle of $A$.
	\end{theorem}
\begin{proof}
	Let $T\in(0,\infty]$ and for admissible $\tau >0$ let $\pi_\tau \coloneq \{t_n= n \tau \colon 0 \le n \le N\}$ be a uniform partition of $[0,T]$.
	Let $g\in L^p_{\mathbb{F}}(\Omega; L^p(0,T;\gamma(H,X_{1/2})))$. In order to shorten the notation, we write $\|g\|_{\Gamma_p}$ for the latter $L^p$-norm.

	Since $A$ has $\SMR(p,T)$, we infer that
	\begin{equation} \label{Definition of SMR for step functions new Shifted}
		\Big (   \sum_{n=0}^{N-1} \int_{t_{n}}^{t_{n+1}} \E \Big \| A \int_0^t e^{-(t-s)A}g(s) \, dW(s) \Big \|_{X_0}^p dt \Big )^{1/p} \le C^A_{\SMR(p,T)} \|g\|_{\Gamma_p}.
	\end{equation}
	In order to show that the exponential Euler (EE) scheme has $\DSMR(p,T)$, it suffices to establish the estimate
	\begin{equation*}
		\Big ( \sum_{n=1}^{N-1} \tau \, \E \Big \| A \sum_{j=0}^{n-1} e^{-\tau(n-j) A} \Delta_j I_g \Big \|_{X_0}^p \Big )^{1/p} \le C \|g\|_{\Gamma_p}.
	\end{equation*}
Note that by 	Proposition \ref{prop:BDGtype2} and \eqref{Eq:Shifted decay 1} we can estimate the $j=n-1$ term by
	\begin{align*}
		\Big ( \sum_{n=1}^{N-1} \tau \, \E \| A  e^{-\tau A} \Delta_{n-1} I_g \|_{X_0}^p \Big )^{1/p}
& \lesssim_{p,X_0} \Big ( \sum_{n=1}^{N-1} \tau \, \E \| A  e^{-\tau A} g \|^p_{L^2((t_{n}, t_{n+1});\gamma(H,X_0))} \Big )^{1/p}
\\ & \lesssim_{p,X_0} \Big ( \sum_{n=1}^{N-1} \tau^{\frac{p}{2}} \, \E \| A  e^{-\tau A} g \|^p_{L^p((t_{n}, t_{n+1});\gamma(H,X_0))} \Big )^{1/p}
\\ & \lesssim_{p,X_0,A} \Big ( \sum_{n=1}^{N-1} \E \|g\|^p_{L^p((t_{n}, t_{n+1});\gamma(H,X_{1/2}))} \Big )^{1/p}
\\ & = \|g\|_{\Gamma_p}.
	\end{align*}
	Therefore it remains to prove that
	\begin{equation}
		\label{Eq:DSMR of EE}
		\Big ( \sum_{n=2}^{N-1} \tau \, \E \Big \| A \sum_{j=0}^{n-2} e^{-\tau(n-j) A} \Delta_j I_g \Big \|_{X_0}^p \Big )^{1/p} \le C \|g\|_{\Gamma_p}.
	\end{equation}
	By the triangle inequality and \eqref{Definition of SMR for step functions new Shifted},
	it suffices to establish the estimate
	\begin{equation*}
		\Big ( \sum_{n=2}^{N-1} \int_{t_{n}}^{t_{n+1}} \E \Big \| A \sum_{j=0}^{n-2} e^{-\tau(n-j) A} \Delta_j I_g - A \int_0^t e^{-(t-s)A}g(s) \, dW(s) \Big \|_{X_0}^p \, dt \Big )^{1/p}  \le C \|g\|_{\Gamma_p}.
	\end{equation*}
    To this end, let $ 2\le n\le N-1$ and $t \in [t_{n},t_{n+1}]$. Note that for $j \le n-1$  we have that $ \1_{[t_j,t_{j+1} \wedge t)} =\1_{[t_j,t_{j+1})}$, whereas  $\1_{[t_j,t_{j+1} \wedge t)}=0$ for $j \ge n+1$ and so by the triangle inequality
	\begin{alignat*}{2}
     & \Big ( \sum_{n=2}^{N-1} \int_{t_{n}}^{t_{n+1}} \E \Big \| A \sum_{j=0}^{n-2} e^{-\tau(n-j) A} \Delta_j I_g - A \int_0^t e^{-(t-s)A}g(s) \, dW(s) \Big \|_{X_0}^p dt\Big)^{1/p}
		\\
		&\leq \Big( \sum_{n=2}^{N-1} \int_{t_{n}}^{t_{n+1}} \E  \Big \|  \sum_{j=0}^{n-2} \int_{t_j}^{t_{j+1}}\Phi(t,s,n,j) g(s) dW(s) \Big \|_{X_0}^p dt\Big)^{1/p}
\\ & \qquad + \Big( \sum_{n=2}^{N-1} \int_{t_{n}}^{t_{n+1}}\E  \Big \| \int_{t_{n-1}}^{t_n} A e^{-(t-s)} g(s)dW(s)\Big \|_{X_0}^p dt\Big)^{1/p} 
\\& \qquad + \Big( \sum_{n=2}^{N-1} \int_{t_{n}}^{t_{n+1}} \E \int_{t_n}^t  A\exp{-(t-s)A} g(s) dW(s)  \Big \|_{X_0}^p dt\Big)^{1/p},
	\end{alignat*}
where $\Phi(t,s,n,j)\coloneq A (\exp{-\tau(n-j)A} - \exp{-(t-s)A})$ appears in the first and main term.

The latter two terms can easily be estimated using the stochastic maximal regularity of $A$. 
	Hence, it remains to prove the estimate
	$$ \sum_{n=2}^{N-1} \int_{t_{n}}^{t_{n+1}}  	\E  \Big \| \int_0^T  \Big  ( \sum_{j=0}^{n-2} \Phi(t,s,n,j) g(s) \1_{[t_j,t_{j+1})}(s) \Big ) dW(s) \Big \|^p dt \le C^p \|g\|_{\Gamma_p}^p.$$
	By Proposition \ref{prop:BDGtype2}, it is enough to estimate
	\begin{equation} \label{Eq:Must bound - SMR implies EE DSMR new Shifted}
		\sum_{n=2}^{N-1} \int_{t_{n}}^{t_{n+1}} \E \Big ( \sum_{j=0}^{n-2} \int_{t_j}^{t_{j+1}} \|\Phi(t,s,n,j) g(s)\|^2_{\gamma(H,X_0)}\, ds \Big)^{p/2} dt \le C^p \|g\|_{\Gamma_p}^p.
	\end{equation}

In order to estimate the latter, we first bound the operator norm of $\Phi(t,s,n,j)$. We claim that for $2\leq n\leq N-1$, $0\leq j\leq n-2$, $t_{n}\leq t\leq t_{n+1}$, and $t_j\leq s\leq t_{j+1}$
\begin{align*}
\|\Phi(t,s,n,j)\|_{\calL(X_{1/2},X_0)}\leq \|\Phi(t,s,n,j)\|_{\calL(X_{1},X_0)}^{1/2} \|\Phi(t,s,n,j)\|^{1/2}_{\calL(X_0)}\lesssim \frac{1}{\tau^{1/2} (n-j-1)}.
\end{align*}
Indeed, the first bound follows by interpolation and the $\|\Phi(t,s,n,j)\|_{\calL(X_{1},X_0)}$-norm is clearly uniformly bounded. To estimate $\|\Phi(t,s,n,j)\|_{\calL(X_0)}$, let $I_{nj}$ denote the interval with endpoints $t-s$ and $\tau(n-j)$. One can check that $|I_{nj}|\leq \tau$ and $\min I_{nj} \geq \tau (n-j-1)$. Therefore,
we write
\begin{align*}
\|\Phi(t,s,n,j)\|_{\calL(X_0)} \leq \int_{I_{nj}} \|A^2e^{-rA} \| dr \lesssim \frac{|I_{nj}|}{(\min I_{nj})^2}\leq \frac{1 }{\tau(n-j-1)^2},
\end{align*}
which clearly implies the claim. 
From the claim we obtain
	\begin{align*}
		\int_{t_j}^{t_{j+1}} \|\Phi(t,s,n,j) g(s)\|^2_{\gamma(H,X_0)}\, ds
		&\lesssim \frac{1}{\tau(n-j-1)^{2}}\|g\|_{L^2(t_j, t_{j+1};\gamma(H,X_{1/2}))}^2.
	\end{align*}
	
Using that the latter is $t$-independent, we see that the left-hand side of \eqref{Eq:Must bound - SMR implies EE DSMR new Shifted} is, up to a multiplicative constant, bounded by
\begin{align*}
&\tau^{1-\frac{p}{2}} \E \sum_{n=2}^{N-1} \Big ( \sum_{j=0}^{n-2}  \1_{n-j-1\ge 1}  \frac{1}{(n-j-1)^{2}} \|g\|_{L^2(t_j, t_{j+1};\gamma(H,X_{1/2}))}^2 \Big)^{p/2}
\\
& \qquad \stackrel{\text{(i)}}{\lesssim}
\tau^{1-\frac{p}{2}} \E \sum_{n=0}^{N-2} \|g\|_{L^2(t_n, t_{n+1};\gamma(H,X_{1/2}))}^p
 \stackrel{\text{(ii)}}{\leq}
\E \sum_{n=1}^{N-2} \|g\|_{L^p(t_n, t_{n+1};\gamma(H,X_{1/2}))}^p
\\ & \qquad \leq
\E \|g\|_{L^p(0,T;\gamma(H,X_{1/2}))}^p,
\end{align*}
where we used the discrete case of Minkowski's convolution inequality in (i) and H\"older's inequality in (ii).
\end{proof}

\subsection{Equivalence of discrete $\SMR$ for different schemes}\label{subs:equivDSMR}

Our next step is to compare rational schemes to the exponential Euler scheme.
\begin{theorem}\label{EE DSMR implies r(z) DSMR Banach for r(infty)=0 shift}
Suppose that Assumption \ref{assum:mainDSMR} holds, where $R_\tau\coloneq r(\tau A)$ is the rational scheme introduced there. Let $T\in(0,\infty]$ and $p \in (2,\infty)$.
Then, the exponential Euler scheme has discrete maximal $\ell^p$-regularity if and only if $R$ has discrete maximal $\ell^p$-regularity. Moreover,
		$$C^{(EE)}_{\DSMR(p,T)} +1 \eqsim_{p,X_0,r,A} C^{R}_{\DSMR(p,T)}+1,$$
where the implicit constant only depends on $p$, $X_0$, the rational function $r$, and the sectoriality constant and angle of $A$. 	\end{theorem}

 \begin{proof}[Proof of Theorem \ref{EE DSMR implies r(z) DSMR Banach for r(infty)=0 shift}]
 	 By Lemma \ref{Lemma:Fractional powers of difference EE and rational} we get that
 	 $\|A (e^{-n \tau A} - r(\tau A)^n)\|_{\L(X_0)}\lesssim_{r,A} \tau^{-1} n^{-\ell-1}$ and
 	 $\|A (e^{-n \tau A} - r(\tau A)^n)\|_{\calL(X_1, X_0)} \lesssim_{r,A}  n^{-\ell}$. Hence by interpolation, 
	\begin{equation} \label{Eq: Estimate for EE iff rational - shifted}
		 \|A (e^{-n \tau A} - r(\tau A)^n)\|_{\calL(X_{1/2},X_0)} \lesssim_{r,A} \frac{1}{\tau^{1/2}n^{\ell+1/2}}.
	\end{equation} 	
	Let $\tau >0$ be admissible, $\pi_\tau \coloneq \{t_n = n \tau \colon 0\le n \le N\}$ be a uniform partition of $[0,T]$ and let $g \in L^p_{\mathbb F}(\Omega;L^p(0,T;(\gamma(H,X_{1/2}))))$.
 	To prove the equivalence, by the triangle inequality, it suffices to bound the difference of the schemes uniformly in $\tau$. Thus it suffices to establish the estimate
 		\begin{equation} \label{eq:toproveEEotherscheme} 	
\sum_{n=1}^{N-1} \tau \, \E \Big \| \sum_{j=0}^{n-1}A \big ( e^{-\tau(n-j) A} -r(\tau A)^{n-j} \big ) \Delta_j I_g \Big \|_{X_0}^p \le C^p \|g\|_{L^p(\Omega;L^p(0,T;(\gamma(H,X_{1/2}))))}^p .
 \end{equation}
By Proposition \ref{prop:BDGtype2} and \eqref{Eq: Estimate for EE iff rational - shifted}, we can estimate
 	\begin{align*}
 		&	\E \Big \| \sum_{j=0}^{n-1}A \big ( e^{-\tau(n-j) A} -r(\tau A)^{n-j} \big ) \Delta_j I_g \Big \|_{X_0}^p
 		\\
 		&\lesssim_{p,X_0} \E \Big(\sum_{j=0}^{n-1}\int_{t_j}^{t_{j+1}} \, \| A \big ( e^{-\tau(n-j) A} -r(\tau A)^{n-j} \big ) g(s) \|^2_{\gamma(H,X_0)} \, ds\Big)^{p/2}
 		\\
 		& \le \E \Big(\sum_{j=0}^{n-1} \| A \big ( e^{-\tau(n-j) A} -r(\tau A)^{n-j} \big ) \|_{\L(X_{1/2},X_0)}^2  \, \|g\|^2_{L^2(t_j, t_{j+1};\gamma(H,X_{1/2}))}\Big)^{p/2}
 		\\
 		& \lesssim_{r,A} \E \Big(\sum_{j=0}^{n-1} \frac{1}{\tau (n-j)^{2\ell+1}} \, \|g\|^2_{L^2(t_j, t_{j+1};\gamma(H,X_{1/2}))}\Big)^{p/2}.
 	\end{align*}
 Therefore, up to a multiplicative constant, the left-hand side of \eqref{eq:toproveEEotherscheme} can be estimated by
 \begin{align*}
\sum_{n=1}^{N-1} \tau \, \E \Big(\sum_{j=0}^{n-1} \frac{1}{\tau (n-j)^{2\ell+1}} \, \|g\|^2_{L^2(t_j, t_{j+1};\gamma(H,X_{1/2}))}\Big)^{p/2}
   & \stackrel{\text{(i)}}{\lesssim} \tau^{1-\frac p2}\E \sum_{n=0}^{N-1} \|g\|^{p}_{L^2(t_n, t_{n+1};\gamma(H,X_{1/2}))} 
  \\
  &\stackrel{\text{(ii)}}{\leq} \E \sum_{n=0}^{N-1} \|g\|^{p}_{L^p(t_n, t_{n+1};\gamma(H,X_{1/2}))} 
 \\ &\le \E \|g\|^{p}_{L^p(0,T;\gamma(H,X_{1/2}))},
 \end{align*}
 where we used the discrete case of Minkowski's convolution inequality in (i) and H\"older's inequality in (ii).
 \end{proof}

As a consequence, we immediately obtain the following:
\begin{corollary} [Equivalence of $\DSMR$ for all rational schemes] \label{Cor: Equivalence of DSMR for rational schemes shift}
 		Suppose that Assumption \ref{assum:mainDSMR}\eqref{it:mainDSMR1} holds. Let $T\in(0,\infty]$ and $p \in [2,\infty)$. Let $r(z)$ and $s(z)$ be consistent of order $\geq 1$, and $A(\theta)$-stable rational functions with $\theta \in (\omega(A), \frac \pi 2]$ and $r(\infty)=s(\infty)=0$.
Let $R_{\tau} = r(\tau A)$ and $S_{\tau} = s(\tau A)$.
 		Then, the scheme $R$ has discrete stochastic maximal $\ell^p$-regularity on $(0,T)$ if and only if $S$ has discrete stochastic maximal $\ell^p$-regularity on $(0,T)$.
 	\end{corollary}

\subsection{Discrete SMR implies SMR}\label{subs:DSMR-SMR}

To prove that discrete stochastic maximal regularity implies its continuous analogue, we need a result on the convergence of the underlying scheme. The literature contains many results of this type, and the reader is, for instance, referred to \cite{GyMil09, JenKlo09, LordPowSha} and references therein for an overview on this topic. The main feature of the result below is that the convergence rate seems to be optimal and contains parameters $\alpha$ and $\beta$, which entail parabolic smoothing. Even the choice $\alpha=\beta=0$ seems to be new in the setting we consider. 
\begin{proposition} \label{prop:Convergece of rational scheme regular G}
Suppose that Assumption \ref{assum:mainDSMR} holds. Let $0\leq \alpha\leq \beta\leq 1$ be such that $\beta-\alpha<1/2$. 
Let $T \in (0,\infty]$ and $p \in [2,\infty)$. For admissible $\tau$, let $\pi_\tau \coloneq \{t_n = n\tau \colon 0 \le n \le N \}$ be a uniform partition of $[0,T]$ with $N = T/\tau$.
Then there is a constant $C$ depending only on $\alpha, \beta,p,X_0,X_1$, the sectoriality constants of $A$, and the function $r$, such that for every piecewise constant $g \coloneq \sum_{n=0}^{N-1} g_n \1_{[t_{n},t_{n+1})}$ with $g_n \in L^p(\Omega,\F_{t_n},\P; \gamma(H,X_{\alpha}))$, $0 \le n \le N-1$,
$$\Big(\sum_{n=0}^{N-1} \|A^{\beta}(y-Y_n)\|^p_{L^p((t_n,t_{n+1})\times \Omega ; X_0)} \Big)^{1/p} \le C \tau^{\frac12+\alpha-\beta}  \|g\|_{L^p((0,T)\times \Omega; \gamma(H,X_\alpha))},$$
where $Y$ is the discrete solution defined in \eqref{Eq:Definition of approximation scheme} and $y$ is the mild solution as defined in \eqref{eq:mildsol}.
\end{proposition}
The constant $C$ does not depend on $\tau$ and $T$. The above result can be seen as an optimal convergence result with parabolic smoothing. 
\begin{proof}
Let $n\ge0$ and $t\in (t_n, t_{n+1}]$. 
		We split the expression as 
		\begin{align*}
				\|A^{\beta}(y(t)-Y_n)\|_{X_0}
			&=\Big\|A^{\beta} \int_0^t \exp{-(t-s)A} g(s) \dWHs - \sum_{j=0}^{n-1} A^{\beta} R_\tau^{n-j} \Delta_j I_g \Big \|_{X_0}
			\\
			& = \Big \| \sum_{j=0}^{n-1} \int_{t_j}^{t_{j+1}} A^{\beta}(\exp{-(t-s)A} - R_\tau^{n-j}) g_j \dWHs + \int_{t_n}^t A^\beta \exp{-(t-s)A} g_n \dWHs \Big \|_{X_0}
			\\
			& \le \Big \| \sum_{j=0}^{n-1} \int_{t_j}^{t_{j+1}} A^{\beta}(\exp{-(t-s)A} - \exp{-\tau(n-j)A} ) g_j \dWHs  \Big \|_{X_0}
			\\ & \ +
			\Big \| \sum_{j=0}^{n-1} \int_{t_j}^{t_{j+1}} A^{\beta}(\exp{-\tau(n-j)A} -R_\tau^{n-j}) g_j \dWHs \Big \|_{X_0}
			+ \Big \| \int_{t_n}^t A^{\beta}\exp{-(t-s)A} g_n \dWHs \Big \|_{X_0}			\\
			& \coloneq I_{1,n}(t)+I_{2,n}(t) + I_{3,n}(t),
		\end{align*}
	where we set $I_{1,0} = I_{2,0} = 0$. 
		
		We estimate the $L^p(\Omega)$ norms of $I_{1,n}(t), I_{2,n}(t)$ and $I_{3,n}(t)$ as follows. For $I_{3,n}(t)$, Proposition \ref{prop:BDGtype2} gives that 
		\begin{align*}
			\|I_{3,n}(t)\|^p_{L^p(\Omega)} &\lesssim \E \Big( \int_{t_n}^t \|A^{\beta}\exp{-(t-s)A}\|^2_{\calL(X_\alpha,X_0)} \|g_n\|_{\gamma(H,X_\alpha)}^2 ds \Big)^{p/2} 
\\ & \lesssim \E \Big( \int_{t_n}^t (t-s)^{-2(\beta-\alpha)}\|g_n\|_{\gamma(H,X_\alpha)}^2 ds \Big)^{p/2} 
\\ & \lesssim  \tau^{p (\frac 12-(\beta-\alpha))} \E \|g_n\|^p_{\gamma(H,X_\alpha)}.
		\end{align*}
		To estimate $I_{1,n}(t)$, note that by Proposition \ref{prop:BDGtype2},
		\begin{align} \label{Eq: EE converges to mild solution new I1}
			\|I_{1,n}(t)\|^p_{L^p(\Omega)} & \lesssim
\E \Big ( \sum_{j=0}^{n-1} \int_{t_j}^{t_{j+1}} \| A^{\beta}(\exp{-(t-s)A} - \exp{-\tau(n-j)A}) g_j\|^2_{\gamma(H,X_0)} \, ds \Big )^{p/2}   \nonumber
			\\
			& \le \E \Big ( \sum_{j=0}^{n-1} \int_{t_j}^{t_{j+1}} \| A^{\beta}(\exp{-(t-s)A} - \exp{-\tau(n-j)A}) \|_{\L(X_\alpha,X_0)}^2 \, \| g_j\|^2_{\gamma(H,X_\alpha)} \, ds \Big )^{p/2} .
		\end{align}
For the term $j=n-1$ and $s\in (t_{n-1}, t_{n})$ we can estimate
\begin{align*}
\| A^{\beta}(\exp{-(t-s)A} - \exp{-\tau A})\|_{\L(X_\alpha,X_0)} \leq \|A^{\beta} \exp{-(t-s)A}\|_{\L(X_\alpha,X_0)} + \|A^{\beta}\exp{-\tau A} \|_{\L(X_\alpha,X_0)} \, \lesssim (t-s)^{-(\beta-\alpha)}. 
\end{align*}
Therefore, 
\begin{align*}
\int_{t_{n-1}}^{t_{n}} \|A^{\beta}( \exp{-(t-s)A} - \exp{-\tau A} )\|_{\L(X_\alpha,X_0)}^2 \, \|g_{n-1}\|^2_{\gamma(H,X_\alpha)} \, ds & \lesssim \int_{t_{n-1}}^{t_n} (t-s)^{-2(\beta-\alpha)} \|g_{n-1}\|^2_{\gamma(H,X_\alpha)} ds
\\ & \lesssim \tau^{1-2(\beta-\alpha)} \|g_{n-1}\|^2_{\gamma(H,X_\alpha)} 
\end{align*}
and thus 
\begin{align*}
\E\Big(\int_{t_{n-1}}^{t_{n}} \| A^{\beta}(\exp{-(t-s)A} - \exp{-\tau A}) \|_{\L(X_\alpha,X_0)}^2 \, \|g_{n-1}\|^2_{\gamma(H,X_\alpha)} \, ds\Big)^{p/2}  \lesssim \tau^{p (\frac 12-(\beta-\alpha))} \E\|g_{n-1}\|^p_{\gamma(H,X_\alpha)}.
\end{align*}
For $j\leq n-2$ and $s\in (t_j, t_{j+1}]$, by \eqref{Eq: Decay rate of semigroup difference} we can estimate 
  \begin{align*}
   \| & A^{\beta}(\exp{-(t-s)A} - \exp{-\tau(n-j)A}) \|_{\L(X_{\alpha},X_0)} \\ & \leq \| A^{\beta-1} (\exp{-(t-s - (n-j-1)\tau)A} - \exp{-\tau A})\|_{\L(X_\beta , X_0)} 
  \|A \exp{-(n-j-1)\tau A}\|_{\L(X_\alpha,X_\beta)}
  \\ &   
   \lesssim (t-s-\tau(n-j))  ((n-j-1)\tau)^{-(1+\beta-\alpha)} \leq \tau^{-(\beta-\alpha)} (n-j-1)^{-(1+\beta-\alpha)}. 
\end{align*}
 Substituting this into \eqref{Eq: EE converges to mild solution new I1} gives
$$
\|I_{1,n}(t)\|^p_{L^p(\Omega)}  \lesssim   \tau^{p (\frac 12-(\beta-\alpha))} \Big[
 \E \Big ( \sum_{j=0}^{n-2} \frac{1}{(n-j-1)^{2(1+\beta-\alpha)}}\| g_j\|^2_{\gamma(H,X_{\alpha})}  \Big )^{p/2}  + \E\|g_{n-1}\|^p_{\gamma(H,X_\alpha)}\Big].
$$	
We now estimate $I_{2,n}(t)$. By Proposition \ref{prop:BDGtype2},	
    \begin{align} \label{Eq: EE converges to mild solution new I2}
			\|I_{2,n}(t)\|^p_{L^p(\Omega)} & \lesssim \E \Big ( \sum_{j=0}^{n-1} \int_{t_j}^{t_{j+1}} \| A^\beta(\exp{-\tau(n-j)A} -R_\tau^{n-j}) g_j\|^2_{\gamma(H,X_0)} \, ds \Big )^{p/2}  \nonumber
			\\
			& \le \E \Big ( \sum_{j=0}^{n-1} \int_{t_j}^{t_{j+1}} \|A^\beta( \exp{-\tau(n-j)A} -R_\tau^{n-j})\|_{\L(X_{\alpha},X_0)}^2 \, \| g_j\|^2_{\gamma(H,X_\alpha)} \, ds \Big )^{p/2} .
		\end{align}
		By \eqref{Ineq:Fractional powers of difference EE and rational}, we get that $\|A^{\beta} (\exp{-\tau(n-j)A} - R_\tau^{n-j})\|_{\L(X_\alpha,X_0)} \lesssim \tau^{-(\beta-\alpha)} (n-j)^{-1-(\beta-\alpha)}$ and thus \eqref{Eq: EE converges to mild solution new I2} gives that 
\[\|I_{2,n}(t)\|^p_{L^p(\Omega)} \lesssim \tau^{p(\frac12-(\beta-\alpha))} \E \Big ( \sum_{j=0}^{n-1} \frac1{(n-j)^{2(1+\beta-\alpha)}} \| g_j\|^2_{\gamma(H,X_\alpha)} \, ds \Big )^{p/2} .\]

Finally, applying the discrete case of Minkowski's convolution inequality, we can conclude that
\begin{align*}
\sum_{n=1}^{N-1} \|A^\beta(y-Y_n)\|^p_{L^p((t_n,t_{n+1})\times \Omega;X_0)} &\leq \sum_{k=1}^3 \sum_{n=0}^{N-1} \int_{t_n}^{t_{n+1}} \|I_{k,n}(t)\|_{L^p(\Omega)}^p dt
\\ & \lesssim \tau^{p(\frac12-(\beta-\alpha))} \E\sum_{n=0}^{N-1} \tau \| g_n\|^p_{\gamma(H,X_\alpha)} . \qedhere
\end{align*}
\end{proof}

\begin{remark}
Proposition \ref{prop:Convergece of rational scheme regular G} seems to be new in this generality. Under the assumption that $A$ has deterministic maximal $L^p$-regularity and $0\in \rho(A)$, a similar result was shown in \cite[Theorem 4.1]{li-Xie-stability} for $\alpha=\beta=0$ and for the implicit Euler scheme with an entirely different method. 
\end{remark}

		\begin{theorem}\label{DSMR implies SMR shift}
		Suppose that Assumption \ref{assum:mainDSMR} holds. Let $T\in(0,\infty]$ and $p \in [2,\infty)$.
If $R$ has discrete stochastic $\ell^p$-maximal regularity on $(0,T)$, then $A$ has stochastic maximal $L^p$-regularity on $(0,T)$.
		\end{theorem}
	\begin{proof}
		Let $\tau_0>0$ be admissible and fixed and let $\pi_{\tau_0} \coloneq \{t_n^0=n \tau_0\colon \ 0\le n \le N_0 \}$ be a uniform partition of $[0,T]$. Let $g \coloneq \sum_{n=0}^{N-1} g_n \1_{[t_n^0,t_{n+1}^0)}$ be fixed but arbitrary, where each $g_n \in L^p(\Omega,\F_{t_n},\P; \gamma(H,D(A))$ and let $y(t)\coloneq\int_0^t e^{-(t-s)A} g(s) \dWHs$. Since the set of such processes $g$ is dense in $L^p_{\mathbb{F}} ((0,T)\times \Omega; \gamma(H,X_{1/2}))$, it suffices to show that
\begin{align}\label{eq:DSMRimpliesSMR}
\|A y \|_{L^p((0,T)\times\Omega;X_0)}
		\le C \|g\|_{L^p((0,T)\times\Omega;\gamma(H,X_{1/2}))}.
\end{align}	
For each integer $k \ge 1$ let $\tau_k= \tau_0/k$ and let $\pi_{\tau_k} \coloneq \{t_n^k = n \tau_k \colon 0 \le n \le N_k\}$ with $N_k = T/\tau$ if $T<\infty$ and $N_k=\infty$ otherwise. We can write $g= \sum_{n=0}^{N_k-1} g_n^k \1_{[t_n^k,t_{n+1}^k)}$, where each $g_n^k \in L^p(\Omega,\F_{t_n^k},\P; \gamma(H,X_1))$ and let $Y_n^k \coloneq \sum_{j=0}^{n-1} R_{\tau_k}^{n-j} \Delta_j I_{g}$, where $\Delta_j I_{g} \coloneq I_g(t_{j+1}^k) - I_g(t_{j}^k)$. 
 Proposition \ref{prop:Convergece of rational scheme regular G} gives that
	\begin{equation*}\label{Eq:DSMR implies SMR convergence of approximate solution for AG}
			\Big( \sum_{n=0}^{N_k-1} \int_{t_n}^{t_{n+1}} \E \| A Y_n^k - A y(t) \|^p_{X_0} dt \Big)^{1/p} \le C  \tau_k^{\frac 12} \|g\|_{L^p((0,T)\times\Omega;\gamma(H,X_1))}
		\end{equation*}
    and since $R$ has $\DSMR(p,T)$,
\begin{equation*}
   \Big( \sum_{n=0}^{N_k-1} \tau \E \| A Y_n^k\|^p_{X_0} \Big)^{1/p} \le C_{\text{DSMR}(p,T)}^R \|g\|_{L^p((0,T)\times\Omega;\gamma(H,X_{1/2}))}.
\end{equation*}
    Hence, by the triangle inequality, 
		\begin{align*}
		&	\| Ay \|_{L^p((0,T)\times\Omega;X_0)} \le C \tau_k^{\frac 12} \|g\|_{L^p((0,T)\times\Omega;\gamma(H,X_{1}))}
				+ C_{\DSMR(p,T)}^R \|g\|_{L^p((0,T)\times\Omega;\gamma(H,X_{1/2}))},
		\end{align*}
		and by letting $k \to \infty$, we obtain 
\eqref{eq:DSMRimpliesSMR}. This shows that $A$ has stochastic $L^p$-maximal regularity on $(0,T)$ with constant $C_{\text{SMR}(p,T)}^A \le C_{\DSMR(p,T)}^R$.
		
	\end{proof}

Observe that Theorem \ref{thm:mainequiv} follows by combining Theorems \ref{SMR implies EE DSMR}, \ref{EE DSMR implies r(z) DSMR Banach for r(infty)=0 shift}, and \ref{DSMR implies SMR shift}. In particular, this proves Theorem \ref{thm:introequiv}.

\section{Permanence properties}\label{sec:perm}
Before we move on to the proof of the maximal estimate of Theorem \ref{thm:maxest}, we discuss some simple properties of discrete stochastic maximal $\ell^p$-regularity, which we can now deduce from Theorem \ref{thm:mainequiv}. Similar results have been discussed in the continuous-time setting in the deterministic case in \cite{Dore} (also see \cite[Section 17.2.e]{Analysis3}), and in the stochastic case in \cite{AVstab,LoVer}. In Subsection \ref{ss:conttime}, we recall some of the required permanence properties in continuous time, and in  Subsection \ref{ss:discreteperm} derive its discrete analogues. Finally, in Subsection \ref{ss:weights} we present a weighted extrapolation result for $\DSMR$.

\subsection{The continuous-time setting}\label{ss:conttime}

In the next result, we first collect some continuous-time results which can be found in \cite{AVstab,LoVer}. The definition in the latter two papers differs from the one we use here, since we shifted the smoothness by $1/2$, and consider complex interpolation spaces instead of fractional powers. Therefore, we need to indicate the necessary changes in the proofs. 

Note that Definition \ref{def:contSMR} does not require Assumption \ref{assum:mainDSMR}, but it is enough to assume that $-A$ generates a strongly continuous analytic semigroup.
\begin{proposition}[Continuous-time setting]\label{prop:continuoustime}
Let $-A$ generate a strongly continuous analytic semigroup on $X_0$. Let $p\in [2, \infty)$ and $T\in (0,\infty]$. Suppose that $A$ has stochastic maximal $L^p$-regularity on $(0,T)$ with respect to a cylindrical Brownian motion on $H$ with $\text{dim}(H)\geq 1$. Then the following hold:
\begin{enumerate}[(1)]
\item\label{it1:continuoustime} If $T<\infty$ and $\lambda\in \C$, then $A+\lambda\in \SMR(p,T)$;
\item\label{it2:continuoustime} If $T=\infty$ and $\Re(\lambda)\geq 0$, then $A+\lambda\in \SMR(p,\infty)$;
\item\label{it3:continuoustime} If $T<\infty$ and $\lim_{t\to \infty}\|e^{-tA}\|_{\L(X_0)} = 0$, then  $A\in \SMR(p,\infty)$;
\item\label{it4:continuoustime} If $\wt{T}\in (0,\infty)$, then $A\in \SMR(p,\wt{T})$;
\item\label{it5:continuoustime} If $q\in (2, \infty)$, then $A\in \SMR(q,T)$;
\item\label{it6:continuoustime} If $\wt{H}$ is another Hilbert space, then $A\in \SMR(p,T)$ with respect to any cylindrical Brownian motion on $\wt{H}$.
\end{enumerate}
\end{proposition}
\begin{proof}
\eqref{it1:continuoustime}, \eqref{it2:continuoustime}: This can be proved in a similar way as \cite[Proposition 3.8]{AVstab}.

\eqref{it3:continuoustime}: The proof of \cite[Theorem 5.2]{AVstab} extends to our setting.

\eqref{it4:continuoustime}: The same argument as in \cite[Proposition 5.1 and Corollary 5.3]{AVstab} can be applied.

\eqref{it5:continuoustime}: \ The $p$-independence for $T = \infty$ can be proved as in \cite[Theorem 8.2]{LoVer}. If $T<\infty$, then we can use a simple shift argument. Let $\lambda\geq 0$ be such that $\lim_{t\to \infty}\|e^{-t(\lambda+A)}\|_{\calL(X_0)}=0$. By \eqref{it2:continuoustime}, $\lambda+A\in \SMR(p,T)$, and thus $\lambda+A\in \SMR(p,\infty)$ by \eqref{it3:continuoustime}. Hence $\lambda+A\in \SMR(q,\infty)$. By \eqref{it4:continuoustime} and \eqref{it1:continuoustime} this implies $\lambda+A\in \SMR(q,T)$.

\eqref{it6:continuoustime}: This can be proved in a similar way as in \cite[Theorem 3.9]{AVstab}.
\end{proof}

We do not know whether the assumption that $-A$ generates a strongly continuous analytic semigroup on $X_0$ can be weakened. Some results in this direction can be found in \cite[Theorem 4.1]{AVstab}.

\subsection{The discrete setting}\label{ss:discreteperm}
The main result on permanence properties in the discrete case can be formulated as follows.
\begin{theorem}
Suppose that Assumption \ref{assum:mainDSMR} holds. Let $p\in [2, \infty)$ and $T\in (0,\infty]$. Suppose that $R$ has discrete stochastic maximal $\ell^p$-regularity on $(0,T)$ with respect to a cylindrical Brownian motion on $H$ with $\text{dim}(H)\geq 1$. Then the following hold:
\begin{enumerate}[(1)]
\item If $\wt{T}\in (0,\infty)$, then $R$ has discrete stochastic maximal $\ell^p$-regularity on $(0,\wt{T})$;
\item If $T<\infty$ and $\lim_{t\to \infty}\|e^{-tA}\|_{\L(X_0)} = 0$, then $R$ has discrete stochastic maximal $\ell^p$-regularity on $(0,\infty)$;
\item If $q\in (2, \infty)$, then $R$ has discrete stochastic maximal $\ell^q$-regularity on $(0,T)$;
\item If $\wt{H}$ is another Hilbert space, then $R$ has discrete stochastic maximal $\ell^p$-regularity on $(0,T)$ with respect to any cylindrical Brownian motion on $\wt{H}$.
\end{enumerate}
\end{theorem}
\begin{proof}
All properties are immediate from Proposition \ref{prop:continuoustime} and Theorem \ref{thm:mainequiv}.
\end{proof}

\subsection{Weighted extrapolation}\label{ss:weights}
In the previous result, we saw that discrete stochastic maximal regularity is independent of $p$, $T$, and $H$. Below, we will show that it is also equivalent to a weighted variant. In the continuous-time setting, such weighted results form a central tool in the theory of evolution equations (see \cite{agresti2025nonlinear,PrussSim,wilke2023linear}). The following can be seen as the discrete stochastic analogue of \cite[Theorem 7.9]{AVstab} and \cite[Theorem 2.4]{PruSim04}.

To extend Definition \ref{def:DSMR} to the weighted setting, let $p\in (2, \infty)$, $\alpha\in (-1,\frac{p}{2}-1)$ and set $w_{\alpha}(t) = t^{\alpha}$. We say that $R\in \DSMR(p,\alpha,T)$ in case Definition \ref{def:DSMR} is satisfied with the estimate \eqref{Ineq:DSMR definition} replaced by
\begin{equation*}
		\| A Y \|_{L^p(\Omega;\ell^p_{\tau,w_{\alpha}}(X_{0}))}  \le C \|g\|_{L^p(\Omega;L^p (0,T,w_{\alpha};\gamma(H,X_{1/2})))}.
	\end{equation*}
Here, for $y = (y_n)_{n\ge0}$ in $X_0$, we set
\[\|y\|_{\ell^p_{\tau,w_{\alpha}}(X_0)} = \Big(\sum_{n\ge0} \tau w_{\alpha}((n+1)\tau) \|y_n\|_{X_0}^p\Big)^{1/p} = \Big(\sum_{n\ge0} \tau w_{\alpha}(t_{n+1}) \|y_n\|_{X_0}^p\Big)^{1/p},\]
and for $f\colon (0,T)\to Z$ we write
\[\|f\|_{L^p (0,T,w_{\alpha};Z)} = \Big(\int_{0}^T \|f(t)\|_Z^p w_{\alpha}(t) dt\Big)^{1/p}.\]
In a similar way, one can define $A\in \SMR(p,\alpha,T)$ by including weights in Definition \ref{def:contSMR}. 

The following result is a weighted extension of Theorem \ref{thm:mainequiv}. 
\begin{theorem}\label{thm:weighted}
Suppose that Assumption \ref{assum:mainDSMR} holds. Let $p\in (2, \infty)$, $\alpha\in (-1,\frac{p}{2}-1)$, and $T\in (0,\infty]$. Then the following are equivalent:
\begin{enumerate}[(1)]
\item $A$ has $\SMR(p,T)$; \label{thm:weighted item 1}
\item $R$ has $\DSMR(p,T)$; \label{thm:weighted item 2}
\item $A$ has $\SMR(p,\alpha,T)$; \label{thm:weighted item 3}
\item $R$ has $\DSMR(p,\alpha,T)$. \label{thm:weighted item 4}
\end{enumerate}
\end{theorem}

This result will be derived from a more general extrapolation result for more general kernels.
\begin{proposition}\label{prop:generalweightkernel}
Let $Y$ be a Banach space, and $Z$ be a Banach space with UMD and type $2$. Let $\Delta = \{(n,j) \in \N^2\colon 0\leq j<n\}$.
Let $p\in (2, \infty)$ and $\alpha\in (-1, \frac{p}{2}-1)$.
Suppose that $K\colon \Delta\to \calL(Y,Z)$ is such that, for some constant $M\ge 0$,
\begin{align*}
\|K(n,j)\|_{\L(Y,Z)}\leq \frac{M}{(\tau(n-j))^{1/2}}, \quad (n,j)\in \Delta.
\end{align*}
For an adapted step process $g\colon \R_+\times\Omega\to \gamma(H,Y)$, let $S_K g$ be the $L^p(\Omega;Y)$-valued sequence given by
\[(S_K g)_0 \coloneq 0 \quad \text{ and } \quad (S_K g)_n \coloneq \sum_{j=0}^{n-1} K(n,j) \Delta_j I_g,\quad n \ge 1.\]
Then the following are equivalent
\begin{enumerate}[(1)]
\item $S_K$ extends to a bounded operator with
\[\|S_K\|_{\calL(L^p_{\mathbb{F}} (\Omega ; L^p (0,T;\gamma(H,Y))), L^p(\Omega;\ell^p_{\tau}(Z)))}\leq C_1;\]
\item $S_K$ extends to a bounded operator with
\[\|S_K\|_{\calL(L^p_{\mathbb{F}} (\Omega ; L^p (0,T,w_{\alpha};\gamma(H,Y))),L^p(\Omega;\ell^p_{\tau,w_{\alpha}}(Z)))}\leq C_2.\]
\end{enumerate}
Moreover, there is a constant $C_{\alpha,p,Z}$ only depending on $\alpha, p, Z$ such that $C_1\leq C_2 + C_{\alpha, p,Z} M$ and $C_2\leq C_1 + C_{\alpha, p,Z} M$.
\end{proposition}
To prove this, we need the following.
\begin{lemma}\label{lem:weighted}
Suppose that the assumptions of Proposition \ref{prop:generalweightkernel} are satisfied. Let $\beta\in (-\infty, \frac12 - \frac1p)$. Then $S_{K, \beta}\colon  L^p_{\mathbb{F}} (\Omega ; L^p (0,T;\gamma(H,Y)))\to L^p(\Omega;\ell^p_{\tau}(Z))$ defined by
\[(S_{K,\beta} g)_0 \coloneq 0 \quad \text{ and } \quad (S_{K, \beta}g)_n \coloneq \sum_{j=0}^{n-1} K(n,j) \int_{t_j}^{t_{j+1}} [(t_{n+1}/s)^{\beta}-1] g(s) d W(s), \quad n \ge1, \]
is bounded of norm $\|S_{K, \beta}\|\leq C_{p,Z} C_{\beta} M$.
\end{lemma}
\begin{proof}
By Proposition \ref{prop:BDGtype2}
\begin{align*}
\E\|(S_{K, \beta}g)_n\|^p_Z& \leq C_{p,Z}^p \E \Big(\sum_{j=0}^{n-1} \int_{t_j}^{t_{j+1}} \|K(n,j) [(t_{n+1}/s)^{\beta}-1] g(s)\|_{\gamma(H,Z)}^2 ds\Big)^{p/2}
\\ & \leq C_{p,Z}^p M^p \E \Big(\sum_{j=0}^{n-1} \int_{t_j}^{t_{j+1}} \frac{1}{t_{n+1}-t_{j+1}}|(t_{n+1}/s)^{\beta}-1|^2 \|g(s)\|_{\gamma(H,Y)}^2 ds\Big)^{p/2}.
\end{align*}
It is elementary to check that for any $t\in (t_{n+1}, t_{n+2})$ and $s\in (t_j, t_{j+1})$,
\begin{equation}\label{ineq:elementary estimate for weighted extrapolation}
    \frac{1}{t_{n+1}-t_{j+1}}|(t_{n+1}/s)^{\beta}-1|^2\leq \frac{3}{t-s}|(t/s)^{\beta}-1|^2.
\end{equation}
Therefore, since $(S_{K, \beta}g)_0 = 0$,
\begin{align*}
\sum_{n\geq 0} \tau \E\|(S_{K, \beta}g)_n\|^p_Z &= \sum_{n\geq 1} \int_{t_{n+1}}^{t_{n+2}} \E\|(S_{K, \beta}g)_n\|^p_Z dt
\\ & \leq 3^{p/2} (C_{p,Z})^p M^p \int_{0}^\infty \E \Big(\int_{0}^{t} \frac{1}{t-s}|(t/s)^{\beta}-1|^2 \|g(s)\|_{\gamma(H,Y)}^2 ds\Big)^{p/2} dt
\\ & \leq 3^{p/2} (C_{p,Z})^p M^p C_{p,\beta}^p \|g\|_{L^p(\Omega ; L^p (0,T;\gamma(H,Y)))}^p,
 \end{align*}
where in the last step we used an estimate of the proof of \cite[Lemma 7.11]{AVstab}.
\end{proof}
By the above lemma, the proof of the proposition follows from the identity
\[(S_K g)_n = t_{n+1}^{\beta} (S_K g_{-\beta})_n - (S_{K,\beta} g)_n,\]
where $\beta = \alpha/p$, $g_{-\beta}(s) = s^{-\beta} g(s)$. Since the argument is almost identical to \cite[Theorem 7.10]{AVstab}, the details are left to the reader.
\qed 

\begin{proof}[Proof of Theorem \ref{thm:weighted}]
\eqref{thm:weighted item 2} $\Leftrightarrow$ \eqref{thm:weighted item 4}: It remains to observe that by \eqref{Ineq:Fractional powers of rational} and an interpolation argument, the kernel $K(n,j)\coloneq  A R_{\tau}^{n-j}$ satisfies $\|K(n,j)\|_{\calL(X_{1/2}, X_0)}\leq \frac{M}{(\tau(n-j))^{1/2}}$ for every $n>j \ge 0$. 

\eqref{thm:weighted item 1} $\Leftrightarrow$ \eqref{thm:weighted item 3}: This can be proved as in \cite[Theorem 7.9]{AVstab} by shifting the space regularity by $1/2$. 

\eqref{thm:weighted item 1} $\Leftrightarrow$ \eqref{thm:weighted item 2}: This is the content of Theorem \ref{thm:mainequiv}. 
\end{proof}

\begin{remark}[Quasi-uniform partitions]
    Theorem \ref{thm:weighted} extends to quasi-uniform partitions $\pi$ given by variable time-steps $(\tau_n)_{n=1}^N$ satisfying
    $	 \mu^{-1} \le \frac{\tau_{\max}}{\tau_{\min}} \le \mu $. In this case the weighted norm is given by $\|y\|^p_{\ell^p_{\pi,w_{\alpha}}(X_0)} =\sum_{n\ge0} \tau_{n+1} w_{\alpha}(t_{n+1}) \|y_n\|_{X_0}^p$.  Indeed, the unweighted equivalence \eqref{thm:weighted item 1} $\Leftrightarrow$ \eqref{thm:weighted item 2} is established in Remark \ref{remark:quasi uniform 1}, whereas the equivalence \eqref{thm:weighted item 1} $\Leftrightarrow$ \eqref{thm:weighted item 3} follows from the continuous-time setting (\cite[Theorem 7.9]{AVstab}). In order to prove the equivalence \eqref{thm:weighted item 2} $\Leftrightarrow$ \eqref{thm:weighted item 4} it suffices to observe that the kernel estimate $\|AR_\pi^{n,j}\|_{\calL(X_{1/2}, X_0)}\leq M(t_n-t_j)^{-1/2}$ holds (see Remark \ref{remark:quasi uniform 1}) and that the proofs of Proposition \ref{prop:generalweightkernel} and Lemma \ref{lem:weighted} extend mutatis mutandis to the quasi-uniform setting. The only significant modification is  that the elementary estimate \eqref{ineq:elementary estimate for weighted extrapolation} of Lemma \ref{lem:weighted} requires the constant $C_\mu = 1+2\mu$ instead of $3$.  
\end{remark}

\section{$\mathcal{R}$-boundedness of discrete stochastic convolutions}\label{sec:Rbdd}
One of the ingredients in the proof of the maximal estimate in Theorem \ref{thm:maxest} is an $\mathcal{R}$-boundedness result for discrete stochastic convolutions. Here, the prefix $\mathcal{R}$ does not refer to the scheme, and therefore we have chosen to use a calligraphic letter instead. $\mathcal{R}$-boundedness plays a crucial role in vector-valued harmonic and stochastic analysis, and the letter $\mathcal{R}$ refers to Rademacher or random. For an overview of $\mathcal{R}$-boundedness and its role in analysis, the reader is referred to \cite[Chapter 8]{hytonen2017analysis}.

\subsection{Definitions}\label{ss:defRbdd}

\newcommand{\cT}{\mathcal{T}}

Let $(r_n)_{n\ge 1}$ be a Rademacher sequence, i.e. $\P(r_n = +1) = \P(r_n=-1) = 1/2$, and the random variables $(r_n)_{n\ge 1}$ are independent.
Let $Y$ and $Z$ be Banach spaces. A family of operators $\cT\subseteq\calL(Y,Z)$ is said to be {\em $\mathcal{R}$-bounded} if there exists a constant $C\ge 0$ such that for all finite sequences $(T_n)_{n=1}^N$ in $\cT$ and $(y_n)_{n=1}^N$ in $Y$,
\begin{equation*}
\Big\|\sum_{n=1}^N r_n T_n y_n\Big\|_{L^2(\Omega;Z)} \le C\Big\|\sum_{n=1}^N r_n y_n\Big\|_{L^2(\Omega;Y)}.
\end{equation*}
The least admissible constant in the above estimate is called the {\em $\mathcal{R}$-bound} of $\cT$
and is denoted by $\mathcal{R}(\cT)$.

\subsection{Main $\mathcal{R}$-boundedness result}\label{ss:mainRbdd}
In the continuous-time setting $\mathcal{R}$-boundedness of stochastic convolutions was obtained in \cite{NVWSMR, NVW11} in order to establish stochastic maximal $L^p$-regularity. These methods were extended to a discrete setting for implicit Euler in \cite{li-Xie-stability}, to establish discrete stochastic maximal regularity.

Unless stated otherwise, in the rest of this section $X_0$ is assumed to be isomorphic to a closed subspace of $L^q(\mathcal{O})$ with $(\mathcal{O}, \Sigma,\mu)$ a $\sigma$-finite measure space. Given a stepsize $\tau >0$, $\mathcal K_{\tau}$ denotes the set of all sequences $k = (k_n)_{n\ge 1}$ such that $k_n \to 0$ and
 $$\sum_{n \ge 1} \sqrt {n \tau } \, |k_{n+1}-k_n| \le 1.$$
 For $k =(k_n)_{n\ge1} \in \mathcal K_{\tau}$ and an elementary adapted process $g\colon [0,\infty) \times \Omega\to L^q(\mathcal O;H)$ we define the process $I^{\tau}(k)g \colon \N \times \Omega \to L^q(\mathcal O)$ via
 $$(I^{\tau}(k)g)_0 \coloneq 0 \quad \text{and } \quad (I^{\tau}(k)g)_n \coloneq \sum_{j=0}^{n-1} k_{n-j} \Delta_j I_g, \quad n \ge 1, $$
where we recall that $I_g(t) = \int_0^t g d W$ and $\Delta_j I_g = I_g(t_{j+1}) - I_g(t_{j})$.
Finally, let $\mathcal{I}^{\tau} \coloneq \{I^{\tau}(k)\colon  k\in \mathcal{K}_{\tau}\}$.

Recall that the notation for the function spaces with weights $w_{\alpha}(t) = t^{\alpha}$, was introduced in Subsection \ref{ss:weights}.
 
 \begin{theorem} \label{Thm: J is uniformly R-bdd}
 	Let $q\in [2, \infty)$. Let $p\in (2, \infty)$ and $\alpha\in (-1,\frac{p}{2}-1)$. In case $q=2$, we additionally allow $p=2$ and $\alpha=0$. Then there exists a constant $C_{p,q,\alpha}$ such that for all $\tau>0$ the family $$\mathcal I^{\tau} \subseteq \mathcal L( L^p_{\mathbb F}(\Omega;L^p(\R_+,w_{\alpha};L^q(\O;H))), L^p(\Omega; \ell^p_{\tau,w_{\alpha}}(L^q(\O))) )$$
 is $\mathcal{R}$-bounded by $C_{p,q,\alpha}$.
 \end{theorem}

\begin{proof}
Let $g$ be elementary adapted. For $k = (k_n) \in \mathcal K_{\tau}$ writing
	$k_n = - \sum_{m=n}^\infty (k_{m+1}-k_m)$ gives
	\begin{align*}
		(I^{\tau}(k) g)_n&= - \sum_{j=0}^\infty \sum_{m=1}^\infty (k_{m+1} -k_m) \1_{j \le n-1} \1_{n-j \le m } \Delta_j I_g
		\\
		& = -\sum_{m=1}^\infty \sqrt m (k_{m+1}-k_m) \sum_{j=0}^\infty \frac{1}{\sqrt m} \1_{j \le n-1} \1_{j \ge n-m} \Delta_j I_g
		\\
		&= -\sum_{m=1}^\infty \sqrt {m\tau} (k_{m+1}-k_m) \, (J^{\tau}(m)g)_n,
	\end{align*}
	where the process $J^{\tau}(m)g \colon \N \times \Omega \to L^q(\O)$ is given by $(J^{\tau}(m)g)_0\coloneq 0$ and for $n \ge1$,
	$$ (J^{\tau}(m)g)_n \coloneq \frac{1}{\sqrt {m\tau }} \sum_{j=0}^\infty \1_{j \le n-1} \1_{j \ge n-m} \Delta_j I_g = \sum_{j=0}^{n-1} k_{n-j}^{(m)} \Delta_j I_g,$$
	where $k_n^{(m)} \coloneq \frac{1}{\sqrt{m\tau}} \1_{1\le n \le m}$.

It follows from the above that $\mathcal{I}^{\tau}$ is contained in the closure of the absolute convex hull of $\mathcal{J}^{\tau}\coloneq  \{J^{\tau}(m)\colon  m \ge 1\}$. Therefore, by \cite[Proposition 8.1.21 and Theorem 8.1.22]{hytonen2017analysis} we see that it suffices to prove the $\mathcal{R}$-boundedness of $\mathcal{J}^{\tau}$.

Let $N \geq 1$, $a_1, \dots, a_N \in \N$ and $g^1, \dots, g^N \in L^p_{\mathbb F}(\Omega;L^p(\R_+,w_{\alpha};L^q(\mathcal{O}; H)))$ be arbitrary and fixed. Note that the sequence $f^n(s) \coloneq \|g^n(s)\|_{H}^2$ belongs to $L^{p/2}_{\mathbb F}(\Omega;L^{p/2}(\R_+,w_{\alpha};L^{q/2}(\mathcal{O})))$.
	
Let $(r_n)_{n=1}^N$ be a Rademacher sequence on a probability space $(\Omega_r, \F_r, \mathbb{P}_r)$. By Proposition \ref{prop:BDGUMD} applied pointwise with respect to $(\omega, m) \in \Omega_r \times \mathbb{N}$ in (i), and the Kahane-Khintchine inequalities (see \cite{hytonen2017analysis}) in (ii) and (iii), we may estimate as follows
		\begin{align*}
	 &
	\mathbb{E}_r \Big\| \sum_{n=1}^N r_n J(a_n) g^n \Big\|_{L^p(\Omega; \ell^p_{\tau,w_{\alpha}}(L^q(\mathcal{O})))}^p
	\\	
	& = \mathbb{E}_r \sum_{m \ge 1} \tau t_{m+1}^{\alpha}\,\E \Big\| \int_0^\infty \sum_{j=0}^\infty \sum_{n=1}^N \frac{r_n}{\sqrt{a_n\tau}} \1_{ [0 \vee (m - a_n),m-1]  }(j) g^n(s) \1_{[t_j,t_{j+1})}(s) \,  dW(s) \Big\|_{ L^q(\mathcal{O})}^p
	\\
 	&
	\overset{(\text i)}{\eqsim}_{p, q} \mathbb{E}_r  \sum_{m\geq 1} \tau t_{m+1}^{\alpha} \mathbb{E}  \Big\| s \mapsto  \sum_{j,n} \frac{r_n}{\sqrt{a_n \tau }} \1_{ [0 \vee (m - a_n),m-1]  }(j) g^n(s) \1_{[t_j,t_{j+1})}(s) \Big \|_{L^q(\O;L^2(\R_+;H))}^p
	\\
	&
	\overset{( \text{ii} )}{\eqsim}_{p,q} \mathbb{E} \sum_{m\geq 1} \tau t_{m+1}^{\alpha} \Big( \mathbb{E}_r \Big\| s \mapsto \sum_{j,n} \frac{r_n}{\sqrt{a_n \tau }} \1_{ [0 \vee (m - a_n),m-1]  }(j) g^n(s) \1_{[t_j,t_{j+1})}(s) \Big\|_{L^q(\mathcal{O}; L^2(\R_+; H))}^q \Big)^{p/q}
	\\
	&
	= \mathbb{E} \sum_{m\geq 1}\tau t_{m+1}^{\alpha}\, \Big( \int_{\O} \E_r \Big\| s \mapsto \sum_{i,n} \frac{r_n}{\sqrt{a_n \tau }} \1_{ [0 \vee (m - a_n),m-1]  }(j) g^n(s) \1_{[t_j,t_{j+1})}(s) \Big\|_{L^2(\R_+; H)}^q d\mu \Big)^{p/q}
	\\
	&
	\overset{(\text{iii})}{\eqsim}_q \mathbb{E} \sum_{m\geq 1}\tau t_{m+1}^{\alpha}\, \Big( \int_{\O} \Big ( \E_r \Big\| s \mapsto \sum_{j,n} \frac{r_n}{\sqrt{a_n \tau }} \1_{ [0 \vee (m - a_n),m-1]  }(j) g^n(s) \1_{[t_j,t_{j+1})}(s) \Big\|_{L^2(\R_+;H)}^2 \Big)^{	q/2} d\mu  \Big)^{p/q}
	\\
& =  \mathbb{E} \sum_{m\geq 1}\tau t_{m+1}^{\alpha}\, \Big( \int_{\O} \Big ( \sum_{j,n} \int_{t_j}^{t_{j+1}} \frac{1}{a_n \tau} \1_{ [0 \vee (m - a_n),m-1]  }(j) \|g^n(s)\|_H^2 ds \Big)^{	q/2} d\mu  \Big)^{p/q}
\\	&	\eqsim_{\alpha} \mathbb{E} \sum_{m\geq 1} \int_{t_m}^{t_{m+1} }\, \Big( \int_{\O} \Big ( \sum_{n=1}^N T^*(a_n \tau) f^n (t_m)  \Big)^{	q/2} d\mu  \Big)^{p/q} t^{\alpha} dt
\\ & \leq 2^{p/2} \mathbb{E} \sum_{m\geq 1} \int_{t_m}^{t_{m+1} }\, \Big( \int_{\O} \Big ( \sum_{n=1}^N T^*((a_n+1) \tau) f^n (t)  \Big)^{	q/2} d\mu  \Big)^{p/q} t^{\alpha} dt
\\ &
	= 2^{p/2}\mathbb{E} \Big\| \sum_{n=1}^N T^*((a_n+1) \tau) f^n \Big\|_{L^{p/2}(\R_+,w_{\alpha};L^{q/2}(\mathcal{O}))}^{p/2},
	\end{align*}
where the operator $T^*(\delta)$ on $L^{p/2}(\R_+,w_{\alpha};L^{q/2}(\mathcal{O}))$ is the adjoint of the one defined in Lemma \ref{lem:MaaNee} below and is given by 
\[
T^*(\delta) f (t) \coloneq  \frac{1}{\delta}\int_{(t-\delta)\vee 0}^t f(s) ds.
\]
Moreover, we used the simple estimate $T^*(a_n \tau) f (t_m)\leq 2 T^*((a_n+1) \tau) f (t)$ for all $t\in [t_m, t_{m+1})$.

Note that $w_{\alpha}\in A_{p/2}$ (see \cite[Example 7.1.7]{Grafclassical}). By Lemma \ref{lem:MaaNee} with $r'=p/2$ and $s'=q/2$ we obtain
\begin{align*}
\Big\| \sum_{n=1}^N T^*((a_n+1) \tau) f^n \Big\|_{L^{p/2}(\R_+,w_{\alpha};L^{q/2}(\mathcal{O}))}^{p/2} & \lesssim_{p,q} \Big\| \sum_{n=1}^N f^n \Big\|_{L^{p/2}(\R_+,w_{\alpha};L^{q/2}(\mathcal{O}))}^{p/2}
\\ & \eqsim \E_r\Big\|\sum_{n=1}^N r_n g^n\Big\|_{L^p(\R_+,w_{\alpha}; L^q(\mathcal{O}; H))}^p,
\end{align*}
where the last step follows from reversing the computations involving the Kahane–Khintchine inequalities. Combining the estimates, gives the required $\mathcal{R}$-boundedness.
\end{proof}

In the previous proof we used a result of \cite[Section 3]{NVWSMR} for the Muckenhoupt weighted setting. For details on such weights, the reader is referred to \cite[Chapter 7]{Grafclassical}.
\begin{lemma} \label{lem:MaaNee}
For $\delta>0$ let $T(\delta)$ be the operator on $L^r(\R;L^s(\mathcal{O}))$ given by
\[T(\delta)f(t)\coloneq \frac{1}{\delta}\int_{t}^{t+\delta} f(\xi) d\xi,\]
where $r\in (1,\infty]$ and $s\in (1, \infty)$ (where $s=\infty$ is also allowed if $r=\infty$).
Let $\frac1r+\frac1{r'} = 1$, $\frac1{s}+\frac1{s'}=1$. Let $v$ be an $A_{r'}$-weight (with $v=1$ if $r=\infty$).
Then there is a constant $C$ only depending on $r, s$ and $[v]_{A_{r'}}$ such that for all $N\geq 1$, $f_1,
\ldots, f_N\in L^{r'}(\R,v;L^{s'}(\mathcal{O}))$ and $\delta_1,\ldots, \delta_N>0$,
\begin{align*}
\Big\| \sum_{n=1}^N T^*(\delta_n) f_n \Big\|_{L^{r'}(\R,v;L^{s'}(\mathcal{O}))} \leq C \Big\|
\sum_{n=1}^N f_n \Big\|_{L^{r'}(\R,v;L^{s'}(\mathcal{O}))}.
\end{align*}
\end{lemma}
To deduce the result from \cite[Lemma 3.3]{NVWSMR} it suffices to note that the Hardy--Littlewood maximal operator $M$ is bounded on $L^{r}(\R,v;\ell^{s})$ for all $v\in A_{r}$ (see \cite[Theorem 7.1.9 and Corollary 7.5.7]{Grafclassical}), where $s=\infty$ is allowed if $r=\infty$. These are the weighted versions of the Fefferman-Stein maximal estimates.

A related class of operators is also $\mathcal{R}$-bounded for similar reasons.
\begin{remark}\label{Temp: Rbdd for i>n}
	 For $k =(k_n)_{n\ge1} \in \mathcal K_{\tau}$ and elementary adapted process $g\colon \R_+\times \Omega\to L^q(\mathcal O;H)$ define the process $\tilde I^{\tau}(k)g \colon \N \times \Omega \to L^q(\mathcal O)$ via
	$$(\tilde I^{\tau}(k)g)_0 \coloneq 0 \quad \text{and } \quad (\tilde I^{\tau}(k)g)_n \coloneq \sum_{j\ge n} k_{j-n} \Delta_j I_g \quad n \ge 1 . $$
Then one can show that $\wt{\mathcal{I}}^{\tau} = \{\tilde I^{\tau}(k)\colon  k\in \mathcal K_{\tau}\}$ is $\mathcal{R}$-bounded by a constant only depending on $p$ and $q$. Indeed, the argument can be done in a similar way as we have seen in Theorem \ref{Thm: J is uniformly R-bdd}. One difference is that one needs to reverse the roles of $T$ and $T^*$.
\end{remark}

\begin{remark}\label{rem:otherBS}
The assertions of Theorem \ref{Thm: J is uniformly R-bdd} and Remark \ref{Temp: Rbdd for i>n} remain valid if $X_0$ is isomorphic to a closed subspace of $L^q(\mathcal{O})$ with $(\mathcal{O}, \Sigma,\mu)$ a $\sigma$-finite measure space. Here one needs to replace $L^q(\mathcal{O};H)$ by $\gamma(H,X_0)$. Indeed, after applying the isomorphism, one can reduce to $L^q(\mathcal{O})$.

The technique to prove Theorem \ref{Thm: J is uniformly R-bdd} originates from \cite{NVWSMR}, but was further analyzed in \cite{NVW11}. Combining \cite[Theorem 4.7]{NVWSMR} with the techniques in Theorem \ref{Thm: J is uniformly R-bdd} and Remark \ref{Temp: Rbdd for i>n} one can see that our results remain valid if $X_0$ is isomorphic to a Banach function space such that the $2$-concavification is a UMD space. In this way, the result follows for spaces such as $L^{q}(L^{r})$, Besov spaces $B^{s}_{q,r}$, and Triebel-Lizorkin spaces $F^{s}_{q,r}$.
\end{remark}

\subsection{Kernels for exponential schemes}\label{ss:kernelexp}
In this subsection, we  check that $k\in \mathcal{K}_\tau$ (as defined in Subsection \ref{ss:mainRbdd}) for sequences $k$ which will be needed in Section \ref{sec:discretemax}. Although the proofs are completely elementary, they are quite tedious, and we give the detailed arguments for the convenience of the reader.
The aim is to prove $\frac{k}{C}\in \mathcal{K}_\tau$, where $C>0$ is independent of $\tau$. In all examples, it will be obvious that $k_n\to 0$ as $n\to \infty$ and thus we only need to check that
$$\sum_{n \ge 1} \sqrt {n \tau } \, |k_{n+1}-k_n| \le C.$$
This shows that $I^{\tau}(k)\in \mathcal{I}^\tau$ and $\wt{I}^{\tau}(k)\in \wt{\mathcal{I}}^\tau$,
and both families are $\mathcal{R}$-bounded by Theorem \ref{Thm: J is uniformly R-bdd} and Remark \ref{Temp: Rbdd for i>n}.

In the proofs below, we use the following elementary facts:
\begin{lemma}
Let $\nu\in (0,\pi/2)$. Then,
\begin{align}\label{eq:meanvalue1exp}
|1- e^{-z}|& \leq |z|, \quad \text{for all $z\in \C$ with $\Re(z) \geq 0$},
\\ \label{eq:pullinginabs} |1-e^{-z}|& \leq M_{\nu}(1-e^{-|z|\cos(\arg{z})}), \quad  \text{for all $z\in \Sigma_{\nu}$.}
\end{align}
\end{lemma}
\begin{proof}
For real $z$, \eqref{eq:meanvalue1exp} can be proved by comparing derivatives. The complex case follows from the real case since
$|1- e^{-z}| \leq |z|\int_0^1 |e^{-tz}| dt = |z|\int_0^1 e^{-t \Re(z)} dt = \frac{|z|}{\Re(z)} |1- e^{-\Re(z)}|\leq |z|$.

At the same time, this gives \eqref{eq:pullinginabs}. Indeed, observe that $\Re(z) = |z| \cos(\arg(z))$ and that $\frac{|z|}{\Re(z)} = \frac{1}{\cos(\arg(z))}$. Thus, the result follows with $M_{\nu} = \frac{1}{\cos(\nu)}$.
\end{proof}

\begin{lemma}\label{lem:exp}
Let $\nu\in (0,\pi/2)$ and $\sigma\in (0,1/2)$. Let $k_n, \phi\colon (0,\infty)\times\Sigma_{\nu}\to \C$ be given by
\begin{align*}
k_n(\tau,\lambda) & = \lambda^{1/2} e^{-n\tau\lambda}, \ \ n\geq 1,
\\ \phi(\tau, \lambda) &= \sum_{j\geq 1} \frac{e^{-j\tau\lambda} - 1}{(\tau\lambda)^{\sigma}} a_j,
\end{align*}
where $|a_j|\leq b/j^{1+\sigma}$ for all $j\geq 1$, and $b\geq 0$ is a constant.
Then there are constants $C_{\nu}$ and $C_{\nu, \sigma}$ such that for all $\tau>0$ and $\lambda\in \Sigma_{\nu}$,
\[\text{both } \ \frac{k(\tau,\lambda)}{b C_{\nu}} \ \ \text{and} \ \ \frac{k(\tau, \lambda) \phi(\tau, \lambda)}{b C_{\nu,\sigma}} \ \ \text{define sequences in $\mathcal{K}_\tau$}.\]
\end{lemma}
\begin{proof}

  Setting $c=\cos \nu$, we see that $|e^{-n\tau \la}| \le e^{-c n\tau |\la|}$. Using this and \eqref{eq:meanvalue1exp} we find that
  \begin{align*}
  |k_{n+1}(\tau, \lambda)-k_n(\tau, \lambda)| = |\lambda|^{1/2} |e^{-n\tau\lambda}| |e^{-\tau\lambda} - 1| \leq |\lambda|^{1/2} e^{-n\tau c|\lambda|} |\tau\lambda|.
  \end{align*}
  It follows that
  \begin{align*}
	\sum_{n \ge 1} \sqrt {n \tau} \, |k_{n+1}(\tau, \lambda)-k_n(\tau, \lambda)| & \le \sum_{n\ge1 } \tau |\la|^{\frac 32} \sqrt{n\tau} e^{-cn\tau |\la|}
\\ & = \sum_{n\ge 1} \int_{t_{n}}^{t_{n+1}} |\la|^{\frac 32} t_n^{\frac 12} e^{-ct_n |\la|} dt
 \\ & \le \int_0^\infty |\la|^{\frac 32} t^{\frac 12} e^{-ct |\la|/2} dt = c^{-3/2}\int_0^\infty s^{\frac 12} e^{-s/2} ds,
	\end{align*}
which is a finite number only depending on $\nu$.

To prove the bound for $k(\tau,\lambda)\phi(\tau,\lambda)$, it remains to prove that $\phi$ is uniformly bounded. To do so, note that by \eqref{eq:pullinginabs}
\begin{align}\label{eq:estphiuniform}
\begin{aligned}
|\phi(\tau, \lambda)|& \leq M_{\nu} b\sum_{j\geq 1} \frac{1-e^{-j |\tau \lambda| c}}{|\tau\lambda|^{\sigma} j^{1+\sigma}}
 \\ & = M_{\nu} b\sum_{j\geq 1} \int_{t_j}^{t_{j+1}} \frac{1-e^{-j |\tau \lambda|c}}{|\tau\lambda|^{\sigma} j^{1+\sigma} \tau} dt
 \\ & \leq M_{\nu} b\int_{0}^{\infty} \frac{1-e^{- t |\lambda|c }}{|\lambda|^{\sigma} (t/2)^{1+\sigma}} dt
 = M_{\nu}b 2^{1+\sigma} c^{\sigma} \int_{0}^{\infty} \frac{1-e^{- t}}{t^{1+\sigma}} dt,
\end{aligned}
\end{align}
which is a finite number only depending on $\nu$ and $\sigma$.
\end{proof}

We will also need the following version.
\begin{lemma}\label{lem:expvariant}
Let $\nu\in (0,\pi/2)$ and $\sigma\in (0,1/2)$. Let $\wt{k}_n\colon (0,\infty)\times\Sigma_{\nu}\to \C$ be given by
\begin{align*}
 \wt{k}_n(\tau, \lambda) &= \sum_{j\geq n+1}  \frac{\lambda^{1/2}}{(\tau\lambda)^{\sigma}} e^{-(j-n)\tau\lambda} a_j, \ \ n\geq 1,
\end{align*}
where $(a_j)_{j\geq 1}$ is a sequence in $\C$ that satisfies $|a_j-a_{j+1}|\leq \frac{b}{j^{2+\sigma}}$, where $b\geq 0$ is a constant.
Then there is a constant $C_{\nu,\sigma}$ such that for all $\tau>0$ and $\lambda\in \Sigma_{\nu}$,
\[\frac{\wt{k}(\tau, \lambda)}{b C_{\nu,\sigma}} \ \ \text{defines a sequence in $\mathcal{K}_\tau$}.\]
\end{lemma}
\begin{proof}
A rewriting gives
 \begin{align*}
 |\wt{k}_{n}(\tau, \lambda) &- \wt{k}_{n+1}(\tau, \lambda) |
 \\ & \qquad = \frac{|\lambda|^{1/2}}{|\tau\lambda|^{\sigma}} \Big|\sum_{j\geq n+1} e^{-(j-n)\tau\lambda} (a_{j} - a_{j+1})\Big|
 \\ & \qquad\leq \frac{|\lambda|^{1/2}}{|\tau\lambda|^{\sigma}} \sum_{j\geq n+1} e^{-(j-n)\tau|\lambda| c } \frac{b}{j^{2+\sigma}}
\\ & \qquad = \frac{b|\lambda|^{1/2}}{|\tau\lambda|^{\sigma}} e^{-\tau|\lambda| c } \frac{1}{(n+1)^{2+\sigma}} + b|\lambda|^{1/2-\sigma} \tau \sum_{j\geq n+2} \int_{t_{j-1}}^{t_{j}} e^{-(t_j-t_n)|\lambda| c } \frac{1}{t_j^{2+\sigma}} ds
\\ & \qquad \leq \frac{b|\lambda|^{1/2}}{|\tau\lambda|^{\sigma}} e^{-\tau|\lambda| c } \frac{1}{(n+1)^{2+\sigma}} + b|\lambda|^{1/2-\sigma} \tau \int_{t_{n+1}}^{\infty} e^{-(s-t_n)|\lambda| c } \frac{1}{s^{2+\sigma}} ds \eqqcolon  A_n+B_n,
 \end{align*}
 where we took out the $j=n+1$ term because later it is helpful to start the integral at $t_{n+1}$ instead of $t_n$. Next, multiplying by $\sqrt{n\tau}$ and summing over $n\geq1$ the $A_n$-term becomes
 \begin{align*}
   \sum_{n\geq 1} \sqrt{n\tau} A_n \leq b|\tau \lambda|^{1/2-\sigma} e^{-\tau|\lambda| c } \sum_{n\geq 1} \frac{1}{(n+1)^{\frac32+\sigma}} \leq b\sup_{r>0} r^{1/2-\sigma} e^{-r c } \sum_{n\geq 1} \frac{1}{(n+1)^{\frac32+\sigma}},
 \end{align*}
 which is a number only depending on $\nu$ and $\sigma$. The $B_n$-term can be estimated as
 \begin{align*}
 \sum_{n\geq 1} \sqrt{n\tau} B_n & \leq b|\lambda|^{1/2-\sigma} \sum_{n\geq 1} t_n^{1/2}  \tau \int_{t_{n+1}}^{\infty} e^{-(s-t_n)|\lambda| c } \frac{1}{s^{2+\sigma}} ds
 \\ & = b|\lambda|^{1/2-\sigma} \sum_{n\geq 1} \int_{t_{n}}^{t_{n+1}} \int_{t_{n+1}}^{\infty} t_n^{1/2} e^{-(s-t_n)|\lambda| c } \frac{1}{s^{2+\sigma}} ds dt
 \\ & \leq b|\lambda|^{1/2-\sigma} \int_{0}^{\infty} \int_{t}^{\infty} t^{1/2} e^{-(s-t)|\lambda| c } \frac{1}{s^{2+\sigma}} ds dt
 \\ & = b|\lambda|^{1/2-\sigma} \int_{0}^{\infty} \int_{1}^{\infty} e^{-t(u-1)|\lambda| c } \frac{1}{t^{1/2+\sigma} u^{2+\sigma}} du dt
 \\ & = b|\lambda|^{1/2-\sigma} \int_{0}^{\infty} \int_{0}^{\infty} e^{-t u|\lambda| c } \frac{1}{t^{1/2+\sigma} (u+1)^{2+\sigma}} dt du
 \\ & = b c^{\sigma-\frac12} \int_{0}^{\infty} \frac{e^{-v}}{v^{1/2+\sigma}} dv \int_{0}^{\infty} \frac{u^{\sigma-\frac12}}{(u+1)^{2+\sigma}} du,
 \end{align*}
 which is a constant depending only on $\nu$ and $\sigma$.
\end{proof}

\subsection{Kernels for rational schemes}
Next, we check $k\in \mathcal{K}_\tau$ for sequences associated to rational schemes. The proofs are similar to those of the exponential function, and sometimes we can rely on the latter in proving estimates.

The following assumption will be in place:
\begin{assumption}\label{assum:rrational}
Let $\theta\in (0,\pi/2]$. Let $r\colon \Sigma_{\theta}\to \C$ be a rational function such that $r$ is consistent of order $\ell\ge1$, $A(\theta)$-stable, i.e.\ $|r(z)|\leq 1$ for
$z\in \Sigma_{\theta}$, and $r(\infty) = 0$.
\end{assumption}

Recall that this assumption is precisely what is needed for Lemma \ref{lem:rational-est}.

\begin{lemma}\label{lem:rational}
Suppose that Assumption \ref{assum:rrational} holds and let $\nu\in (0,\theta)$. Let $\sigma\in (0,1/2)$.
Let $k_n^r, \phi^r\colon (0,\infty)\times\Sigma_{\nu}\to \C$ be defined by
\begin{align*}
k_n^r(\tau,\la)& =\la^{\frac 12} r(\tau \la)^n, \ \ n\geq 1,
\\ \phi^r(\tau, \lambda) &= \sum_{j \ge 1} \frac{r(\tau \la)^j-1}{(\tau \lambda)^{\sigma}} a_j,
\end{align*}
where $|a_j|\leq b/j^{1+\sigma}$ for all $j\geq 1$, and $b\geq 0$ is a constant.
Then there are constants $C_{\nu,\theta}$ and $C_{\nu,\theta,\sigma}$ such that for all $\tau>0$ and $\lambda\in \Sigma_{\nu}$,
\[\text{both} \ \ \frac{k_n^r(\tau,\lambda)}{b C_{\nu,\theta}} \ \ \text{and} \ \ \frac{k_n^r(\tau,\lambda) \phi^r(\tau, \lambda)}{b C_{\nu,\theta,\sigma}} \ \ \text{define sequences in $\mathcal{K}_\tau$}.\]
\end{lemma}
\begin{proof}
Let $z=\tau \la$. Then
\[\sum_{n\geq 1}\sqrt {n\tau} |k_{n+1}^r(\tau, \lambda)-k_n^r(\tau, \lambda)| = |z|^{\frac 12} |r(z)-1| \sum_{n\geq 1} \sqrt{n} |r(z)|^n.\]
We bound the latter uniformly in $z$. First, consider $z\in \Sigma_{\theta}$ with $|z|\geq 1$. Then
by \eqref{Ineq:Difference of exp and rational z>1}
	\begin{align*}
|r(z)-1| \, |z|^{\frac12} \sum_{n\ge1} \sqrt n \, |r(z)|^n
 \leq C_1 (C_1+1) \sum_{n\ge1} \sqrt n \, e^{-c_0n},
\end{align*}
which is a constant only depending on $\theta$ and $\nu$.

Next, consider $z\in \Sigma_{\theta}$ with $|z|<1$. Then by \eqref{Ineq:growth rational z<1} and \eqref{eq:meanvalue1exp}, $|r(z)|^n \leq C_1 e^{-c_0 n|z|}$, and
\[|r(z)-1|\leq |r(z)-e^{-z}| + |e^{-z}-1| \leq (C_1 + 1)|z|.\]
Therefore,
\begin{align*}
|r(z)-1| \, |z|^{\frac12} \sum_{n\ge1} \sqrt n \, |r(z)|^n
& \leq \frac{|r(z)-1|}{|z|} \sum_{n\ge1} \sqrt{n} |z|^{3/2} e^{-c_0n|z|}
\leq (C_1 + 1) \sum_{n\ge1} \sqrt{n} |z|^{3/2} e^{-c_0n|z|},
\end{align*}
which can be bounded as in the proof of Lemma \ref{lem:exp}.

To prove the same for $k^r(\tau,\lambda) \phi^r(\tau,\lambda)$, it is enough to bound $\phi^r$ uniformly. To do so note that
\begin{align*}
|\phi^r(\tau,\lambda)| \leq b \sum_{j \ge 1} \frac{|r(z)^j-1|}{|z|^{\sigma} j^{1+\sigma}}.
\end{align*}
For $z\in \Sigma_{\nu}$ and $|z|\geq 1$ using that $|r(z)|\leq 1$, we see that
$|\phi^r(\tau,\lambda)| \leq 2b\sum_{j \ge 1} \frac{1}{j^{1+\sigma}}$.

For $|z|<1$ note that by \eqref{Ineq:Difference of exp and rational z<1} and \eqref{eq:pullinginabs}
\[|r(z)^j-1|\leq |r(z)^j-e^{-jz}|+|e^{-jz}-1|\leq C_1 (j |z|^{1} e^{-c_0j|z|}+ 1- e^{-c j|z|}), \]
where $c = \cos(\nu)$. Therefore,
\begin{align*}
|\phi^r(\tau,\lambda)| \leq b C_1 \sum_{j \ge 1} \frac{|z|^{1 -\sigma} e^{-c_0j|z|}}{j^{\sigma}} +  b C_1 \sum_{j \ge 1} \frac{1-e^{-c j|z|}}{|z|^{\sigma} j^{1+\sigma}}.
\end{align*}
The first term on the right-hand side satisfies
\begin{align*}
b\sum_{j \ge 1} \frac{|z|^{1 -\sigma} e^{-c_0j|z|}}{j^{\sigma}} = b\sum_{j \ge 1} \int_{t_{j-1}}^{t_{j}} \frac{|\lambda| e^{-c_0 t_j |\lambda|}}{|\lambda|^{\sigma} t_j^{\sigma}} dt \leq b\int_{0}^{\infty} \frac{|\lambda| e^{-c_0 t |\lambda|}}{|\lambda|^{\sigma} t^{\sigma}} dt
 = b c^{\sigma-1}_0 \int_{0}^{\infty} \frac{e^{-t}}{t^{\sigma}} dt.
\end{align*}
The second term in the estimate for $\phi^r$ is uniformly bounded by \eqref{eq:estphiuniform}.
\end{proof}

Finally, we will need the following variant of Lemma \ref{lem:expvariant} as well.
\begin{lemma}\label{lem:rationalvariant}
Suppose that Assumption \ref{assum:rrational} holds and let $\nu\in (0,\theta)$. Let $\sigma\in (0,1/2)$. Let $\wt{k}_n^r\colon (0,\infty)\times\Sigma_{\nu}\to \C$ be given by
\begin{align*}
\wt{k}_n^r (\tau, \lambda) &= \sum_{j\geq n+1} \frac{\lambda^{1/2}}{(\tau\lambda)^{\sigma}} r(\tau\lambda)^{j-n} a_j, \ \ n\geq 1,
\end{align*}
where $(a_j)_{j\geq 1}$ in $\C$ satisfies $|a_j-a_{j+1}|\leq \frac{b}{j^{2+\sigma}}$, where $b\geq 0$ is a constant.
Then there is a constant $C_{\nu,\theta,\sigma}$ such that for all $\tau>0$ and $\lambda\in \Sigma_{\nu}$,
\[\frac{\wt{k}^r(\tau, \lambda)}{b C_{\nu,\theta,\sigma}} \ \ \text{defines a sequence in $\mathcal{K}_\tau$}.\]
\end{lemma}
\begin{proof}
As in the proof of Lemma \ref{lem:expvariant}, using $z = \tau\lambda$ we can write
\begin{align*}
 |\wt{k}_{n}^r(\tau, \lambda) - \wt{k}_{n+1}^r(\tau, \lambda) |
 &= \frac{|\lambda|^{1/2}}{|z|^{\sigma}} \Big|\sum_{j\geq n+1} r(z)^{j-n} (a_j - a_{j+1})\Big|
 \\ & \leq \frac{b|\lambda|^{1/2}}{|z|^{\sigma}} \sum_{j\geq n+1} \frac{|r(z)|^{j-n}}{{j^{2+\sigma}}} .
\\ & = \frac{b|\lambda|^{1/2}}{|z|^{\sigma}} |r(z)| \frac{1}{(n+1)^{2+\sigma}} + \frac{b|\lambda|^{1/2}}{|z|^{\sigma}} \sum_{j\geq n+2} \frac{|r(z)|^{j-n}}{{j^{2+\sigma}}}
 \eqqcolon A_n + B_n.
\end{align*}
The $A_n$ term multiplied by $\sqrt{n\tau}$ and summed over all $n\geq 1$ gives
\begin{align*}
\sum_{n\geq 1} \sqrt{n\tau} A_n = b\sum_{n\geq 1} \sqrt{n \tau} \frac{|\lambda|^{1/2}}{|z|^{\sigma}} \frac{|r(z)|}{{(n+1)^{2+\sigma}}} \leq b|z|^{1/2-\sigma} |r(z)| \sum_{n\geq 1} \frac{1}{{(n+1)^{\frac32+\sigma}}}.
\end{align*}
For $|z|\leq 1$, we have $|z|^{1/2-\sigma} |r(z)|\leq 1$. For $|z|>1$, by \eqref{Ineq:Difference of exp and rational z>1},
$|z|^{1/2-\sigma} |r(z)| \leq C_1$.

It remains to bound $\sum_{n\geq 1} \sqrt{n\tau} B_n$. Note that for $|z|\geq 1$ again by \eqref{Ineq:Difference of exp and rational z>1}
\begin{align*}
\sum_{n\geq 1} \sqrt{n\tau} B_n & \leq b C_1 \sum_{n\geq 1} \sqrt{n} \sum_{j\geq n+2} \frac{e^{-c_0(j-n)}}{{j^{2+\sigma}}},
\end{align*}
and the latter is finite as can be seen from the proof of Lemma \ref{lem:expvariant} by taking $\lambda = c_0/c$.

For $|z|<1$ by \eqref{Ineq:growth rational z<1}
$|r(z)|^{j-n} \leq b C_1 e^{-c_0(j-n)|z|}$ and therefore
\begin{align*}
\sum_{n\geq 1} \sqrt{n\tau} B_n&\leq b C_1 \sum_{n\geq 1} \sqrt n |z|^{1/2-\sigma} \sum_{j\geq n+2} \frac{e^{-c_0(j-n)|z|}}{{j^{2+\sigma}}},
\end{align*}
which can be bounded in the same way as in Lemma \ref{lem:expvariant}.
\end{proof}

\section{Discrete maximal estimates}\label{sec:discretemax}

In this section, we will consider a maximal estimate with parabolic regularization, which includes Theorem \ref{thm:maxest} as a special case. In this maximal estimate one needs to bound $\E\sup_{n\geq 1} \|Y_n\|^p_{Z}$ (for a suitable norm $\|\cdot\|_Z$). Since $Y$ does not have a martingale structure, one often cannot apply stochastic calculus techniques, and therefore it is unclear how to deal with the supremum over $n$ inside an expectation. Especially, if the norm $\|\cdot\|_Z$ is chosen to express optimal parabolic regularization, this leads to serious issues. In the deterministic setting there is no expectation, so that this difficulty is less prominent. Maximal estimates in the latter case are well-known and the reader is referred to Remark \ref{rem:MaxDet} for further details. 

\subsection{Maximal estimate with parabolic regularization}

The following result is the main result of this section and, in particular, includes Theorem \ref{thm:maxest}.

Recall from \eqref{Eq:Definition of approximation scheme} that $Y$ is given by the recursive formula $Y_0=0$ and
\begin{equation*}
		Y_{n+1} \coloneq R_\tau Y_{n} + R_\tau \Delta_n I_g, \quad n\geq 0,
\end{equation*}	
where $\Delta_n I_g$ is as defined below \eqref{Eq:Definition of approximation scheme} and $R_{\tau}$ is as in Assumption \ref{assum:mainDSMR}. Moreover, by \eqref{Eq: Uniform step discrete stochastic convolution}, 
$$
Y_n = \sum_{j=0}^{n-1} R_{\tau}^{n-j} \Delta_j I_g, \ \ n\geq 1.
$$

\begin{theorem}\label{thm:maximalest}
Let $q\in [2, \infty)$ and suppose that $X_0$ is isomorphic to a closed subspace of $L^q(\mathcal{O})$ with $(\mathcal{O}, \Sigma,\mu)$ a $\sigma$-finite measure space. Suppose that Assumption \ref{assum:mainDSMR} holds, that $A$ has a bounded $H^\infty$-calculus on $X_0$ of angle $<\pi/2$, and that $0\in \rho(A)$.
Then for any $p\in (2, \infty)$ and $\alpha\in [0,\frac{p}{2}-1)$, there is a constant $C$ such that for every $g\in L^{p}_{\mathbb F}(\Omega;L^p(\R_+,w_{\alpha};\gamma(H,X_{1/2})))$ and every stepsize $\tau>0$,
\begin{align*}
	\E \sup_{n\ge1} \| Y_n \|_{X_{1-\frac {1+\alpha}{p},p}}^p & \le C^p \E\|g\|^{p}_{L^p(\R_+,w_{\alpha};\gamma(H,X_{1/2}))},
\\ \E \sup_{n\ge1} (\tau n)^{\alpha} \| Y_n \|_{X_{1-\frac {1}{p},p}}^p &\le C^p \E\|g\|^{p}_{L^p(\R_+,w_{\alpha};\gamma(H,X_{1/2}))},
\end{align*}	
where $Y$ is given by \eqref{Eq:Definition of approximation scheme}.
\end{theorem}

Since $X_0$ is assumed to be isomorphic to a closed subspace of $L^q$, it follows that any of the complex interpolation spaces $X_{\beta}$ for $\beta\in [0,1]$, are isomorphic to a closed subspace of $L^q$. Indeed, since we assume that $A$ has a bounded $H^\infty$-calculus it follows that $X_{\beta} = D(A^{\beta})$ isomorphically (see Lemma \ref{lem:BIP}). Since $D(A^{\beta})$ is isomorphic to $X_0$ (use $(1+A)^{\beta}$ as an isomorphism), the result follows.

\subsection{Proof of Theorem \ref{thm:maximalest}}

\subsubsection{Reduction to a continuous-time setting}

It is clear that by density, it suffices to consider  $g\in L^p(\Omega;L^p(\R_+;\gamma(H,X_{1})))$ with support on a bounded subinterval of $[0,\infty)$. Let $u \in L^p(\R_+\times\Omega;X_1)$ be the linear interpolation of the discrete solution, i.e.
\begin{align}\label{eq:linearinterp}
  u(n\tau +s\tau)=(1-s)Y_n+sY_{n+1}, \ \ \ s \in [0,1], n \ge 0.
\end{align}
From Corollary \ref{cor:DSMRLq} one sees that $A Y\in L^p(\Omega;\ell^p_{\tau,w_{\alpha}}(X_0))$. Moreover, since $A$ is invertible we have $Y\in L^p(\Omega;\ell^p_{\tau,w_{\alpha}}(X_1))$ and $u\in L^p(\Omega;L^p(\R_+,w_{\alpha};X_1))$ with
\begin{align}\label{eq:estimateuLp}
\E\|u\|_{L^p(\R_+,w_{\alpha};X_1)}^p &\eqsim \E \sum_{n\geq 0} \int_{t_n}^{t_{n+1}} \|A u(t)\|^p_{X_0} t^{\alpha} dt
 \\ & \nonumber = \E \sum_{n\geq 0} \int_{0}^{1} \|A u(n\tau + s\tau)\|^p_{X_0} \tau (n\tau + s\tau)^{\alpha} ds
 \\ & \nonumber \leq \E \sum_{n\geq 0} \int_{0}^{1} [(1-s)\|A Y_n\|^p_{X_0}+s \|AY_{n+1}\|_{X_0}^p] \tau (n\tau + s\tau)^{\alpha} ds
 \\ & \nonumber\leq \E \sum_{n\geq 0} \tau t_{n+1}^{\alpha} \|A Y_n \|^p_{X_0}
 \\ & \nonumber \lesssim \E\|g\|_{L^p(\R_+,w_{\alpha};\gamma(H,X_{1/2}))}^p,
\end{align}
where we also use the convexity of $|\cdot|^p$.

Since $u$ is piecewise linear, one can even check that $u'$ exists in the weak sense and $u'\in L^p(\Omega;L^p(\R_+;X_1))$ and also $u'\in L^p(\Omega;L^p(\R_+,w_{\alpha};X_1))$ due to the extra regularity assumed on $g$. In particular, $u\in C_b(\R_+;X_1)$.

The proof of Theorem \ref{thm:maximalest} would be complete if we can prove the bound
\begin{align}\label{eq:supestimu}
\E \sup_{t\geq 0} \|u(t)\|_{X_{1-\frac {1+\alpha}{p},p}}^p + \E \sup_{t\geq 0} t^{\alpha} \|u(t)\|_{X_{1-\frac {1}{p},p}}^p\leq C^p \E\|g\|^{p}_{L^p(\R_+,w_{\alpha};\gamma(H,X_{1/2}))}.
\end{align}
Indeed, this is immediate from $u(t_n) = Y_n$. Now a crucial trick to avoid estimating the expectation of a supremum of a stochastic process is to use the following trace embedding of \cite[Theorem 1.2]{ALV21} (see also \cite[Theorem 1.1]{MV14}): for $\sigma\in (\frac12- \frac1p, \frac12)$
\begin{align*}
L^{p}(\R_+,w_{\alpha};X_{1}) \cap H^{\sigma,p}(\R_+,w_{\alpha};X_{1-\sigma})&\hookrightarrow C_b([0,\infty);X_{1-\frac {1+\alpha}{p},p}),
\\ L^{p}(\R_+,w_{\alpha};X_{1}) \cap H^{\sigma,p}(\R_+,w_{\alpha};X_{1-\sigma})&\hookrightarrow C_b((0,\infty),w_{\alpha};X_{1-\frac {1}{p},p}),
\end{align*}
where $C_b((0,\infty),w_{\alpha};E)$ are the continuous functions $f\colon (0,\infty)\to E$ for which $\|f\|_{C_b((0,\infty),w_{\alpha};E)} \coloneq  \sup_{t>0} w_{\alpha}(t) \|f(t)\|_E<\infty$.

Indeed, \eqref{eq:supestimu} follows as soon as we have proved the estimate
\begin{align}\label{eq:supestimu2}
\E \|u\|_{H^{\sigma,p}(\R_+;X_{1-\sigma})}^p \leq C^p_{\sigma} \E\|g\|^{p}_{L^p(\R_+;\gamma(H,X_{1/2}))}, \ \ \sigma\in [0,1/2).
\end{align}
The case $\sigma = 0$ follows from \eqref{eq:estimateuLp}. From now on let $\sigma\in (0,1/2)$ be arbitrary. To prove \eqref{eq:supestimu2}, due to the fact that $0\in \rho(A)$ it is enough to show that
\begin{align*}
\E \|A^{1-\sigma} u\|_{H^{\sigma,p}(\R_+,w_{\alpha};X_{0})}^p\leq C^p_{\sigma} \E\|g\|^{p}_{L^p(\R_+,w_{\alpha};\gamma(H,X_{1/2}))}, \ \ \ \sigma\in [0,1/2).
\end{align*}
It is well-known that $-\partial_t$ on $L^p(\R_+,w_{\alpha};X_0)$ with domain $W^{1,p}(\R_+,w_{\alpha};X_0)$ has a bounded $H^\infty$-calculus of angle $\pi/2$ and $D(\partial_t^{\sigma}) = H^{\sigma,p}(\R_+,w_{\alpha},X_0)$ (see \cite[Theorem 6.8]{LMV18}). In particular, it follows that
\[\|A^{1-\sigma} u\|_{H^{\sigma,p}(\R_+,w_{\alpha};X_{0})}\eqsim \|A^{1-\sigma} u\|_{L^{p}(\R_+,w_{\alpha};X_{0})} + \|(-\partial_t)^{\sigma} A^{1-\sigma} u\|_{L^{p}(\R_+,w_{\alpha};X_{0})}.\]
Since we have already estimated the $L^p$-norm of $\|A^{1-\sigma} u\|_{X_0}\lesssim \|A u\|_{X_0}$, it remains to prove
\begin{align}\label{eq:estpartialu}
\E \|(-\partial_t)^{\sigma} A^{1-\sigma} u\|_{L^p(\R_+,w_{\alpha};X_0)}^p\lesssim \E\|g\|^{p}_{L^p(\R_+,w_{\alpha};\gamma(H,X_{1/2}))}.
\end{align}
Note that $\partial_t$ generates the left translation semigroup on $L^p(\R_+;X_0)$. Since $A^{1-\sigma} u$ is in $W^{1,p}(\R_+,w_{\alpha};X_0)$, it follows from the Balakrishnan formula for the fractional power (see \cite[Theorem 3.2.2]{MaSa}) that
\begin{align}\label{eq:Balak}
(-\partial_t)^\sigma A^{1-\sigma} u(t) = C_{\sigma} \int_0^\infty A^{1 -\sigma} \frac{u(t)-u(t+h)}{h^{1+\sigma}} \, dh,
\end{align}
with $C_{\sigma} \neq 0$.

\subsubsection{Estimating the fractional derivative}
By \eqref{eq:Balak} we can write
\begin{align*}
	\E&\|A^{1-\sigma} (-\partial)^\sigma u\|^p_{L^p(\R_+,w_{\alpha};X_0)} \eqsim \E\int_0^\infty \Big \| \int_0^\infty A^{1 -\sigma} \frac{u(t)-u(t+h)}{h^{1+\sigma}} \, dh \Big\|^p_{X_0} w_{\alpha}(t) \, dt
	\\
	&= \E\sum_{n\ge0} \int_{t_n}^{t_{n+1}} \Big \| \sum_{j\ge0} \int_{t_j}^{t_{j+1}} A^{1 -\sigma} \frac{u(t)-u(t+h)}{h^{1+\sigma}} \, dh \Big\|^p_{X_0} t^{\alpha} \, dt
	\\
	& \leq \E\sum_{n\ge0} \tau \int_{0}^{1} \Big \| \sum_{j\ge0} \tau \int_{0}^{1} A^{1 -\sigma} \frac{u(n\tau +s\tau)-u((n+j)\tau+(s+r)\tau)}{(j\tau +r\tau)^{1+\sigma}} \, dr \Big\|^p_{X_0} t_{n+1}^{\alpha}\, ds.
\end{align*}
 For $s+r\in [0,1)$ we can write
\begin{align*}
	u(n\tau +&s\tau)-u((n+j)\tau+(s+r)\tau)
	\\ &= (1-s)Y_n+sY_{n+1}-[(1-(s+r))Y_{n+j}+(s+r)Y_{n+j+1}]
	\\ &=(1-s-r) (Y_n-Y_{n+j}) + (s+r)(Y_{n+1}-Y_{n+j+1}) + r(Y_{n}-Y_{n+1}),
\end{align*}
For $s+r\in [1, 2)$ we can write
\begin{align*}
	u(&n\tau +s\tau)-u((n+j)\tau+(s+r)\tau)
	\\ &= (1-s)Y_n+sY_{n+1}-[(1-(s+r-1))Y_{n+j+1}+(s+r-1)Y_{n+j+2}]
  \\ & = (2-s-r)(Y_n - Y_{n+j+1}) + (s+r-1)(Y_{n+1}-Y_{n+j+2}) +(1-r)(Y_{n+1}-Y_{n}).
\end{align*}
Therefore, from the triangle inequality, it follows that it suffices to estimate each of the following:
\begin{align*}
	T_{1} &\coloneq \sum_{n\ge0} \tau \int_0^1 \E \Big \| \sum_{j \ge 1} \int_0^{1-s} \tau A^{1 - \sigma} (1-s-r) \frac{Y_n-Y_{n+j}}{(j\tau+r\tau)^{1+\sigma}} \, dr \Big\|^p_{X_0} t_{n+1}^{\alpha} \, ds
\\
	T_{2} & \coloneq \sum_{n\ge0} \tau \int_0^1 \E \Big \| \sum_{j \ge 1} \int_0^{1-s} \tau A^{1 - \sigma} (s+r)\frac{Y_{n+1}-Y_{n+j+1}}{(j\tau+r\tau)^{1+\sigma}} \, dr \Big\|^p_{X_0} t_{n+1}^{\alpha}\, ds
\\
T_{3} & \coloneq \sum_{n\ge0} \tau \int_0^1 \E \Big \| \sum_{j \ge 0} \int_0^{1-s} \tau A^{1 - \sigma} r\frac{Y_n-Y_{n+1}}{(j\tau+r\tau)^{1+\sigma}} \, dr \Big\|^p_{X_0} t_{n+1}^{\alpha}ds
\\
T_{4} &\coloneq \sum_{n\ge0} \tau \int_0^1 \E \Big \| \sum_{j \ge 1} \int_{1-s}^1 \tau A^{1 - \sigma} (2-s-r)\frac{Y_n-Y_{n+j+1}}{(j\tau+r\tau)^{1+\sigma}} \, dr \Big\|^p_{X_0} t_{n+1}^{\alpha}\, ds
\\
T_{5} & \coloneq \sum_{n\ge0} \tau \int_0^1 \E \Big \| \sum_{j \ge 0} \int_{1-s}^1 \tau A^{1 - \sigma} (s+r-1)\frac{Y_{n+1}-Y_{n+j+2}}{(j\tau+r\tau)^{1+\sigma}} \, dr \Big\|^p_{X_0} t_{n+1}^{\alpha}\, ds
\\
 	T_{6} & \coloneq \sum_{n\ge0} \tau \int_0^1 \E \Big \| \sum_{j \ge 1} \int_{1-s}^{1} \tau A^{1 - \sigma} (1-r)\frac{Y_{n+1}-Y_{n}}{(j\tau+r\tau)^{1+\sigma}} \, dr \Big\|^p_{X_0} t_{n+1}^{\alpha}\, ds
 \\ T_7& \coloneq \sum_{n\ge0} \tau \int_0^1 \E \Big \| \int_{1-s}^{1} \tau A^{1 - \sigma} (1-s)\frac{Y_{n+1}-Y_{n}}{(r\tau)^{1+\sigma}} \, dr\Big\|^p_{X_0} t_{n+1}^{\alpha} \, ds.
\end{align*}

The terms $T_1, T_2, T_3$ come from the case $s+r\in [0,1)$, and the other terms
from $s+r\in [1,2)$. Note that the term $T_7$ is just the $j=0$ term of $T_4$ and $T_6$ combined to avoid the creation of a singularity.
The proofs of the estimates for $T_1, T_2, T_4, T_5$ are very similar. To give the bounds, we only present the details for $T_1$. It is enough to prove that for all $s\in (0,1)$:
\begin{align}\label{eq:tobeprovedT1}
	T_{1,s} &\coloneq \sum_{n\ge0} \tau \E \Big \| \sum_{j \ge 1} \tau^{-\sigma} \psi(j,s) A^{1 - \sigma} (Y_{n+j}-Y_n) \Big\|^p_{X_0}\leq C\|g\|_{L^p(\R_+;\gamma(H,X_{1/2}))},
\end{align}
where $C$ is independent of $s$ and
\begin{align*}
\psi(j,s) &= \int_0^{1-s} \frac{(1-s-r)}{(j+r)^{1+\sigma}} \,dr.
\end{align*}
We need two properties of $\psi$, which are both straightforward to check:
\begin{align}\label{eq:propphi}
\sup_{j\geq 1, s\in (0,1)} j^{1+\sigma} |\psi(j,s)|<\infty \ \ \text{and} \ \ \sup_{j\geq 1, s\in (0,1)} j^{2+\sigma} \Big|\psi(j+1,s) - \psi(j,s)\Big|<\infty.
\end{align}
The details of \eqref{eq:tobeprovedT1} will be given in Subsection \ref{sss:Texp}. The proofs are based on the $H^\infty$-calculus of $A$ and the $\mathcal{R}$-boundedness results provided in Section \ref{sec:Rbdd}.

\subsubsection{Estimating terms $T_3, T_6$ and $T_7$}
To estimate $T_7$ note that
\begin{align*}
\sum_{n\ge0} \tau \int_0^1 \E &\Big \| \int_{1-s}^{1} \tau A^{1 - \sigma} (1-s)\frac{Y_{n+1}-Y_{n}}{(r\tau)^{1+\sigma}} \, dr\Big\|^p_{X_0} t_{n+1}^{\alpha} \, ds
\lesssim \sum_{n\ge0} \tau \E \Big \| \tau^{-\sigma} A^{1 - \sigma} (Y_{n+1}-Y_{n}) \, \Big\|^p_{X_0} t_{n+1}^{\alpha}.
\end{align*}
Before we continue, we first note that for $T_3$ and $T_6$ one has
\begin{align*}
C_3&\coloneq \sum_{j\geq 0} \int_0^{1-s} \frac{r}{(j+r)^{1+\sigma}} \,dr\leq \int_0^{1} \frac{1}{r^{\sigma}} \,dr+\sum_{j\geq 1} \frac{1}{j^{1+\sigma}} \,dr<\infty,
\\ C_6&\coloneq \sum_{j\geq 1} \int_{1-s}^1 \frac{1-r}{(j+r)^{1+\sigma}} \,dr\leq \sum_{j\geq 1} \frac{1}{j^{1+\sigma}} \,dr<\infty
\end{align*}
and therefore,
\begin{align}\label{eq:toboundT3}
T_{3} \leq C_3^p \sum_{n\ge0} \tau \E \Big \| \tau^{-\sigma} A^{1 - \sigma} (Y_{n+1}-Y_{n})\Big\|^p_{X_0} t_{n+1}^{\alpha} .
\end{align}
The same estimate holds for $T_6$. Therefore, in order to bound $T_3, T_6, T_7$, it suffices to bound the right-hand side of \eqref{eq:toboundT3}. From the formula for $Y$ given in \eqref{Eq:Definition of approximation scheme}, we see that
$$Y_{n+1}-Y_n= R_\tau \Delta_n I_g+\sum_{m=0}^{n-1} (R_\tau-I)R_\tau^{n-m}\Delta_m I_g.$$
Thus it suffices to bound each of the terms
\begin{align*}
	T_{8} & \coloneq \sum_{n\ge0} \tau \E \Big \| \tau^{-\sigma} A^{1-\sigma} R_{\tau} \Delta_n I_g \Big\|^p_{X_0} t_{n+1}^{\alpha},
	\\
	T_{9} & \coloneq  \sum_{n\ge 1} \tau \E \Big \| \tau^{-\sigma}  \sum_{m=0}^{n-1} A^{1-\sigma} (R_{\tau}-I)(R_{\tau}^{n-m}) \Delta_m I_g \Big\|^p_{X_0} t_{n+1}^{\alpha}.
\end{align*}
By Proposition \ref{prop:BDGtype2} and \eqref{Ineq:Fractional powers of rational} we can estimate
\begin{align*}
	T_{8}&  \lesssim_{p,X_0} \sum_{n\ge0} \tau \E \Big\| \tau^{-\sigma} A^{1-\sigma} R_{\tau} g \Big\|^p_{L^2(t_n, t_{n+1};\gamma(H,X_0))} t_{n+1}^{\alpha}
	\\
	&\lesssim\sum_{n\ge0} \tau \E \tau^{-p/2} \|A^{1/2}g \|^p_{L^2((t_n, t_{n+1});\gamma(H,X_{0}))} t_{n+1}^{\alpha}
\\ & \lesssim 	 \sum_{n\ge0} \E \|A^{1/2}g \|^p_{L^p((t_n, t_{n+1}), w_{\alpha};\gamma(H,X_{0}))} \lesssim \E \|g\|^p_{L^p(\R_+,w_{\alpha};\gamma(H,X_{1/2}))},
\end{align*}
where in the penultimate step we used H\"older's inequality and in the last step we used Lemma \ref{lem:BIP}.
By Proposition \ref{prop:BDGtype2} and Lemma \ref{Lemma:Fractional powers of difference EE and rational} (writing $R_{\tau}-I = (R_{\tau}-e^{-\tau A}) + (e^{-\tau A}-I)$ and using \eqref{Eq: Decay rate of semigroup difference}) we get that
\begin{align*}
	T_{9}&  = \sum_{n\ge 1} \tau  \E \Big \| \tau^{-\sigma}    \sum_{m=0}^{n-1} A^{-(\sigma+\frac12)} (R_{\tau}-I) A^{1+\frac12}R_{\tau}^{n-m} \Delta_m I_g \Big\|^p_{X_0}  t_{n+1}^{\alpha}
	\\
	&\lesssim_{p,X} \sum_{n\ge 1} \tau \E \Big ( \sum_{m=0}^{n-1} \Big \| \tau^{-\sigma}  A^{-(\sigma+\frac 12)} (R_{\tau}-I) AR_{\tau}^{n-m} A^{\frac12} g\Big\|^2_{L^2(t_m, t_{m+1};\gamma(H,X_{0}))} \Big)^{p/2} t_{n+1}^{\alpha}
        \\
         & \le \sum_{n\ge 1} \tau \E \Big ( \sum_{m=0}^{n-1}  \| \tau^{-\sigma}  A^{-(\sigma+\frac12)}(R_{\tau}-I) \|^2_{\L(X_0)} \|AR_{\tau}^{n-m}\|^2_{\L(X_0)} \|A^{\frac12} g \|^2_{L^2(t_m, t_{m+1};\gamma(H,X_{0}))} \Big)^{p/2} t_{n+1}^{\alpha}
	\\
	&\lesssim \sum_{n\ge0} \tau \E \Big ( \sum_{m=0}^{n-1} \tau \frac{1}{(\tau(n-m))^{2}} \|g\|^2_{L^2(t_m, t_{m+1};\gamma(H,X_{1/2})} \Big)^{p/2} t_{n+1}^{\alpha}
\\ 	& \stackrel{\text{(i)}}{\lesssim}\sum_{n\ge0} \E \Big ( \sum_{m=0}^{n-1} \frac{1}{(n-m)^2} \|g\|^2_{L^p(t_m, t_{m+1}, w_{\alpha};\gamma(H,X_{1/2})} t_{m+1}^{-2\alpha/p}\Big)^{p/2} t_{n+1}^{\alpha}
\\ &  = \E \sum_{n\ge0}  \Big ( \sum_{m=0}^{n-1} G_m^2 K(n,m) \Big)^{p/2}  \stackrel{\text{(ii)}}{\lesssim}\E \sum_{m\ge0} G_m^p  = \E  \|g\|^p_{L^p(\R_+,w_{\alpha};\gamma(H,X_{1/2}))},
\end{align*}
where in (i) we used H\"older's inequality and in (ii)  we used Schur's lemma (see \cite[Appendix A.2]{GrafModern}) in $\ell^{p/2}$ for the non-negative kernel $K(n,m) \coloneq \frac1{(n-m)^2} \big(\frac{n+1}{m+1}\big)^{2\alpha/p}  \1_{0\le m <n}$ with sequences $u_n=v_n=(n+1)^{-\frac1{(\frac p2)(\frac p2)'}}$, noting that $0\le \alpha<\frac p2-1$,  and where we set $G_m \coloneq \|g\|_{L^p(t_m, t_{m+1}, w_{\alpha};\gamma(H,X_{1/2}))}$.

\subsubsection{Estimating the main term $T_{1,s}$}\label{sss:Texp}

In this section, we prove the estimate \eqref{eq:tobeprovedT1}. For this, we use the operator-valued $H^\infty$-calculus. This method was first used to establish stochastic maximal $L^p$-regularity in continuous-time in \cite{vanneerven2014stochasticintegrationbanachspaces}.

We write $R_\tau = r(\tau A)$ where $r(z)$ is either the exponential function or a rational function as in Assumption \ref{assum:mainDSMR}. Fix $\nu\in (\omega(A),\theta)$. In order to rewrite $T_{1,s}$, defined in \eqref{eq:tobeprovedT1}, note that
\begin{align*}
Y_{n+j}-Y_n &= \sum_{m=n}^{n+j-1} r(\tau A)^{n+j-m}\Delta_m I_g +\sum_{m=0}^{n-1} (R_{\tau}^j-I)r(\tau A)^{n-m}\Delta_m I_g
\end{align*}
and thus moving part of the power of $A$ to the process $g$, and using a functional calculus representation (see \cite{hytonen2017analysis,Analysis3}), we can write
\begin{align*}
 A^{1-\sigma}(Y_{n+j}-Y_n) &= \sum_{m=n}^{n+j-1} A^{1/2-\sigma} R_{\tau}^{n+j-m} \Delta_m I_{A^{1/2}}g +\sum_{m=0}^{n-1} A^{1/2-\sigma}(R_{\tau}^j-I)R_{\tau}^{n-m}\Delta_m I_{A^{1/2}g}
\\ & = \frac{1}{2\pi i} \int_{\partial \Sigma_{\nu}} \Big[\sum_{m=n}^{n+j-1}f_{j,m,n}^{(1)} (\lambda) R(\lambda,A) \Delta_m I_{A^{1/2}g} + \sum_{m=0}^{n-1} f_{j,m,n}^{(2)}(\lambda)R(\lambda,A) \Delta_m I_{A^{1/2}g}\Big] d\lambda,
\end{align*}
where $f_{j,m,n}^{(1)}, f_{j,m,n}^{(2)}\colon \Sigma_{\nu}\to \C$ are given by
\begin{align*}
f_{j,m,n}^{(1)}(\lambda) &\coloneq \lambda^{1/2-\sigma} r(\tau\lambda)^{n+j-m},
\\
f_{j,m,n}^{(2)}(\lambda)&\coloneq \lambda^{1/2-\sigma}(r(\tau \lambda)^j-I)r(\tau\lambda)^{n-m}.
\end{align*}
Therefore, to estimate $T_{1,s}$ it suffices to bound $T_{1,s}^{(1)}$ and $T_{1,s}^{(2)}$, which are given by
\begin{align*}
T_{1,s}^{(1)} & = \E \sum_{n\geq 0} \tau \Big\|\int_{\partial \Sigma_\nu} \Big[\sum_{j \ge 1} \tau^{-\sigma} \psi(j,s) \sum_{m=n}^{n+j-1} f_{j,m,n}^{(1)}(\lambda) R(\lambda,A) \Delta_m I_{A^{1/2}g}\Big] d\lambda\Big\|_{X_0}^p,
\\
T_{1,s}^{(2)} & = \E \sum_{n\geq 0} \tau \Big\|\int_{\partial \Sigma_\nu} \Big[\sum_{j \ge 1} \tau^{-\sigma} \psi(j,s) \sum_{m=0}^{n-1} f_{j,m,n}^{(2)}(\lambda) R(\lambda,A) \Delta_m I_{A^{1/2}g}\Big] d\lambda\Big\|_{X_0}^p.
\end{align*}
To rewrite $T_{1,s}^{(1)}$ in the form of a discrete convolution operator, note that for $h\in L^{p}_{\mathbb F}(\Omega;L^p(\R_+;\gamma(H,X_{0})))$,
\begin{align*}
\sum_{j \ge 1} \tau^{-\sigma} \psi(j,s) \sum_{m=n}^{n+j-1} f_{j,m,n}^{(1)}(\lambda) \Delta_m I_{h} &= \sum_{m\geq n} \sum_{j\geq m-n+1} \psi(j,s) \tau^{-\sigma} \lambda^{1/2-\sigma} r(\tau\lambda)^{n+j-m} \Delta_m I_{h}
\\ & = \sum_{m\geq n} k_{m-n}^{(1)}(\lambda) \Delta_m I_{h} = \wt{I}^{\tau}(k^{(1)}(\lambda)) h,
\end{align*}
where
\[k_m^{(1)}(\lambda) = \sum_{j\geq m+1} \psi(j,s) \tau^{-\sigma} \lambda^{1/2-\sigma} r(\tau\lambda)^{j-m},\]
and $\wt{I}^{\tau}(k(\lambda))$ is as in Remark \ref{Temp: Rbdd for i>n}. We did not make the $\tau$-dependence explicit in $k$, but of course it also depends on $\tau$. It follows that $T_{1,s}^{(1)}$ can be written as
\begin{align*}
T_{1,s}^{(1)} & = \Big\|\int_{\partial \Sigma_\nu} \wt{I}^{\tau}(k^{(1)}(\lambda)) R(\lambda,A) A^{1/2}g d\lambda\Big\|^p_{L^p(\Omega;\ell^p_{\tau,w_{\alpha}}(X_0))}.
\end{align*}
Let us introduce the short-hand notation
\[\calL^\tau \coloneq  \calL(L^{p}_{\mathbb F}(\Omega;L^p(\R_+,w_{\alpha};\gamma(H,X_{0}))), L^p(\Omega;\ell^p_{\tau,w_{\alpha}}(X_0))).\]
One can check that the operator-valued function $\lambda \mapsto \wt{I}^{\tau}(k^{(1)}(\lambda))$ is in $H^1(\Sigma_\theta; \calL^\tau)$, and commutes with the resolvent of $A$ (seen as an operator on $L^p(\Omega;\ell^p_{\tau,w_{\alpha}}(X_0))$). We claim that the above operator-valued function as a family in $\calL^\tau$ has $\mathcal{R}$-bounded range (with uniform estimates in $s$ and $\tau$). As soon as we have checked this, it follows from the boundedness of the $H^\infty$-calculus of $A$ and \cite[Theorem 16.3.4]{Analysis3} that
\begin{align*}
T_{1,s}^{(1)}\leq C \|A^{1/2} g\|_{L^{p}(\Omega;L^p(\R_+;\gamma(H,X_{0})))} \lesssim \|g\|_{L^{p}(\Omega;L^p(\R_+,w_{\alpha};\gamma(H,X_{1/2})))},
\end{align*}
where $C$ does not depend on $s$ and $\tau$, and where in the last step we used Lemma \ref{lem:BIP}.

To prove the claim, note that by Remark \ref{Temp: Rbdd for i>n} it suffices to show that there is a constant $C$ independent of $\lambda$ and $s$ such that $k^{(1)}(\lambda)/C$ is in $\mathcal K_{\tau}$. This follows from Lemmas \ref{lem:expvariant}, \ref{lem:rationalvariant} and \eqref{eq:propphi}.

A similar argument can be used to estimate $T_{1,s}^{(2)}$. Indeed,
\begin{align*}
\sum_{j \ge 1} \tau^{-\sigma} \psi(j,s) \sum_{m=0}^{n-1} f_{j,m,n}^{(2)}(\lambda) \Delta_m I_{h} &=
\sum_{m=0}^{n-1} \sum_{j \ge 1} \tau^{-\sigma} \psi(j,s) f_{j,m,n}^{(2)}(\lambda) \Delta_m I_{h}
\\ & = \sum_{m=0}^{n-1} k_{n-m}^{(2)}(\lambda) \Delta_m I_{h} = I^{\tau}(k^{(2)}(\lambda)) h,
\end{align*}
where
\[k_m^{(2)}(\lambda) = \sum_{j \ge 1} \tau^{-\sigma} \psi(j,s)
\lambda^{1/2-\sigma}(r(\tau \lambda)^j-I)r(\tau\lambda)^{m},\]
and $I_\tau(k(\lambda))$ is as in Section \ref{ss:mainRbdd}.
It follows that $T_{1,s}^{(2)}$ can be written as
\begin{align*}
T_{1,s}^{(2)} & = \Big\|\int_{\partial \Sigma_\nu} I^{\tau}(k^{(2)}(\lambda)) R(\lambda,A) A^{1/2}g d\lambda\Big\|^p_{L^p(\Omega;\ell^p_{\tau,w_{\alpha}}(X_0))}.
\end{align*}
Now the proof can be completed as we did for $T_{1,s}^{(1)}$, where this time the $\mathcal{R}$-boundedness of the range of $\lambda \mapsto I^{\tau}(k^{(2)}(\lambda))$ follows from Theorem \ref{Thm: J is uniformly R-bdd}, Lemmas \ref{lem:exp}, \ref{lem:rational} and \eqref{eq:propphi}.

\begin{remark}\label{rem:MaxDet}
In the deterministic setting, maximal estimates in trace spaces are well-known, and can be found in \cite[Theorem 2.3.2]{AsSo} and \cite[Lemma 14]{kemmochi2018discrete}. The deterministic setting is much simpler because one can formulate an equivalent description of the trace norm in a discrete setting. Moreover, expectations do not play any role here.

Our argument for proving maximal estimates can also be used to give an alternative proof for the deterministic analogue. Indeed, typically, discrete maximal regularity in the deterministic setting involves the quantity
\begin{align*}
\|D_{\tau}Y\|_{\ell^p_{\tau}(X_0)} + \|Y\|_{\ell^p_{\tau}(X_1)},
\end{align*}
where $(D_{\tau}Y)_n = \frac{Y_{n+1}-Y_n}{\tau}$. Extending $Y$ to a continuous-time function $u$ as in \eqref{eq:linearinterp} immediately gives $u\in W^{1,p}(\R_+;X_0)\cap L^p(\R_+;X_1)$, and thus the trace regularity $u\in C_b([0,\infty);X_{1-\frac1p,p})$ follows from the classical Lions--Peetre trace method for real interpolation (see \cite[Appendix L]{Analysis3}). 
\end{remark}

\subsection{Proof of the maximal estimate of Proposition \ref{prop:maxHilbert}} 
We will actually prove the following more general result, which reduces to Proposition \ref{prop:maxHilbert} if $s=2$.
\begin{proposition}\label{prop:maxHilbert-with-s}
Let $X_0$ be a Hilbert space. Suppose that Assumption \ref{assum:mainDSMR} holds with $\theta=\pi/2$. Suppose that $0\in \rho(A)$ or that $A$ has a bounded $H^\infty$-calculus.
Then for every $s\in (0,\infty)$ there is a constant $C$ such that for every $g\in L^{s}_{\mathbb F}(\Omega;L^2(\R_+;\gamma(H,X_{1/2})))$ and every stepsize $\tau>0$,
\begin{equation*}
	\E \sup_{n\ge1} \| Y_n \|_{X_{1/2}}^s \le C^s \E\|g\|^{s}_{L^2(\R_+;\gamma(H,X_{1/2}))},
\end{equation*}	
where $Y=(Y_n)_{n\ge0}$ is given by \eqref{Eq:Definition of approximation scheme-intro}.
\end{proposition}

\begin{proof}
Note that $X_{1/2} = (X_0, X_1)_{1/2,2}$ and $\gamma(H,X_{1/2}) = \L_2(H,X_{1/2})$ with equivalent norms (see \cite[Corollary C.4.2]{Analysis1} and \cite[Proposition 9.1.9]{hytonen2017analysis}). If $0\in \rho(A)$, then $A$ has a bounded $H^\infty$-calculus on $(X_0, X_1)_{1/2,2}$ by Dore's theorem (see \cite[Corollary 16.3.23]{Analysis3}). If $A$ has a bounded $H^\infty$-calculus on $X_0$, then this holds even without the condition $0\in \rho(A)$ by interpolation.
From \cite[Theorem 11.13]{kunstmann2004maximal} it follows that there is a Hilbert space norm $\|\cdot\|_Z$ which is equivalent to $X_{1/2}$ under which $(e^{-tA})_{t\geq 0}$ is a contraction semigroup. Let $K_1,K_2>0$ be such that $K_1\|x\|_{Z} \leq \|x\|_{X_{1/2}} \leq K_2\|x\|_{Z}$. Note that since $|r(z)|\le 1$ on $\Sigma_{\pi/2}$, there exists $\sigma>\pi/2$ small enough such that $r \in H^\infty (\Sigma_\sigma) $. Hence, from \cite[Theorem 10.2.24]{hytonen2017analysis} it follows that $R_{\tau}$ is contractive on $Z$.

Next, we extend the discrete dilation argument in \cite[Proposition 5.1]{klioba2024pathwise} to our setting. By the Sz.-Nagy dilation theorem \cite[Theorem I.4.2]{dilation} we can find a Hilbert space $\wt{Z}$, a contractive injection $Q\colon Z\to \wt{Z}$, a contractive projection $P\colon \wt{Z}\to Z$ , and a unitary operator $\wt{R}_\tau$ on $\wt{Z}$ such that \[R_\tau^j =P \wt{R}_\tau^j Q, \ \ \ j\geq 0.\]
It follows that
\begin{align*}
\|Y_n\|_{X_{1/2}} & \leq K_2 \|Y_n\|_{Z} = K_2\Big\|\sum_{j=0}^{n-1} R_\tau^{n-j} \, \Delta_j I_g \Big\|_{Z} \leq K_2\Big\|\sum_{j=0}^{n-1} \wt{R}_\tau^{-j} Q \Delta_j I_g \Big\|_{\wt{Z}},
\end{align*}
where we used that $\wt{R}_\tau$ is unitary.
Therefore, by Proposition \ref{prop:BDGtype2} we obtain
\begin{align*}
\E \sup_{n\geq 1}\|Y_n\|_{X_{1/2}}^s & \leq K_2^s \E \sup_{n\geq 1} \Big\|\sum_{j=0}^{n-1} \wt{R}_\tau^{-j} Q \Delta_j I_g \Big\|_{\wt{Z}}^s
\\ & \leq K_2^s C_{s,X_0}^s \E \Big(\sum_{j\geq 0} \int_{t_j}^{t_{j+1}} \| \wt{R}_\tau^{-j} Q g(t)\|_{\calL_2(H,\wt{Z})}^2 dt\Big)^{s/2}
\\ & \leq K_2^s C_{s,X_0}^s \E \Big(\sum_{j\geq 0} \int_{t_j}^{t_{j+1}} \| g(t)\|_{\calL_2(H,Z)}^2 dt\Big)^{s/2}
\\ & \leq K_1^{-s} K_2^s C_{s,X_0}^s \E \| g\|_{L^2(\R_+;\calL_2(H,X_{1/2}))}^s.
\end{align*}
\end{proof}

\begin{remark}
In the above, it is enough to find an equivalent norm in which $R_{\tau}$ is a contraction.
\begin{enumerate}[(1)]
\item In case $s=2$ and $R_{\tau}$ is a contraction on a Hilbert space, then using Doob's maximal inequality for second moments and the above method, one can check that
\[\E \sup_{n\ge1} \| Y_n \|_{X_{1/2}}^2 \le 4 \E\|g\|^{2}_{L^2(\R_+;\gamma(H,X_{1/2}))}.\]

\item In case $R_{\tau}$ is a contraction on a $2$-smooth Banach space $X_{1/2}$, a similar result as in Proposition \ref{prop:maxHilbert-with-s} was proved in \cite[Proposition 5.4]{van2021maximal}.
\end{enumerate}
\end{remark}

\subsection*{Data availability}
Data sharing is not applicable to this article as no new data were created or analyzed in
this study.

\subsection*{Conflict of interest statement}
On behalf of all authors, the corresponding author states that there is no conflict of interest.

\end{document}